\documentclass[11pt,a4paper]{amsart}%
\usepackage{amsmath}
\usepackage{amsfonts}
\usepackage{graphicx}
\usepackage{amssymb}
\usepackage[english]{babel}
\usepackage{amsthm}
\usepackage{color}%
\providecommand{\U}[1]{\protect\rule{.1in}{.1in}}

\def \N {\mathbb{N}}

\def \G {\mathbb{G}}

\def \e {\varepsilon}

\def \de {\partial}
\def \LL {\mathcal{L}}

\def \bd {\boldsymbol}
\theoremstyle{definition}
\newtheorem{definition}{Definition}[section]

\newtheorem{remark}[definition]{Remark}

\theoremstyle{plain}

\newtheorem{theorem}[definition]{Theorem}
\newtheorem{proposition}[definition]{Proposition}
\newtheorem{lemma}[definition]{Lemma}
\newtheorem{corollary}[definition]{Corollary}
\numberwithin{equation}{subsection}
\begin{document}
\title[Schauder estimates for KFP operators]{Schauder estimates for Kolmogorov-Fokker-Planck operators with coefficients
measurable in time and H\"{o}lder continuous in space}
\author[S.\,Biagi]{Stefano Biagi}
\author[M.\thinspace Bramanti]{Marco Bramanti}
\address[S.\thinspace Biagi and M.\thinspace Bramanti]{Dipartimento di Matematica \\
\indent Politecnico di Milano, Via Bonardi 9, 20133 Milano, Italy}
\email{stefano.biagi@polimi.it; marco.bramanti@polimi.it}

\begin{abstract}
We consider degenerate Kolmogorov-Fokker-Planck operators%
\[
\mathcal{L}u=\sum_{i,j=1}^{q}a_{ij}(x,t)\partial_{x_{i}x_{j}}^{2}%
u+\sum_{k,j=1}^{N}b_{jk}x_{k}\partial_{x_{j}}u-\partial_{t}u,\text{
\ \ }(x,t)\in\mathbb{R}^{N+1},N\geq q\geq1
\]
such that the corresponding model operator having constant $a_{ij}$ is
hypoelliptic, translation invariant w.r.t. a Lie group operation in
$\mathbb{R}^{N+1}$ and $2$-homogeneous w.r.t. a family of nonisotropic
dilations. The coefficients $a_{ij}$ are bounded and H\"{o}lder continuous in
space (w.r.t. some distance induced by $\mathcal{L}$ in $\mathbb{R}^{N}$) and
only bounded measurable in time; the matrix $\left\{  a_{ij}\right\}
_{i,j=1}^{q}$ is symmetric and uniformly positive on $\mathbb{R}^{q}$. We
prove \textquotedblleft partial Schauder a priori estimates\textquotedblright%
\ of the kind%
\[
\sum_{i,j=1}^{q}\Vert\de_{x_{i}x_{j}}^{2}u\Vert_{C_{x}^{\alpha}(S_{T})}+\Vert
Yu\Vert_{C_{x}^{\alpha}(S_{T})}\leq c\left\{  \Vert\mathcal{L}u\Vert
_{C_{x}^{\alpha}(S_{T})}+\Vert u\Vert_{C^{0}(S_{T})}\right\}
\]
for suitable functions $u$, where
\[
\Vert f\Vert_{C_{x}^{\alpha}(S_{T})}=\sup_{t\leq T}\sup_{x_{1},x_{2}%
\in\mathbb{R}^{N},x_{1}\neq x_{2}}\frac{\left\vert f\left(  x_{1},t\right)
-f\left(  x_{2},t\right)  \right\vert }{\left\Vert x_{1}-x_{2}\right\Vert
^{\alpha}}+\left\Vert f\right\Vert _{L^{\infty}(S_{T})}.
\]
We also prove that the derivatives $\de_{x_{i}x_{j}}^{2}u$ are locally
H\"{o}lder continuous in space and time while $\de_{x_{i}}u$ and $u$ are
globally H\"{o}lder continuous in space and time.

\end{abstract}
\maketitle
\tableofcontents




\section{Introduction and main results\label{sec intro}}

\subsection{The problem and its context}

Let $N\geq q\geq1\ $be fixed. We consider a Kolmogorov-Fokker-Planck (KFP, in
short) operator of the form%
\begin{equation}
\mathcal{L}u=\sum_{i,j=1}^{q}a_{ij}(x,t)\partial_{x_{i}x_{j}}^{2}%
u+\sum_{k,j=1}^{N}b_{jk}x_{k}\partial_{x_{j}}u-\partial_{t}u,\qquad
(x,t)\in\mathbb{R}^{N+1}. \label{L}%
\end{equation}
The first-order part of the operator, also called \emph{the drift term}, will
be briefly denoted by%
\begin{equation}
Yu=\sum_{k,j=1}^{N}b_{jk}x_{k}\partial_{x_{j}}u-\partial_{t}u.
\label{eq:driftY}%
\end{equation}
Throughout the paper, points of $\mathbb{R}^{N+1}$ will be sometimes denoted
by the compact notation%
\[
\xi=(x,t),\quad\eta=(y,s).
\]
We will make the following assumptions:

\begin{itemize}
\item[\textbf{(H1)}] $A_{0}(x,t)=(a_{ij}(x,t))_{i,j=1}^{q}$ is a symmetric
uniformly positive matrix on $\mathbb{R}^{q}$ of bounded coefficients defined
in $\mathbb{R}^{N+1}$, so that
\begin{equation}
\nu|\xi|^{2}\leq\sum_{i,j=1}^{q}a_{ij}(x,t)\xi_{i}\xi_{j}\leq\nu^{-1}|\xi|^{2}
\label{nu}%
\end{equation}
for some constant $\nu>0$, every $\xi\in\mathbb{R}^{q}$, every $x\in
\mathbb{R}^{N}$ and a.e.\thinspace$t\in\mathbb{R}$. The co\-ef\-fi\-cients
will be assumed measurable w.r.t.\thinspace$t$ and H\"{o}lder continuous
w.r.t.\thinspace$x$, in a sense that will be made precise later.
\vspace*{0.1cm}(See Assumption (H3)).

\item[\textbf{(H2)}] The matrix $B= (b_{ij})_{i,j=1}^{N}$ satisfies the
following condition: for $m_{0}=q$ and suitable positive integers $m_{1}%
,\dots,m_{k}$ such that
\begin{equation}
\label{m-cond}m_{0}\geq m_{1}\geq\ldots\geq m_{k}\geq1\quad\mathrm{and}\quad
m_{0}+m_{1}+\ldots+m_{k}=N,
\end{equation}
we have
\begin{equation}
\label{B}B=
\begin{pmatrix}
\mathbb{O} & \mathbb{O} & \ldots & \mathbb{O} & \mathbb{O}\\
B_{1} & \mathbb{O} & \ldots & \ldots & \ldots\\
\mathbb{O} & B_{2} & \ldots & \mathbb{O} & \mathbb{O}\\
\vdots & \vdots & \ddots & \vdots & \vdots\\
\mathbb{O} & \mathbb{O} & \ldots & B_{k} & \mathbb{O}%
\end{pmatrix}
\end{equation}
where every block $B_{j}$ is an $m_{j}\times m_{j-1}$ matrix of rank $m_{j}$
(for $j=1,2,\ldots,k$).
\end{itemize}

\medskip

\noindent We explicitly note that, when $q<N$, the operator $\mathcal{L}$ is
\emph{ultraparabolic}; in this context, the model operator is the so-called
\emph{Kolmogorov operator} $\mathcal{K}$, which arose in the seminal paper by
Kolmogorov \cite{Kolmo} on Brownian motion and the theory of gases.
As\-su\-ming that the ambient space $\mathbb{R}^{N}$ has even dimension, say
$N=2n$, this model operator $\mathcal{K}$ has the following explicit
expression
\[
\mathcal{K}=\Delta_{u}+\langle u,\nabla_{v}\rangle-\partial_{t},\quad
\text{with $u,v\in\mathbb{R}^{n}$ and $t\in\mathbb{R}$}.
\]
Clearly, $\mathcal{K}$ can be obtained from \eqref{L} by choosing
\[
q=n<N,\quad A_{0}=\mathrm{Id}_{n},\quad m_{0}=m_{1}=n,\quad B=%
\begin{pmatrix}
\mathbb{O}_{n} & \mathbb{O}_{n}\\
\mathrm{Id}_{n} & \mathbb{O}_{n}%
\end{pmatrix}
.
\]
Even if it fails to be parabolic, one can easily check that $\mathcal{K}$
satisfies \emph{H\"{o}rmander's rank condition}, and thus $\mathcal{K}$ is
$C^{\infty}$-hypoelliptic by H\"{o}rmander's hypoellipticity theorem
\cite{Horm}; however, this fact was implicitly proved by Kolmogorov himself
several years prior to \cite{Horm} by exhibiting the explicit fundamental
solution for $\mathcal{K}$. It is worth men\-tio\-ning that, in the
introduction of his paper \cite{Horm}, H\"{o}rmander presents the operator
$\mathcal{K}$ as the main `inspiration' for his study: in fact, $\mathcal{K}$
is a hypoelliptic operator \emph{not satisfying} the sufficient conditions for
the hypoellipticity established by H\"{o}rmander himself in his previous work
\cite{HormPrev}. \vspace*{0.1cm}

Starting with the results by H\"{o}rmander, at the beginning of the '90s the
class of KFP operators \emph{with constant coefficients} $a_{ij}$ (of which
the degenerate Kolmogorov operator $\mathcal{K}$ is a particular example) has
been deeply studied by Lanconelli and Polidoro \cite{LP} under a
\emph{geometric viewpoint}. More precisely, they proved that the operator
\[
\mathcal{L}u=\sum_{i,j=1}^{q}a_{ij}\partial_{x_{i}x_{j}}^{2}u+Yu
\]
possesses the following rich underlying geometric structure:

\begin{itemize}
\item[(a)] $\mathcal{L}$ is left-invariant on the \emph{non-commutative Lie
group} $\G=(\mathbb{R}^{N+1},\circ)$, where the composition law $\circ$ is
defined as follows
\begin{align*}
(y,s)\circ(x,t)  &  =(x+E(t)y,t+s)\\
(y,s)^{-1}  &  =(-E(-s)y,-s),
\end{align*}
and $E(t)=\exp(-tB)$ (which is defined for every $t\in\mathbb{R}$ since the
matrix $B$ is nilpotent). For a future reference, we explicitly notice that
\begin{equation}
(y,s)^{-1}\circ(x,t)=(x-E(t-s)y,t-s), \label{eq:convolutionG}%
\end{equation}
and that the Lebesgue measure is the Haar measure, which is also invariant
with respect to the inversion. \vspace*{0.1cm}

\item[(b)] $\mathcal{L}$ is homogeneous of degree $2$ with respect to a
nonisotropic family of \emph{dilations} in $\mathbb{R}^{N+1}$, which are
automorphisms of $\G$ and are defined by
\begin{equation}
D(\lambda)(x,t)\equiv(D_{0}(\lambda)(x),\lambda^{2}t)=(\lambda^{q_{1}}%
x_{1},\ldots,\lambda^{q_{N}}x_{N},\lambda^{2}t), \label{dilations}%
\end{equation}
where the $N$-tuple $(q_{1},\ldots,q_{N})$ is given by
\[
(q_{1},\ldots,q_{N})=(\underbrace{1,\ldots,1}_{m_{0}},\,\underbrace{3,\ldots
,3}_{m_{1}},\ldots,\underbrace{2k+1,\ldots,2k+1}_{m_{k}}).
\]
The integer%
\begin{equation}
\textstyle Q=\sum_{i=1}^{N}q_{i}>N \label{eq:defQhomdim}%
\end{equation}
is called the \emph{homogeneous dimension} of $\mathbb{R}^{N}$, while $Q+2$ is
the homogeneous dimension of $\mathbb{R}^{N+1}$. We explicitly point out that
the exponential matrix $E(t)$ satisfies the following homogeneity property
\begin{equation}
E(\lambda^{2}t)=D_{0}(\lambda)E(t)D_{0}\Big(\frac{1}{\lambda}\Big),
\label{LP 2.20}%
\end{equation}
for every $\lambda>0$ and every $t\in\mathbb{R}$ (see \cite[Rem.\,2.1.]{LP}).
\end{itemize}

Actually, in \cite{LP} the Authors study constant-coefficients KFP operators
corresponding to a wider class of matrices $B$, which are not nilpotent; these
more general operators are hypoelliptic, left-invariant with respect to the
above operation $\circ$, but they are not necessarily homogeneous. For these
operators, an explicit fun\-da\-men\-tal solution is exhibited in \cite{LP}.
We refer to the introduction of the paper \cite{BP} for more details and
references about the quest of a fundamental solution for KFP operators
\emph{before }the paper \cite{LP}.

After the seminal paper \cite{LP}, more general families of degenerate KFP
operators of the kind \eqref{L}, satisfying the same structural conditions on
the matrices $A_{0}$ and $B$ but with variable coefficients $a_{ij}\left(
x,t\right)  $, have been studied by several authors. In particular, Schauder
estimates have been investigated by Di Francesco-Polidoro in \cite{DP}, on
bounded domains, assuming the coefficients $a_{ij}$ H\"{o}lder continuous with
respect to the intrinsic distance induced in $\mathbb{R}^{N+1}$ by the vector
fields $\partial_{x_{1}},...\partial_{x_{q}},Y$. We point out also the papers
by Lunardi \cite{Lu}, Priola \cite{Pr}, Imbert-Mouhot \cite{IM}, Wang-Zhang
\cite{WZ}, and the references therein, on related issues about Schauder
estimates for KFP operators.

Recent researches, especially in the field of stochastic differential
equations (see e.g. \cite{PP}), which are the main motivation to study KFP
operators, suggest the importance of developing a theory allowing the
coefficients $a_{ij}$ to be rough in $t$ (say, $L^{\infty}$), and H\"{o}lder
continuous (in a suitable sense) only w.r.t. the space variables. The Schauder
estimates that one can reasonably expect under this mild assumption consist in
controlling the H\"{o}lder seminorms w.r.t. $x$ of the derivatives involved in
the equations, uniformly in time (we will be more precise in a moment).

For \emph{uniformly parabolic }operators, \emph{partial Schauder estimates},
i.e. the control of the supremum in $t$ of the H\"{o}lder quotient in space of
$\partial_{x_{i}x_{j}}^{2}u$, under the analogous assumption on the
coefficients and the right-hand side of the equation, have been proved already
in 1969 by Brandt \cite{Br69}. In 1980 Knerr \cite{K80} proved that, under the
same assumptions, $\partial_{x_{i}x_{j}}^{2}u$ are actually H\"{o}lder
continuous also in time, on bounded cylinders. See also the paper \cite{Lib}
by Lieberman, 1992, containing a unified presentation of these and related
results. More recently, Krylov and Priola \cite{KP}, 2010, have extended
partial Schauder estimates (on the whole space) to parabolic operators with
lower order unbounded terms while Lorenzi \cite{LL}, 2011, has proved similar
global estimates for operators with possibly unbounded coefficients. We also
point out the more recent paper \cite{DK} by Dong-Kim, 2019, containing futher
generalizations to operators with coefficients merely measurable w.r.t.
\emph{several} variables.

In the present paper we establish global partial Schauder estimates for
degenerate KFP operators \eqref{L} satisfying assumptions (H1)-(H2), with
coefficients $a_{ij}$ H\"{o}lder continuous in space, bounded measurable in
time (see Assumption (H3) here below and Theorem
\ref{Thm main a priori estimates} for the precise statement). We also show
that the second derivatives $\partial_{x_{i}x_{j}}^{2}u$ (for $i,j=1,2,...,q$)
are actually locally H\"{o}lder continuous also w.r.t. time, so extending to
this degenerate context the result proved for uniformly parabolic operators by
Knerr \cite{K80}.

Our technique to establish partial Schauder estimates is deeply rooted in the
study of the model operator
\begin{equation}
\mathcal{L}u=\sum_{i,j=1}^{q}a_{ij}(t)\partial_{x_{i}x_{j}}^{2}u+\sum
_{k,j=1}^{N}b_{jk}x_{k}\partial_{x_{j}}u-\partial_{t}u, \label{L model t}%
\end{equation}
with coefficients \emph{only depending on time }(in a merely $L^{\infty}$
way), which has been started in \cite{BP}. In that paper, an explicit
fundamental solution is built for operators \eqref{L model t}. This
fundamental solution will be the key tool used in the present paper.

Partial Schauder estimates for degenerate KFP operators have been proved also
in the recent paper \cite{CRHM} by Chaudru de Raynal, Honor\'{e}, Menozzi,
with different techniques and without getting the H\"{o}lder control in time
of second order derivatives.

We also quote the preprints of other two papers on KFP operators with
coefficients H\"{o}lder continuous in space and $L^{\infty}$ in time:
\cite{LPP}, by Lucertini, Pagliarani, Pascucci, dealing with the construction
of a fundamental solution for these operators, with consequent results about
the Cauchy problem; and \cite{HW}, by Henderson and Wang, containing partial
Schauder estimates for a special class of KFP operators, with applications to
the Landau equation. The results in \cite{LPP}, \cite{HW} are independent from
and do not contain our results.

Finally, we point out the paper \cite{M} by Menozzi, containing $L^{p}$
estimates for the second order derivatives for KFP operators with coefficients
$a_{ij}$ continuous in space and $L^{\infty}$ in time.

\subsection{Assumptions and main results}

We can now start giving some precise definitions which will allow to state our
main result.

Let us introduce the metric structure related to the operator $\mathcal{L}$
that will be used throughout the following. The vector fields
\[
X_{1}=\partial_{x_{1}},\ldots,X_{q}=\partial_{q},X_{0}=Y
\]
form a system of H\"{o}rmander vector fields in $\mathbb{R}^{N+1}$,
left-invariant w.r.t.\thinspace the com\-po\-si\-tion law $\circ$. The vector
fields $X_{i}=\partial_{x_{i}}$ (with $i=1,...,q$) are homogeneous of degree
$1$, while $X_{0}=Y$ is homogeneous of degree $2$ w.r.t.\thinspace the
dilations $D(\lambda)$. As every set of H\"{o}rmander vector fields with
drift, the system
\[
\mathbf{X}=\{X_{0},X_{1},\ldots,X_{q}\}
\]
induces a (weighted) control distance $d_{\mathbf{X}}$ in $\mathbb{R}^{N+1}$;
we now review this definition in our special case. First of all, given
$\xi=(x,t),\,\eta=(y,s)\in\mathbb{R}^{N+1}$ and $\delta>0$, we denote by
$C_{\xi,\eta}(\delta)$ the class of \emph{absolutely continuous} curves
\[
\varphi:[0,1]\longrightarrow\mathbb{R}^{N+1}%
\]
which satisfy the following properties:

\begin{itemize}
\item[(i)] $\varphi(0) = \xi$ and $\varphi(1) = \eta$; \vspace*{0.1cm}

\item[(ii)] for almost every $t\in\lbrack0,1]$ one has
\[
\textstyle\varphi^{\prime}(t)=\sum_{i=1}^{q}a_{i}(t)\varphi_{i}(t)+a_{0}%
(t)Y_{\varphi(t)},
\]
where $a_{0},\ldots,a_{q}:[0,1]\rightarrow\mathbb{R}$ are measurable functions
such that
\[
\text{$|a_{i}(t)|\leq\delta$ (for $i=1,\ldots,q$)\quad and \quad
$|a_{0}(t)|\leq\delta^{2}$}\qquad\text{a.e. on $[0,1]$}.
\]

\end{itemize}

We then define
\[
d_{\mathbf{X}}(\xi,\eta)=\inf\big\{\delta>0:\,\exists\,\,\varphi\in
C_{\xi,\eta}(\delta)\big\}.
\]
Since $X_{0},X_{1},\ldots,X_{q}$ satisfy H\"{o}rmander's rank condition, it is
well-known that the function $d_{\mathbf{X}}$ is a distance in $\mathbb{R}%
^{N+1}$ (see, e.g., \cite[Prop.\,1.1]{NSW}); in particular, for every fixed
$\xi,\eta\in\mathbb{R}^{N+1}$ there always exists $\delta>0$ such that
$C_{\xi,\eta}(\delta)\neq\varnothing$. In addition, by the
invariance/homogeneity properties of the $X_{i}$'s, we see that

\begin{itemize}
\item[(a)] $d_{\mathbf{X}}$ is left-invariant with respect to $\circ$, that
is,
\begin{equation}
\label{d_X}d_{\mathbf{X}}(\xi,\eta) =d_{\mathbf{X}}( \eta^{-1}\circ\xi,0)
\end{equation}

\item[(b)] $d_{\mathbf{X}}$ is is jointly $1$-homogeneous with respect to
$D(\lambda)$, that is
\begin{equation}
\label{rho_X}d_{\mathbf{X}}( D(\lambda)\xi,D(\lambda)\eta) = \lambda
d_{\mathbf{X}}(\xi,\eta)\qquad\text{ for every $\lambda>0$}.
\end{equation}

\end{itemize}

As a consequence of \eqref{d_X}, the function $\rho_{\mathbf{X}}(\xi) :=
d_{\mathbf{X}}(\xi,0)$ satisfies

\begin{enumerate}
\item $\rho_{\mathbf{X}}(\xi^{-1}) = \rho_{\mathbf{X}}(\xi)$; \vspace*{0.05cm}

\item $\rho_{\mathbf{X}}(\xi\circ\eta) \leq\rho_{\mathbf{X}}(\xi)
+\rho_{\mathbf{X}}(\eta)$;
\end{enumerate}

moreover, by \eqref{rho_X} we also have

\begin{itemize}
\item[(1)'] $\rho_{X}(\xi) \geq0$ and $\rho_{\mathbf{X}}(\xi)
=0\,\Leftrightarrow\,\xi=0$; \vspace*{0.05cm}

\item[(2)'] $\rho_{\mathbf{X}}(D(\lambda)\xi) = \lambda\rho_{\mathbf{X}}(\xi)$,
\end{itemize}

and this means that $\rho_{\mathbf{X}}$ is a \emph{homogeneous norm} in
$\mathbb{R}^{N+1}$. \medskip

We now observe that also the function
\begin{equation}
\rho(\xi)=\rho(x,t):=\Vert x\Vert+\sqrt{|t|}=\sum_{i=1}^{N}|x_{i}|^{1/q_{i}%
}+\sqrt{|t|} \label{eq:defrhonorm}%
\end{equation}
is a homogeneous norm in $\mathbb{R}^{N+1}$ (i.e., it satisfies properties
(1)'-(2)' above), and therefore it is \emph{globally equivalent} to $\rho_{X}%
$: there exist $c_{1},c_{2}>0$ such that
\[
c_{1}\rho_{\mathbf{X}}(\xi)\leq\rho(\xi)\leq c_{2}\rho_{\mathbf{X}}(\xi
)\qquad\forall\,\,\xi\in\mathbb{R}^{N+1}.
\]
As a consequence of this fact, the map
\begin{equation}
d(\xi,\eta):=\rho(\eta^{-1}\circ\xi) \label{d}%
\end{equation}
is a left-invariant, $1$-homogeneous \emph{quasi-distance} on $\mathbb{R}%
^{N+1}$. This means, precisely, that there exists a `structural constant'
$\bd{\kappa}>0$ such that
\begin{align}
d(\xi,\eta)  &  \leq\bd{\kappa}\big(d(\xi,\zeta)+d(\eta,\zeta)\big)\qquad
\forall\,\,\xi,\eta,\zeta\in\mathbb{R}^{N+1};\label{quasitriangle}\\
d(\xi,\eta)  &  \leq\bd{\kappa}\,d(\eta,\xi)\qquad\forall\,\,\xi,\eta
\in\mathbb{R}^{N+1}. \label{quasisymmetric}%
\end{align}
The quasi-distance $d$ is \emph{globally equivalent} to the control distance
$d_{\mathbf{X}}$; hence, we will systematically use this quasi-distance $d$
and the associated balls
\[
B_{r}(\xi):=\big\{\eta\in\mathbb{R}^{N+1}:\,d(\eta,\xi)<r\big\}\qquad
(\text{for $\xi\in\mathbb{R}^{N+1}$ and $r>0$}).
\]

\begin{remark}
\label{rem:propd} For a future reference, we list below some properties $d$.

\begin{enumerate}
\item Owing to \eqref{eq:convolutionG}, we see that $d$ has the following
explicit expression
\begin{equation}
d(\xi,\eta)=\Vert x-E(t-s)y\Vert+\sqrt{|t-s|}, \label{eq:explicitd}%
\end{equation}
for every $\xi=(x,t),\,\eta=(y,s)\in\mathbb{R}^{N+1}$. \medskip

\item Since $E(0)=\mathbb{O}$, from \eqref{eq:explicitd} we get
\begin{equation}
d((x,t),(y,t))=\Vert x-y\Vert\qquad\text{for every $x,y\in\mathbb{R}^{N}$ and
$t\in\mathbb{R}$}, \label{d stesso t}%
\end{equation}
from which we derive that the quasi-distance $d$ \emph{is sym\-me\-tric when
applied to points with the same $t$-coordinate}. We explicitly emphasize that
an a\-na\-lo\-gous property for points with the same $x$-coordinate \emph{does
not hold}: in fact, for every fixed $x\in\mathbb{R}^{N}$ and $t,s\in
\mathbb{R}$ we have
\[
d((x,t),(x,s))=\Vert x-E(t-s)x\Vert+\sqrt{|t-s|}\neq\sqrt{|t-s|}.
\]

\item Let $\xi\in\mathbb{R}^{N+1}$ be fixed, and let $r>0$. Since $d$
satisfies the quasi-triangular inequality \eqref{quasitriangle}, if $\eta
_{1},\eta_{2}\in B_{r}(\xi)$ we have
\[
d(\eta_{1},\eta_{2})<2\bd{\kappa}r.
\]

\item Taking into account the very definition of $d$, and bearing in mind that
$\rho$ is a \emph{homogeneous norm} in $\mathbb{R}^{N+1}$, it is readily seen
that
\begin{equation}
B_{r}(\xi)=\xi\circ B_{r}(0)=\xi\circ D_{r}\big(B_{1}(0)\big)\quad
\forall\,\,\xi\in\mathbb{R}^{N+1},\,r>0. \label{eq:balltraslD}%
\end{equation}
From this, since the Lebesgue measure is a Haar measure on $\G=(\mathbb{R}%
^{N+1},\circ)$, we immediately obtain the following identity
\begin{equation}
|B_{r}(\xi)|=|B_{r}(0)|=\omega\,r^{Q+2} \label{measure ball}%
\end{equation}
where $\omega:=|B_{1}(0)|>0$. Identity \eqref{measure ball} illustrates the
role of $Q+2$ as the homogeneous dimension of $\mathbb{R}^{N+1}$
(w.r.t.\thinspace the dilations $D(\lambda)$).
\end{enumerate}
\end{remark}

The quasi-distance $d$ allows us to define the H\"{o}lder spaces which will be
used in the paper. We are interested both in H\"{o}lder norms which measure
the \emph{joint continuity in $(x,t)$} and in H\"{o}lder norms which measure
the \emph{continuity in $x$ alone, for fixed $t$}.

\begin{definition}
\label{def:Holderspacesd} For $-\infty\leq\tau<T\leq+\infty$, let
$\Omega=\mathbb{R}^{N}\times\left(  \tau,T\right)  $, and let $f:\Omega
\rightarrow\mathbb{R}$. Given any number $\alpha\in(0,1)$, we introduce the
notation:
\begin{align*}
|f|_{C^{\alpha}(\Omega)}  &  =\sup\left\{  \frac{|f(\xi)-f(\eta)|}{d(\xi
,\eta)^{\alpha}}:\,\text{$\xi,\eta\in\Omega$ and $\xi\neq\eta$}\right\}
\\[0.15cm]
|f|_{C_{x}^{\alpha}(\Omega)}  &  =\operatorname*{supess}_{t\in\left(
\tau,T\right)  }\sup\left\{  \frac{|f(x,t)-f(y,t)|}{d((x,t),(y,t))^{\alpha}%
}:\,\text{$x,y\in$ }\mathbb{R}^{N},\text{$x\neq y$}\right\} \\
&  =\operatorname*{supess}_{t\in\left(  \tau,T\right)  }\sup\left\{
\frac{|f(x,t)-f(y,t)|}{\Vert x-y\Vert^{\alpha}}:\,\text{$x,y\in$ }%
\mathbb{R}^{N},\text{$x\neq y$}\right\}
\end{align*}
(where the last equality holds by \eqref{d stesso t}). Accordingly, we define
the spaces $C^{\alpha}(\Omega)$ and $C_{x}^{\alpha}(\Omega)$ as follows:
\begin{align}
&  C^{\alpha}(\Omega):=\left\{  f\in C(\Omega)\cap L^{\infty}(\Omega
):\,|f|_{C^{\alpha}(\Omega)}<\infty\right\} \label{eq:defCalfa}\\[0.1cm]
&  C_{x}^{\alpha}(\Omega):=\left\{  f\in L^{\infty}(\Omega):\,|f|_{C_{x}%
^{\alpha}(\Omega)}<\infty\right\}  \label{eq:defCalfax}%
\end{align}

\end{definition}

\begin{remark}
\label{rem:distance} The space $C^{\alpha}(\Omega)$ endowed with the following
norm
\[
\Vert f\Vert_{C^{\alpha}(\Omega)}:=\Vert f\Vert_{L^{\infty}(\Omega
)}+|f|_{C^{\alpha}(\Omega)}\qquad(f\in C^{\alpha}(\Omega))
\]
is a Banach space. Analogously, the space $C_{x}^{\alpha}(\Omega)$ endowed
with the norm
\[
\Vert f\Vert_{C_{x}^{\alpha}(\Omega)}:=\Vert f\Vert_{L^{\infty}(\Omega
)}+|f|_{C_{x}^{\alpha}(\Omega)}\qquad(f\in C_{x}^{\alpha}(\Omega))
\]
is a Banach space.
\end{remark}

We can now make precise our regularity assumption on the coefficients $a_{ij}$.

\begin{itemize}
\item[\textbf{(H3)}] There exists $\alpha\in(0,1)$ such that
\[
a_{ij}\in C_{x}^{\alpha}\left(  \mathbb{R}^{N+1}\right)  \qquad\text{for every
$1\leq i,j\leq q$}.
\]
In all the estimates appearing in next sections, the number
\begin{equation}
\Lambda=\max_{i,j=1,...,q}\Vert a_{ij}\Vert_{C_{x}^{\alpha}(\mathbb{R}^{N+1}%
)}, \label{Lambda}%
\end{equation}
together with the ellipticity constant $\nu$ in \eqref{nu}, will quantify the
de\-pen\-den\-ce of the constants on the coefficients $a_{ij}$.
\end{itemize}

We now turn to define the functions spaces to which our \emph{solution $u$}
will belong.

\begin{definition}
\label{def:spacesS} Throughout the following, given $T\in\mathbb{R}$ we set
\[
S_{T}:=\mathbb{R}^{N}\times(-\infty,T).
\]
We then define $\mathcal{S}^{0}(S_{T})$ as the space of all fun\-cti\-ons
$u:\overline{S}_{T}\rightarrow\mathbb{R}$ such that

\begin{itemize}
\item[(i)] $u\in C(\overline{S_{T}}) \cap L^{\infty}(S_{T})$;

\item[(ii)] for every $1\leq i,j\leq q$, the \emph{distributional derivatives}
$\partial_{x_{i}}u,\partial_{x_{i}x_{j}}^{2}u\in L^{\infty}(S_{T})$;

\item[(iii)] the \emph{distributional derivative} $Yu\in L^{\infty}(S_{T})$.
\end{itemize}

Moreover, given any number $\alpha\in(0,1)$, we define
\[
\mathcal{S}^{\alpha}(S_{T}):=\{u\in\mathcal{S}^{0}(S_{T}):\,\partial_{x_{i}%
}u,\text{$\partial_{x_{i}x_{j}}u,\,Yu\in C_{x}^{\alpha}(S_{T})$ for $1\leq
i,j\leq q$}\}.
\]
Finally, given any $\tau\in\mathbb{R}$ with $\tau<T$, we define
\begin{align*}
&  \mathcal{S}^{0}(\tau;T)=\{u\in\mathcal{S}^{0}(S_{T}):\,\text{$u(x,t)=0$ for
every $t\leq\tau$}\},\\
&  \mathcal{S}^{\alpha}(\tau;T):=\mathcal{S}^{\alpha}(S_{T})\cap
\mathcal{S}^{0}(\tau;T)\qquad(\text{for $\alpha\in(0,1)$}).
\end{align*}

\end{definition}

\begin{remark}
\label{rem:regulLLu} On account of assumption (H3), we immediately obtain the
fol\-low\-ing facts which shall be repeatedly used throughout the rest of the paper.

\begin{itemize}
\item[(1)] If $u\in\mathcal{S}^{0}(S_{T})$, then $\mathcal{L}u\in L^{\infty
}(S_{T})$. \vspace*{0.1cm}

\item[(2)] If $u\in\mathcal{S}^{\alpha}(S_{T})$, then $\mathcal{L}u\in
C_{x}^{\alpha}(S_{T})$. \vspace*{0.1cm}

\item[(3)] If $u\in\mathcal{S}^{0}(S_{T})$ and $\partial_{x_{i}x_{j}}%
^{2}u,\,\mathcal{L}u\in C_{x}^{\alpha}(S_{T})$ ($1\leq i,j\leq q$), then
$Yu\in C_{x}^{\alpha}(S_{T})$.
\end{itemize}
\end{remark}

Some of the results in the next sections are proved under the assumption that
$u\in\mathcal{S}^{0}(S_{T})$ and $\mathcal{L}u\in C_{x}^{\alpha}(S_{T})$; this
is slightly weaker than assuming $u\in\mathcal{S}^{\alpha}(S_{T})$.

\begin{remark}
[Regularity of functions in $S^{\alpha}(S_{T})$]%
\label{Remark regularity S-alfa} We will prove in the subsequent sections the
following `higher-regularity' results:

\begin{itemize}
\item[(1)] if $u\in\mathcal{S}^{0}(S_{T})$, then $u$ and $\partial_{x_{1}%
}u,\ldots,\partial_{x_{q}}u$ are locally H\"{o}lder-continuous in the joint
variables (see, precisely, Proposition \ref{Prop interpolaz sharp});
\vspace*{0.1cm}

\item[(2)] if $u\in\mathcal{S}^{\alpha}(S_{T})$ for some $\alpha\in(0,1)$,
then the distributional derivatives $\partial_{x_{i}x_{j}}^{2}u$ (for $1\leq
i,j\leq q$) are actually continuous (and locally H\"{o}lder continuous in a
weaker sense) on $S_{T}$ (see Theorem
\ref{Thm global Schauder in space time 2}).
\end{itemize}

As a consequence, every function $u\in\mathcal{S}^{\alpha}(S_{T})$ actually
has \emph{classical continuous derivatives} $\partial_{x_{i}}^{2}%
u,\partial_{x_{i}x_{j}}^{2}u$ (for $1\leq i,j\leq q$); instead, the
distributional derivative $Yu$ is continuous in space for every fixed $t$, but
it may be only $L^{\infty}$ w.r.t.\thinspace time.
\end{remark}

We are finally in position to state our main result.

\begin{theorem}
[Schauder estimates]\label{Thm main a priori estimates} Let $\mathcal{L}$ be
an operator as in \eqref{L}, and assume that \emph{(H1)}, \emph{(H2)},
\emph{(H3)} hold, for some $\alpha\in(0,1)$.

Then, the following Schauder-type estimates hold true.

\begin{enumerate}
\item For every $T>0$ there exists a constant $c>0$, depending on $T$,
$\alpha$, the matrix $B$ in \eqref{B} and the numbers $\nu$ and $\Lambda$ in
\eqref{nu}-\eqref{Lambda}, respectively, such that
\begin{align*}
\sum_{i,j=1}^{q}\Vert\partial_{x_{i}x_{j}}^{2}u\Vert_{C_{x}^{\alpha}(S_{T})}
&  +\Vert Yu\Vert_{C_{x}^{\alpha}(S_{T})}+\sum_{i=1}^{q}\Vert\partial_{x_{i}%
}u\Vert_{C^{\alpha}(S_{T})}+\Vert u\Vert_{C^{\alpha}(S_{T})}\\
&  \leq c\big(\Vert\mathcal{L}u\Vert_{C_{x}^{\alpha}(S_{T})}+\Vert
u\Vert_{C^{0}(S_{T})}\big).
\end{align*}
{for every function $u\in\mathcal{S}^{\alpha}(S_{T})$}. \vspace*{0.1cm}

\item For every $T>\tau>-\infty$ and every compact set $K\subset\mathbb{R}%
^{N}$ there exists a constant $c>0$, depending on $K,\tau,T,\alpha
,B,\nu,\Lambda$, such that
\begin{align*}
&  |\partial_{x_{i}x_{j}}^{2}u(\xi)-\partial_{x_{i}x_{j}}^{2}u(\eta)|\\
&  \qquad\leq c\big(\Vert\mathcal{L}u\Vert_{C_{x}^{\alpha}(S_{T})}+\Vert
u\Vert_{C^{\alpha}(S_{T})}\big)\big(d(\xi,\eta)^{\alpha}+|t-s|^{\alpha/q_{N}%
}\big)
\end{align*}
for every $\xi=(x,t),\,\eta=(y,s)\in K\times\left[  \tau,T\right]  $ and every
$u\in\mathcal{S}^{\alpha}(S_{T})$. Here, the number $q_{N}$ is the largest
exponent in the dilations $D(\lambda)$, see \eqref{dilations}.
\end{enumerate}
\end{theorem}

\begin{remark}
(i). As observed in Remark \ref{Remark regularity S-alfa}, the finiteness of
the quantities
\[
\textstyle\sum_{i=1}^{q}\Vert\partial_{x_{i}}u\Vert_{C^{\alpha}(S_{T})}%
,\quad\Vert u\Vert_{C^{\alpha}(S_{T})},
\]
as well as the finiteness of the space-time H\"{o}lder quotient in point (2)
of the above theorem, are not obvious \emph{a priori} for a function in
$\mathcal{S}^{\alpha}(S_{T})$, but they will be actually proved.

(ii). While $\partial_{x_{i}x_{j}}^{2}u$ ($i,j=1,2,...,q$) are locally
H\"{o}lder continuous in space and time, note that a similar property cannot
be assured, in general, for $Yu$. To see this, it is enough to consider an
equation of the kind%
\[
\mathcal{L}u(x,t)=f(t)
\]
with $f$ bounded discontinuous function, and $u$ independent of $x$.

(iii). Since, in the degenerate case, $q_{N}\geq3$, the term $|t_{1}%
-t_{2}|^{\alpha/q_{N}}$ in the right-hand side of
\eqref{eq:schauderspacetimeuxixj} is larger than the `expected'
\[
|t_{1}-t_{2}|^{\alpha/2},
\]
(at least when $|t_{1}-t_{2}|\leq1$). Also, the constant ${c}>0$ depends on
the fixed compact set $K\times\lbrack\tau,T]\subseteq S_{T}$. On the other
hand, we observe that this mild $t$\--con\-ti\-nui\-ty of $\partial
_{x_{i}x_{j}}^{2}u$ is obtained \emph{without any $t$-continuity assumption on
$\mathcal{L}u$}. Moreover, from the proof of Theorem
\ref{Thm local Schauder time} it will be apparent that in the uniformly
parabolic case ($B=0$ and $q=N$) our argument would give exactly%
\[
|\partial_{x_{i}x_{j}}^{2}u(\xi)-\partial_{x_{i}x_{j}}^{2}u(\eta)|\leq
c\big\{\Vert\mathcal{L}u\Vert_{C_{x}^{\alpha}(S_{T})}+\Vert u\Vert_{C^{\alpha
}(S_{T})}\big\}\big(|x-y|+|t-s|^{\alpha/2}\big).
\]
Our result is therefore consistent with the classical result by Knerr
\cite{K80} which holds for uniformly parabolic operators on bounded cylinders.
\end{remark}

\subsection{Structure of the paper}

After a short section of preliminaries (\S \ref{sec prelim}), the paper will
proceed in two main steps: the study of the model operator \eqref{L model t}
with coefficients only depending on $t$ (\S \ref{sec operators t}) and the
study of operators \eqref{L} of general type (\S \ \ref{Sec operators a(x,t)}%
). In section \ref{sec operators t} we deepen the study of the fundamental
solution for model operators \eqref{L model t} computed in the previous paper
\cite{BP}. Thanks to the stronger assumption that we make in this paper on the
matrix $B$ with respect to those assumed in \cite{BP} (the corresponding
operators with \emph{constant }$a_{ij}$ in this paper are both left invariant
and homogeneous, while in \cite{BP} they are only left invariant) it is
possible to sharpen the estimates on the fundamental solutions. Actually, in
\S \ \ref{sec sharp estimates fund sol} we establish sharp upper bounds on the
fundamental solution and its space derivatives \emph{of every order}, and
other relevant propertis of this kernels. These upper bounds and properties
allow us to establish, in \S \ \ref{sec repr formulas}, suitable
representation formulas for a function $u$ and its derivatives $\partial
_{x_{i}x_{j}}^{2}u$ in terms of $\mathcal{L}u$. In turn, thanks to these
representation formulas we will establish H\"{o}lder estimates in space for
$\partial_{x_{i}x_{j}}^{2}u$ in \S \ \ref{sec schauder space model op}, and
local H\"{o}lder estimates in space and time for $\partial_{x_{i}x_{j}}^{2}u$
in \S \ \ref{section Schauder time}. These results are established with
techniques of singular integrals, and refer to operators with coefficients
only depending on $t$. Starting with these results, in
\S \ref{Sec operators a(x,t)} analogous results are established for operators
with coefficients $a_{ij}\left(  x,t\right)  $, exploiting the classical
perturbative method used for Schauder estimates. First, in
\S \ref{sec local schauder space}, H\"{o}lder estimates for $\partial
_{x_{i}x_{j}}^{2}u$ are proved for functions with small support.\ Then, in
\S \ \ref{sec interpolation}, some interpolation inequalities on first order
derivatives are proved, which allow to get, in
\S \ \ref{sec global schauder space}, global Schauder estimates in space,
extended in \S \ \ref{sec schauder time} to H\"{o}lder estimates in space and
time on $\partial_{x_{i}x_{j}}^{2}u$.

\bigskip

\noindent\textbf{Acknowledgements.} We wish to thank Sergio Polidoro for
several useful discussions on the subject of this paper.

\section{Preliminaries and known results\label{sec prelim}}

The following \textquotedblleft Lagrange' theorem\textquotedblright, which is
well known for systems of left-invariant homogeneous H\"{o}rmander vector
fields, will be useful.

\begin{theorem}
\label{Lagrange} There exist an absolute constant $c>0$ and a number
$\delta\in(0,1)$, depending on $\bd{\kappa}$ in
\eqref{quasitriangle}-\eqref{quasisymmetric}, such that, for every fixed
${\xi_{0}}\in\mathbb{R}^{N+1}$, every $r>0$ and every $f$ Lipschitz-continuous
in $\overline{B}_{r}(\xi_{0})$, one has
\[
|f(\xi)-f({\xi_{0}})|\leq c\Big(d(\xi,{\xi_{0}})\cdot\sup_{B_{R}({\xi_{0}}%
)}\sqrt{\sum_{i=1}^{q}|\partial_{x_{i}}f|^{2}}+d(\xi,{\xi_{0}})^{2}\cdot
\sup_{B_{R}({\xi_{0}})}|Yf|\Big),
\]
for every $\xi\in B_{r}({\xi_{0}})$. Moreover, one also has
\[
|f(\xi)-f(\eta)|\leq c\Big(d(\xi,\eta)\cdot\sup_{B_{R}({\xi_{0}})}\sqrt
{\sum_{i=1}^{q}|\partial_{x_{i}}f|^{2}}+d(\xi,\eta)^{2}\cdot\sup_{B_{R}%
({\xi_{0}})}|Yf|\Big)
\]
for every $\xi,\eta\in B_{\delta r}({\xi_{0}})$.
\end{theorem}

The next geometric lemma follows by standard computations in doubling metric
measure spaces, recalling \eqref{measure ball}.

\begin{lemma}
\label{lem:doubling} Let $\alpha>0$ be fixed, and let $Q$ be as in
\eqref{eq:defQhomdim}. Then, there exists a constant $c_{\alpha}>0$ such that,
for every $\xi\in\mathbb{R}^{N+1}$ and every $r>0$, one has
\begin{align}
\int_{\{\eta:\,d(\xi,\eta)<r\}}\frac{1}{d(\xi,\eta)^{Q+2-\alpha}}\,\eta &
\leq c_{\alpha}r^{\alpha}\label{doubling <R}\\[0.1cm]
\int_{\{\eta:\,d(\xi,\eta)>r\}}\frac{1}{d(\xi,\eta)^{Q+2+\alpha}}\,d\eta &
\leq\frac{c_{\alpha}}{r^{\alpha}}. \label{doubling >R}%
\end{align}

\end{lemma}

We also state the following simple fact which shall be repeatedly used
throughout the rest of the paper.

\begin{lemma}
\label{lem:equivalentd} There exists an absolute constant $\bd{\vartheta}>0$
such that, if $\xi_{1},\xi_{2}$ and $\eta$ are points in $\mathbb{R}^{N+1}$
which satisfy $d(\xi_{1},\eta)\geq2\bd{\kappa}\,d(\xi_{1},\xi_{2})$, one has
\begin{equation}
\bd{\vartheta}^{-1}d(\xi_{2},\eta)\leq d(\xi_{1},\eta)\leq\bd{\vartheta}d(\xi
_{2},\eta), \label{constant M}%
\end{equation}
Here, $\bd{\kappa}>0$ is the constant appearing in \eqref{quasitriangle}-\eqref{quasisymmetric}.
\end{lemma}

Thanks to Lemmas \ref{lem:doubling}-\ref{lem:equivalentd} we can establish the
following $C^{\alpha}$ continuity result about \textquotedblleft fractional
integrals\textquotedblright\, which will be useful in our estimates.

\begin{proposition}
[Fractional integrals]\label{Prop fractional generale}Let $Q$ be as in
\eqref{eq:defQhomdim} and $\beta\in\lbrack1,Q+2)$. Moreover, let $k=k(\xi
,\eta)$ be a kernel satisfying the following properties:

\begin{enumerate}
\item there exists a constant $c_{1}>0$ such that
\begin{equation}
|k(\xi,\eta)|\leq\frac{c_{1}}{d(\xi,\eta)^{Q+2-\beta}}\quad\forall\,\,\xi
\neq\eta\in\mathbb{R}^{N+1}; \label{Frac1}%
\end{equation}

\item there exist constants $\sigma,\,c_{2} > 0$ such that
\begin{equation}
\label{Frac2}|k(\xi_{1},\eta)-k(\xi_{2},\eta)| \leq c_{2} \frac{d(\xi_{1}%
,\xi_{2})}{d(\xi_{1},\eta)^{Q+3-\beta}} \quad\forall\,\,d(\xi_{1},\eta)
\geq\sigma\,d(\xi_{1},\xi_{2}).
\end{equation}

\end{enumerate}

For every fixed $\overline{\xi}\in\mathbb{R}^{N+1}$ and $r>0$, we introduce
the function space
\[
\mathbb{X}_{\infty}(B_{r}(\overline{\xi})):=\{f\in L^{\infty}(\mathbb{R}%
^{N+1}):\,\text{$f\equiv0$ a.e.\thinspace in $\mathbb{R}^{N+1}\setminus
B_{r}(\overline{\xi})$}\},
\]
and we define the linear operator
\[
\mathbb{X}_{\infty}(B_{r}(\overline{\xi}))\ni f\mapsto Tf(\xi)=\int%
_{\mathbb{R}^{N+1}}k(\xi,\eta)f(\eta)\,d\eta.
\]
Then, for every $\alpha\in(0,1)$ there exists an `absolute' constant $c>0$,
depending on $\alpha,\beta$ but \emph{in\-de\-pen\-dent of $f,\overline{\xi
},r$ and of the kernel $k$}, such that
\begin{align}
\Vert Tf\Vert_{L^{\infty}(B_{r}(\overline{\xi}))}  &  \leq cc_{1}r^{\beta
}\Vert f\Vert_{L^{\infty}(B_{R}(\overline{\xi}))}\label{eq:estimsupfSI}\\
|Tf|_{C^{\alpha}(B_{r}(\overline{\xi}))}  &  \leq c(c_{1}+c_{2})r^{\beta
-\alpha}\Vert f\Vert_{L^{\infty}(B_{R}(\overline{\xi}))}.
\label{eq:estimCalfafSI}%
\end{align}

\end{proposition}

\begin{proof}
Let $f\in\mathbb{X}_{\infty}(B_{r}(\overline{\xi}))$ be arbitrarily fixed.
Using \eqref{doubling <R} and \eqref{Frac1}, and taking into account Remark
\ref{rem:propd}-(3), for every $\xi\in B_{r}(\overline{\xi})$ we have
\begin{align*}
|Tf(\xi)|  &  \leq\int_{\{d(\eta,\overline{\xi})<r\}}\frac{c_{1}}{d(\xi
,\eta)^{Q+2-\beta}}|f(\eta)|\,d\eta\\
&  \leq c_{1}\Vert f\Vert_{L^{\infty}(B_{r}(\overline{\xi}))}\int%
_{\{d(\eta,\overline{\xi})<r\}}\frac{1}{d(\xi,\eta)^{Q+2-\beta}}\,d\eta\\
&  \leq c_{1}\Vert f\Vert_{L^{\infty}(B_{r}(\overline{\xi}))}\int%
_{\{d(\xi,\eta)<2\bd{\kappa}r\}}\frac{1}{d(\xi,\eta)^{Q+2-\beta}}\,d\eta\\
&  \leq c_{1}c_{\beta}(2\bd{\kappa})^{\beta}r^{\beta}\Vert f\Vert_{L^{\infty
}(B_{R}(\overline{\xi}))},
\end{align*}
hence
\[
\Vert Tf\Vert_{L^{\infty}(B_{r}(\overline{\xi}))}\leq c^{\prime}c_{1}r^{\beta
}\Vert f\Vert_{L^{\infty}(B_{R}(\overline{\xi}))}\qquad\text{(with $c^{\prime
}:=c_{\beta}(2\bd{\kappa})^{\beta}$)},
\]
which is \eqref{eq:estimsupfSI}. Moreover, for every $\xi_{1},\xi_{2}\in
B_{r}(\overline{\xi})$ one has
\begin{equation}%
\begin{split}
&  |Tf(\xi_{1})-Tf(\xi_{2})|\leq\Vert f\Vert_{L^{\infty}(B_{r}(\overline{\xi
}))}\int_{B_{r}(\overline{\xi})}|k(\xi_{1},\eta)-k(\xi_{2},\eta)|\,d\eta\\
&  \quad=\Vert f\Vert_{L^{\infty}(B_{r}(\overline{\xi}))}\bigg(\int%
_{\{\eta:\,d(\xi_{1},\eta)\geq\sigma d(\xi_{1},\xi_{2})\}}+\int_{\{\eta\in
B_{r}(\overline{\xi}):\,d(\xi_{1},\eta)<\sigma d(\xi_{1},\xi_{2}%
)\}}\bigg)\{\cdots\}\,d\eta\\[0.2cm]
&  \quad\equiv\Vert f\Vert_{L^{\infty}(B_{r}(\overline{\xi}))}\cdot
\big(\mathrm{A}+\mathrm{B}\big),
\end{split}
\label{eq:estimAB}%
\end{equation}
Next, by \eqref{doubling <R}, \eqref{Frac2} and Remark \ref{rem:propd}-(3), we
get
\begin{equation}%
\begin{split}
\mathrm{A}  &  \leq c_{2}\int_{B_{r}^{\prime}(\overline{\xi})}\frac{d(\xi
_{1},\xi_{2})}{d(\xi_{1},\eta)^{Q+3-\beta}}\,d\eta\leq\frac{c_{2}}%
{\sigma^{\alpha}}\cdot d(\xi_{1},\xi_{2})^{\alpha}\int_{B_{r}(\overline{\xi}%
)}\frac{d(\xi_{1},\eta)^{1-\alpha}}{d(\xi_{1},\eta)^{Q+3-\beta}}\,d\eta\\
&  \leq\frac{c_{2}}{\sigma^{\alpha}}\cdot d(\xi_{1},\xi_{2})^{\alpha}%
\int_{\{d(\xi_{1},\eta)<2\bd{\kappa}r\}}\frac{1}{d\left(  \xi_{1},\eta\right)
^{Q+2-(\beta-\alpha)}}\,d\eta\\[0.2cm]
&  (\text{by \eqref{doubling <R}, since $0<\alpha<1\leq\beta$})\\[0.2cm]
&  \leq c\,d(\xi_{1},\xi_{2})^{\alpha}r^{\beta-\alpha}.
\end{split}
\label{eq:estimIntegralASI}%
\end{equation}
As to $\mathrm{B}$, again by \eqref{doubling <R} and \eqref{Frac1} we get
\begin{equation}%
\begin{split}
\mathrm{B}  &  \leq\int_{B_{r}^{\prime\prime}(\overline{\xi})}(|k(\xi_{1}%
,\eta)|+|k(\xi_{2},\eta)|\,d\eta\\
&  \leq c_{1}\int_{B_{r}^{\prime\prime}(\overline{\xi})}\Big(\frac{1}%
{d(\xi_{1},\eta)^{Q+2-\beta}}+\frac{1}{d(\xi_{2},\eta)^{Q+2-\beta}}%
\Big)d\eta\\
&  \leq c_{1}\bigg(\int_{\{d(\xi_{1},\eta)<\sigma d(\xi_{1},\xi_{2})\}}%
\frac{1}{d(\xi_{1},\eta)^{Q+2-\beta}}\,d\eta\\
&  \qquad\qquad+\int_{\{d(\xi_{2},\eta)<\bd{\kappa}^{2}(\sigma+1)d(\xi_{1}%
,\xi_{2})\}}\frac{1}{d(\xi_{1},\eta)^{Q+2-\beta}}\,d\eta\bigg)\\[0.2cm]
&  \leq c\,d(\xi_{1},\xi_{2})^{\beta}\\[0.2cm]
&  (\text{by Remark \ref{rem:propd}-(3), since $\xi_{1},\xi_{2}\in
B_{r}(\overline{\xi})$})\\[0.2cm]
&  \leq c\,d(\xi_{1},\xi_{2})^{\alpha}r^{\beta-\alpha}.
\end{split}
\label{eq:estimIntegralBSI}%
\end{equation}
Due to the arbitrariness of $\xi_{1},\xi_{2}\in B_{r}(\overline{\xi})$, by
\eqref{eq:estimAB}-to-\eqref{eq:estimIntegralBSI} we get
\[
|Tf|_{C^{\alpha}(B_{r}(\overline{\xi}))}\leq c\,r^{\beta-\alpha}\Vert
f\Vert_{L^{\infty}(B_{r}(\overline{\xi}))},
\]
so the proof is complete.
\end{proof}

We end this section with another useful technical lemma.

\begin{lemma}
\label{Lemma E(s)x} There exists an absolute constant $c>0$ such that
\begin{equation}
\label{eq:Etxhom}\|E(t)x\| \leq c\rho(x,t) = c\big(\|x\|+\sqrt{|t|}%
\big) \qquad\forall\,\,x\in\mathbb{R}^{N},\,t\in\mathbb{R}.
\end{equation}

\end{lemma}

\begin{proof}
First of all, since the function $(x,t)\mapsto\Vert E(t)x\Vert$ is continuous
on $\mathbb{R}^{N+1}$, it is possible to find a constant $M>0$ such that
\[
\Vert E(\tau)\xi\Vert\leq M\quad\text{for every $\Vert\xi\Vert\leq1$ and
$|\tau|\leq1$}.
\]
We then fix $(x,t)\in\mathbb{R}^{N+1}\setminus\{(0,0)\}$ and define%
\[
\lambda=\Vert x\Vert+\sqrt{|t|}\quad\text{and}\quad(\xi,\tau):=\left(
D_{0}\left(  \frac{1}{\lambda}\right)  x,\frac{t}{\lambda^{2}}\right)  ,
\]
Since $\Vert\cdot\Vert$ is $D_{0}$-homogeneous of degree $1$, it is immediate
to recognize that $\Vert\xi\Vert,\,|\tau|\leq1$; thus, by \eqref{LP 2.20} we
get
\[
M\geq\Vert E(\tau)\xi\Vert=\left\Vert E\left(  \frac{t}{\lambda^{2}}\right)
D_{0}\left(  \frac{1}{\lambda}\right)  x\right\Vert =\left\Vert D_{0}\left(
\frac{1}{\lambda}\right)  E(t)x\right\Vert =\frac{1}{\lambda}\Vert
E(t)x\Vert,
\]
so that%
\[
\Vert E(t)x\Vert\leq M\lambda=c\left(  \Vert x\Vert+\sqrt{|t|}\right)  ,
\]
and this gives the desired \eqref{eq:Etxhom} for $(x,t)\neq(0,0)$. Since this
estimate is clearly satisfied when $x=t=0$, the proof is complete.
\end{proof}

\section{Operators with measurable coefficients $a_{ij}\left(  t\right)
$\label{sec operators t}}

\subsection{Known results on the fundamental solution}

Throughout this section, we consider an operator $\mathcal{L}$ of the form
\eqref{L} and satisfying (H1)-(H2), with bounded me\-a\-su\-ra\-ble
coefficients $a_{ij}$ only depending on $t$, that is,
\begin{equation}
\mathcal{L}u=\sum_{i,j=1}^{q}a_{ij}\left(  t\right)  \partial_{x_{i}x_{j}}%
^{2}u+\sum_{k,j=1}^{N}b_{jk}x_{k}\partial_{x_{j}}u-\partial_{t}u,\qquad
(x,t)\in\mathbb{R}^{N+1}. \label{eq:LLsolot}%
\end{equation}
In \cite{BP}, an explicit fundamental solution for $\mathcal{L}$ is computed,
and its properties are studied. The next theorem summarizes some results in
\cite{BP} that we will need.

We point out that, since our assumption (H1) on the matrix $B$ is stronger
than the one made in \cite{BP} (here the model operator with constant $a_{ij}$
is both left invariant and homogeneous, while in \cite{BP} it is only left
invariant), here we specialize the formulas and results to our simpler situation.

\begin{theorem}
[Fundamental solution for operators with $t$-variable coefficients]%
\label{Thm fund sol coeff t dip} Under assumptions \emph{(H1)-(H2)} above, let
$C(t,s)$ be the $N\times N$ matrix defined as
\begin{equation}
\label{eq-EC}C(t,s) =\int_{s}^{t}E(t-\sigma)\cdot%
\begin{pmatrix}
A_{0}(\sigma) & 0\\
0 & 0
\end{pmatrix}
\cdot E(t-\sigma)^{T}\,d\sigma\quad(\text{with $t > s$})
\end{equation}
\emph{(}we recall that $E(\sigma) = \exp(-\sigma B)$, see \eqref{B}\emph{)}.
Then, the matrix $C(t,s)$ is \emph{sym\-me\-tric and positive definite} for
every $t>s$. Moreover, if we define
\begin{equation}%
\begin{split}
&  \Gamma(x,t;y,s)\\
&  \quad= \frac{1}{(4\pi)^{N/2}\sqrt{\det C(t,s)}} e^{-\frac{1}{4}\langle
C(t,s)^{-1}(x-E(t-s)y),\,x-E(t-s)y\rangle} \cdot\mathbf{1}_{\{t > s\}}%
\end{split}
\label{eq.exprGammapernoi}%
\end{equation}
\emph{(}where $\mathbf{1}_{A}$ denotes the indicator function of a set
$A$\emph{)}, then $\Gamma$ enjoys the following properties, so that $\Gamma$
is the \emph{fundamental solution} for $\mathcal{L}$ with pole at $(y,s)$.

\begin{enumerate}
\item In the open set $\mathcal{O}:=\{(x,t;y,s)\in\mathbb{R}^{2N+2}%
:\,(x,t)\neq(y,s)\}$, the function $\Gamma$ is \emph{jointly continuous} in
$(x,t;y,s)$ and $C^{\infty}$ with respect to $x,y$. Mo\-re\-o\-ver, for every
multi-indexes $\alpha,\beta$ the functions
\[
\partial_{x}^{\alpha}\partial_{y}^{\beta}\Gamma=\frac{\partial^{\alpha+\beta
}\Gamma}{\partial x^{\alpha}\partial y^{\beta}}%
\]
are \emph{jointly continuous} in $(x,t;y,s)\in\mathcal{O}$. Finally, $\Gamma$
and $\partial_{x}^{\alpha}\partial_{y}^{\beta}\Gamma$ are \emph{Lip\-schitz
continuous} with respect to $t,s$ in any region $\mathcal{R}$ of the form
\[
\mathcal{R}=\{(x,t;y,s)\in\mathbb{R}^{2N+2}:\,H\leq s+\delta\leq t\leq K\},
\]
where $H,K\in\mathbb{R}$ and $\delta>0$ are arbitrarily fixed.

\item For every fixed $y\in\mathbb{R}^{N}$ and $t>s$, we have
\[
\lim_{|x|\rightarrow+\infty}\Gamma(x,t;y,s)=0.
\]

\item For every fixed $(y,s)\in\mathbb{R}^{N+1}$, we have
\[
\big(\mathcal{L}\Gamma(\cdot;y,s)\big)(x,t)=0\qquad\text{for every
$x\in\mathbb{R}^{N}$ and a.e.\thinspace$t$}.
\]

\item For every fixed $x\in\mathbb{R}^{N}$ and every $t>s$, we have
\begin{equation}
\int_{\mathbb{R}^{N}}\Gamma(x,t;y,s)\,dy=1. \label{eq:integralGamma1}%
\end{equation}

\item For every $f\in C(\mathbb{R}^{N})\cap L^{\infty}(\mathbb{R}^{N})$ and
every $s\in\mathbb{R}$, the function
\[
u(x,t)=\int_{\mathbb{R}^{N}}\Gamma(x,t;y,s)f(y)\,dy
\]
is the unique solution to the Cauchy problem
\begin{equation}%
\begin{cases}
\mathcal{L}u=0 & \text{in $\mathbb{R}^{N}\times(s,\infty)$}\\
u(\cdot,s)=f &
\end{cases}
\label{eq:pbCauchyThmBP}%
\end{equation}
In particular, $u(\cdot,s)\rightarrow f$ uniformly in $\mathbb{R}^{N}$ as
$t\rightarrow s^{+}$.
\end{enumerate}

Finally, the function $\Gamma^{\ast}(x,t;y,s):=\Gamma(y,s;x,t)$ satisfies dual
properties of \emph{(2)-(4)} with respect to the \emph{formal adjoint} of
$\mathcal{L}$, that is,
\[
\textstyle\mathcal{L}^{\ast}=\sum_{i,j=1}^{q}a_{ij}(s)\partial_{y_{i}y_{j}%
}-\sum_{k,j=1}^{N}b_{jk}y_{k}\partial_{y_{i}}+\partial_{s},
\]
and thus $\Gamma^{\ast}$ is the fundamental solution of $\mathcal{L}^{\ast}$.
\end{theorem}

The precise definition of \emph{solution to the Cauchy problem}
\eqref{eq:pbCauchyThmBP} requires some care, see \cite[Definitions 1.2 and
1.3]{BP} for the details.
Let us now further specialize our class of operators to the model operators
with \emph{constant coefficients} $a_{ij}$. Keeping our assumption (H2) on the
matrix $B,$ let%
\begin{equation}
\mathcal{L}_{\alpha}u=\alpha\sum_{i=1}^{q}\partial_{x_{i}x_{i}}^{2}%
u+\sum_{k,j=1}^{N}b_{jk}x_{k}\partial_{x_{j}}u-\partial_{t}u \label{L-alpha}%
\end{equation}
for some $\alpha>0$. Then the results of the above theorem apply in a simpler
form. Actually, the following facts are proved already in \cite{LP}.

\begin{theorem}
[Fundamental solution for operators with constant coefficients]%
\label{Thm fund sol cost coeff} Let $\alpha>0$ be fixed, and let
$\Gamma_{\alpha}$ be the fundamental solution of the operator $\mathcal{L}%
_{\alpha}$ in \eqref{L-alpha}, whose existence is guaranteed by Theorem
\ref{Thm fund sol coeff t dip}. Then:

\begin{enumerate}
\item $\Gamma_{\alpha}$ is a \emph{kernel of convolution type}, that is,
\begin{equation}
\label{eq:Gammaalfaconvolution}%
\begin{split}
\Gamma_{\alpha}(x,t;y,s)  &  =\Gamma_{\alpha}\big(x-E(t-s)y,t-s;0,0\big)\\
&  = \Gamma_{\alpha}\big((y,s)^{-1}\circ(x,t);0,0\big);
\end{split}
\end{equation}

\item The matrix $C(t,s)$ in \eqref{eq-EC} takes the simpler form
\begin{equation}
\label{C_0}C(t,s) =C_{0}(t-s),
\end{equation}
where $C_{0}(\tau)$ is the $N\times N$ matrix defined as
\[
C_{0}(\tau) = \alpha\int_{0}^{\tau}E(t-\sigma)\cdot%
\begin{pmatrix}
I_{q} & 0\\
0 & 0
\end{pmatrix}
\cdot E(t-\sigma)^{T}d\sigma\qquad(\tau> 0).
\]
Furthermore, one has the `homogeneity property'
\begin{equation}
\label{C omogenea}C_{0}(\tau) =D_{0}(\sqrt{\tau})C_{0}(1) D_{0}(\sqrt{\tau
})\qquad\forall\,\,\tau> 0.
\end{equation}

\end{enumerate}

In particular, by combining \eqref{eq.exprGammapernoi} with
\eqref{C_0}-\eqref{C omogenea}, we can write
\begin{equation}
\label{eq.exprGammaalfa}%
\begin{split}
&  \Gamma_{\alpha}(x,t;0,0) =\frac{1}{( 4\pi\alpha)^{N/2} \sqrt{\det C_{0}%
(t)}}e^{-\frac{1}{4\alpha}x^{T}C_{0}(t)^{-1}x}\\
&  \qquad=\frac{1}{(4\pi\alpha)^{N/2}t^{Q/2}\sqrt{\det C_{0}(1)}} e^{-\frac
{1}{4\alpha}\langle C_{0}(1)^{-1} \big(D_{0}\big(\frac{1}{\sqrt{t}%
}\big)x\big),\,D_{0}\big(\frac{1}{\sqrt{t}}\big)x\rangle}.
\end{split}
\end{equation}

\end{theorem}

In \cite[Thm.1.7]{BP}, the next useful comparison result is proved.

\begin{theorem}
\label{Thm comparison Gammas} Let $\Gamma$ be as in Theorem
\ref{Thm fund sol coeff t dip}, and let $\nu>0$ be as in \eqref{nu}. Then, for
every $s,t\in\mathbb{R}$ with $s<t$, one has the following estimate
\begin{equation}
\nu\,C_{0}(t-s)^{-1}\leq C(t,s)^{-1}\leq\nu^{-1}C_{0}(t-s)^{-1},
\label{eq:estimCCzero}%
\end{equation}
\emph{in the sense of quadratic forms in $\mathbb{R}^{N}$}. As a consequence,
we obtain
\begin{equation}
\nu^{N}\Gamma_{\nu}(x,t;y,s)\leq\Gamma(x,t;y,s)\leq\frac{1}{\nu^{N}}%
\Gamma_{\nu^{-1}}(x,t;y,s), \label{G G_nu}%
\end{equation}
where $\Gamma_{\nu}$ is the fundamental solution of the operator
$\mathcal{L}_{\nu}$ in \eqref{L-alpha}.
\end{theorem}

\subsection{Sharp estimates on the fundamental
solution\label{sec sharp estimates fund sol}}

Taking into account all the results recalled so far, we now aim at proving
\emph{sharp Gaussian estimates} for the space de\-ri\-va\-ti\-ves of the
fundamental solution $\Gamma$ of the operator $\mathcal{L}$. As we shall see,
these estimates will play a key r\^{o}le in our argument.

In order to clearly state our results, we first introduce an \emph{ad-hoc}
multi-index notation which shall be useful to deal with differential operators
acting \emph{on the $2N$ variables $x,y\in\mathbb{R}^{N}$}. For a multi-index
\[
\bd\ell=(\ell_{1},\ldots,\ell_{2N})\in\N^{2N},
\]
let
\[
D_{(x,y)}^{\bd\ell}\,f(x,y):=(\partial_{x_{1}})^{\ell_{1}}\cdots
(\partial_{x_{N}})^{\ell_{N}}(\partial_{y_{1}})^{\ell_{N+1}}\cdots
(\partial_{y_{N}})^{\ell_{2N}}f(x,y).
\]
Moreover, setting $\upsilon=(q_{1},\ldots,q_{N},q_{1},\ldots,q_{N}%
)\in\mathbb{R}^{2N}$ (where the $q_{i}$'s are the exponents appearing in the
dilation $D_{0}(\lambda)$, see \eqref{dilations}), we define
\[
\textstyle|\bd\ell|:=\sum_{i=1}^{2N}\ell_{i}\qquad\text{and}\qquad
\omega(\bd\ell):=\textstyle\sum_{i=1}^{2N}\upsilon_{i}\ell_{i}.
\]
We will refer to $|\bd\ell|$ and $\omega(\bd\ell)$ as, respectively, the
\emph{length} and the \emph{order} of $\bd\ell$.

\begin{remark}
\label{rem:ordersoloN} Throughout the rest of the paper, we will sometimes
need to give a meaning to $\omega(\bd\alpha)$ when $\bd\alpha$ is a
multi-index in $\N^{N}$, that is, $\bd\alpha= (\alpha_{1},\ldots,\alpha_{N})$.
By analogy, if this is the case we agree to define
\[
\textstyle
\omega(\bd\alpha) := \omega(\bd\alpha^{\prime}= (\bd\alpha,\bd 0)) = \sum_{i =
1}^{N}\alpha_{i}q_{i}.
\]

\end{remark}

Using the notion of \emph{length}, we can introduce an \emph{order relation}
between multi-in\-dex\-es: if $\bd\ell= (\ell_{1},\ldots,\ell_{2N}),\,
\bd\kappa= (\kappa_{1},\ldots,\kappa_{2N})\in\N^{2N}$, we say that
\[
\bd\ell\prec\bd\kappa
\]
if one of the following conditions is satisfied:

\begin{itemize}
\item[(i)] $|\bd\ell| < |\bd\kappa|$;

\item[(ii)] $|\bd\ell| = |\bd\kappa|$ and $\ell_{1} < \kappa_{1}$;

\item[(iii)] $|\bd\ell| = |\bd\kappa|$ and there exists $1\leq i\leq2N-1$ such
that
\[
\text{$\ell_{1} = \kappa_{1},\ldots,\ell_{i} = \kappa_{i}$\quad and\quad
$\ell_{i+1}<\kappa_{j+1}$}.
\]

\end{itemize}

After all these preliminaries, we can state our first main result.

\begin{theorem}
\label{Thm bound derivatives} Let $\Gamma$ be as in Theorem
\ref{Thm fund sol coeff t dip}, and let $\nu>0$ be as in \eqref{nu}. Moreover,
let $\bd\alpha=(\boldsymbol{\alpha}_{1},\boldsymbol{\alpha}_{2})\in\N^{2N}$ be
a fixed multi-index. Then, there exist $c=c(\nu,\bd\alpha)>0$ and a constant
$c_{1}>0$, independent of $\nu$ and $\bd\alpha$, such that
\begin{equation}%
\begin{split}
\left\vert D_{(x,y)}^{\bd\alpha}\,\Gamma(\xi;\eta)\right\vert  &  =\left\vert
D_{x}^{\bd\alpha_{1}}D_{y}^{\bd\alpha_{2}}\Gamma(\xi;\eta)\right\vert \\
&  \leq\frac{{c}}{(t-s)^{\omega(\bd\alpha)/2}}\,\Gamma_{c_{1}\nu^{-1}}%
(\xi;\eta)\\
&  \leq\frac{c}{d(\xi,\eta)^{Q+\omega(\bd\alpha)}}%
\end{split}
\label{eq:mainestim}%
\end{equation}
for every $\xi,\eta\in\mathbb{R}^{N}$ with $t\neq s$. The resulting inequality%
\[
\left\vert D_{\left(  x,y\right)  }^{\bd\alpha}\Gamma(\xi;\eta)\right\vert
\leq\frac{c}{d(\xi,\eta)^{Q+\omega(\bd\alpha)}}%
\]
actually holds for every $\xi,\eta\in\mathbb{R}^{N}$ with $\xi\neq\eta$.
\end{theorem}

\begin{remark}
\label{rem:regulGammaMV} Let $(y,s)\in\mathbb{R}^{N+1}$ be fixed. Since we
know from Theorem \ref{Thm fund sol coeff t dip}-(3) that $(\mathcal{L}%
\Gamma(\cdot;y,s))(x,t)=0$ for every $x\in\mathbb{R}^{N}$ and a.e.\thinspace
$t$, we can express
\[
\partial_{t}(D_{(x,y)}^{\bd\alpha}\,\Gamma)\qquad(\text{for every
$\bd\alpha\in\N^{2N}$})
\]
as a combination of quantities that, by \eqref{eq:mainestim} and the
exponential decay of the right-hand of this inequality as $t\rightarrow s^{+}$
(with $x\neq y$), are locally essentially bounded in $\mathbb{R}%
^{N+1}\setminus\{(y,s)\}$. In particular, the same is true of $Y(D_{(x,y)}%
^{\bd\alpha}\,\Gamma)$.
\end{remark}

Before proving Theorem \ref{Thm bound derivatives}, we establish the following
technical lemma.

\begin{lemma}
\label{lem:formeqdis} Let $A=(a_{ij})_{i,j = 1}^{N}$ and $B=(b_{ij})_{i,j =
1}^{N}$ be two $N\times N$ \emph{symmetric and positive definite matrices}
such that, in the sense of quadratic forms, one has%
\begin{equation}
A\leq cB\qquad\text{ for some $c>0$}. \label{eq:hpAleqB}%
\end{equation}
Then, denoting by $\Vert\cdot\Vert$ the maximum norm of a matrix, we have
\begin{equation}
\Vert A\Vert\leq2c\Vert B\Vert. \label{eq:disnormAnormB}%
\end{equation}
In particular, if $G$ is any $N\times N$ matrix with real coefficients, then
\begin{equation}
\Vert G\,A\,G^{T}\Vert\leq2c\Vert G\,B\,G^{T}\Vert. \label{eq:normawithG}%
\end{equation}

\end{lemma}

\begin{proof}
First of all, since \eqref{eq:hpAleqB} holds in the sense of quadratic forms,
we have
\begin{equation}
\langle B\xi,\xi\rangle\leq c\,\langle A\xi,\xi\rangle\qquad\forall\,\,\xi
\in\mathbb{R}^{N}. \label{eq:disugQFLemma}%
\end{equation}
As a consequence, choosing $\xi=e_{i}$ (for $i=1,\ldots,N$) and reminding that
both $A$ and $B$ are \emph{positive definite}, we readily have
\begin{equation}
0<a_{ii}\leq cb_{ii}\leq c\max_{h,k}|b_{hk}|=c\Vert B\Vert.
\label{eq:diseqaiiLemma}%
\end{equation}
On the other hand, choosing $\xi=e_{i}\pm e_{j}$ in \eqref{eq:disugQFLemma}
(with $i\neq j$), we get
\[
a_{ii}+a_{jj}\pm2a_{ij}\leq c(b_{ii}+b_{jj}\pm2b_{ij})\leq4c\Vert B\Vert;
\]
from this, since $a_{ii},\,a_{jj}>0$, we derive
\begin{equation}
|a_{ij}|\leq2c\Vert B\Vert. \label{eq:diseqaijLemma}%
\end{equation}
Gathering \eqref{eq:diseqaiiLemma}-\eqref{eq:diseqaijLemma}, we immediately
obtain \eqref{eq:disnormAnormB}. To prove \eqref{eq:normawithG} we observe
that, if $G$ is any $N\times N$ matrix, from \eqref{eq:hpAleqB} it easily
follows that
\[
GAG^{T}\leq c\,GBG^{T};
\]
hence, the desired \eqref{eq:normawithG} is an immediate consequence of \eqref{eq:disnormAnormB}.
\end{proof}

Using Lemma \ref{lem:formeqdis}, we can now give the proof of Theorem
\ref{Thm bound derivatives}.

\begin{proof}
[Proof (of Theorem \ref{Thm bound derivatives})]We first observe that, if
$\bd{\alpha}=\bd{0}$, estimate \eqref{eq:mainestim} is al\-re\-a\-dy contained
in Theorem \ref{Thm comparison Gammas}; hence, we can assume in what follows
that
\[
\bd{\alpha}\neq\bd{0}.
\]
We now fix \emph{once and for all} $s,t\in\mathbb{R}$ satisfying $s<t$ and we
notice that, by using the explicit expression of $\Gamma$ given in
\eqref{eq.exprGammapernoi}, we can write
\begin{equation}
\Gamma(x,t;y,s)=(f_{t,s}\circ p_{t,s})(x,y)\qquad\forall\,\,x,y\in
\mathbb{R}^{N}, \label{eq:Gammacompfp}%
\end{equation}
where the functions $f_{t,s}$ and $p_{t,s}$ are given, respectively, by
\begin{align*}
&  f_{t,s}(z)=\frac{1}{(4\pi)^{N/2}\sqrt{\det C(t,s)}}e^{z}\qquad\text{and}\\
&  p_{t,s}(x,y)=-\frac{1}{4}\langle C(t,s)^{-1}(x-E(t-s)y),x-E(t-s)y\rangle.
\end{align*}
Starting from \eqref{eq:Gammacompfp}, and exploiting the multivariate version
of the Fa\`{a} di Bruno formula established in \cite[formula (2.1)]{CoSa}, we
obtain
\begin{equation}%
\begin{split}
&  D_{(x,y)}^{\bd\alpha}\,\Gamma(x,t;y,s)=D_{(x,y)}^{\bd\alpha}(f_{t,s}\circ
p_{t,s})(x,y)\\
&  \qquad=\Gamma(x,t;y,s)\cdot\sum_{\lambda=1}^{r}\sum_{m=1}^{r}\sum
_{p_{m}(\lambda,\bd\alpha)}\prod_{i=1}^{m}\frac{\bd\alpha!}{k_{i}%
!\,(\bd\ell_{i}!)^{k_{i}}}\,\big[D_{(x,y)}^{\bd\ell_{i}}\,p_{t,s}%
(x,y)\big]^{k_{i}},
\end{split}
\label{eq:DGammafirst}%
\end{equation}
where $r:=|\bd\alpha|\geq1$ and
\begin{align*}
p_{m}(\lambda,\bd\alpha)  &  =\big\{(k_{1},\ldots,k_{m};\bd\ell_{1}%
,\ldots,\bd\ell_{m})\in\N^{m}\times(\N^{2N})^{m}:\,k_{i}>0,\\[0.05cm]
&  \qquad\,\,\bd0\prec\bd\ell_{1}\cdots\prec\bd\ell_{m}\,\,\text{and}%
\,\,\textstyle\sum_{i=1}^{m}k_{i}=\lambda,\,\,\sum_{i=1}^{m}k_{i}\bd\ell
_{i}=\bd\alpha\big\}.
\end{align*}
We now observe that, since the function $p_{t,s}$ is a homogeneous polynomial
of degree $2$ in the variables $x,y$, one obviously has
\[
D_{(x,y)}^{\bd\ell}\,p_{t,s}\equiv0\qquad\text{$\forall\,\,\bd\ell\in\N^{2N}$
with $|\bd\ell|\geq3$};
\]
hence, formula \eqref{eq:DGammafirst} can be rewritten as follows
\begin{equation}%
\begin{split}
&  D_{(x,y)}^{\bd\alpha}\,\Gamma(x,t;y,s)\\
&  \qquad=\Gamma(x,t;y,s)\cdot\sum_{(\lambda,m)\in\mathcal{S}}\,\,\sum
_{p_{m}(\lambda,\bd\alpha)}\prod_{i=1}^{m}\frac{\bd\alpha!}{k_{i}%
!\,(\bd\ell_{i}!)^{k_{i}}}\,\big[D_{(x,y)}^{\bd\ell_{i}}\,p_{t,s}%
(x,y)\big]^{k_{i}},
\end{split}
\label{eq:DGammasec}%
\end{equation}
where $\mathcal{S}$ is the subset of $\{1,\ldots,r\}\times\{1,\ldots,r\}$
defined as
\[
\mathcal{S}:=\big\{(\lambda,m):\,\text{$|\bd\ell_{i}|\leq2$ for all
$(k_{1},\ldots,k_{m};\bd\ell_{1},\ldots,\bd\ell_{m})\in p_{m}(\lambda,\alpha
)$}\big\}.
\]
Then, by combining formula \eqref{eq:DGammasec} with the \emph{global
pointwise estimates} for $\Gamma$ con\-tai\-ned in Theorem
\ref{Thm comparison Gammas}, for every $x,y\in\mathbb{R}^{n}$ we obtain
\begin{equation}%
\begin{split}
&  |D_{(x,y)}^{\bd\alpha}\,\Gamma(x,t;y,s)|\\
&  \qquad\leq c\,\Gamma_{\nu^{-1}}(x,t;y,s)\cdot\sum_{(\lambda,m)\in
\mathcal{S}}\,\sum_{p_{m}(\lambda,\bd\alpha)}\prod_{i=1}^{m}\big|D_{(x,y)}%
^{\bd\ell_{i}}\,p_{t,s}(x,y)\big|^{k_{i}},
\end{split}
\label{eq:DGammater}%
\end{equation}
where $c>0$ is a constant only depending on $\bd\alpha$ and $\nu$. On account
of \eqref{eq:DGammater}, in order to prove \eqref{eq:mainestim} we need to
provide precise estimates for
\[
|D_{(x,y)}^{\bd\ell}\,p_{t,s}(x,y)|\qquad(\text{when $0<|\bd\ell|\leq2$}).
\]
To this end, we distinguish some different cases. In what follows, we denote
by the same $c$ any positive constant which depends only on $\nu$ and
$\bd\alpha$. \medskip

\noindent\textbf{Case I: $\bd\ell=(\bd e_{i},\bd0)$}. In this case, a direct
computation gives
\[
D_{(x,y)}^{\bd\ell}\,p_{t,s}(x,y)=\partial_{x_{i}}p_{t,s}(x,y)=-\frac{1}%
{2}\big[C(t,s)^{-1}(x-E(t-s)y)\big]_{i};
\]
hence, setting $v:=x-E(t-s)y$ and reminding that
\[
\big[D_{0}(\lambda)v\big]_{i}=\lambda^{q_{i}}v_{i},
\]
we obtain the following chain of inequalities:
\begin{align*}
|D_{(x,y)}^{\bd\ell}\,p_{t,s}(x,y)|  &  =\frac{1}{2}\big|\big[C(t,s)^{-1}%
v\big]_{i}\big|=\frac{c}{(t-s)^{q_{i}/2}}\big|\big[D_{0}(\sqrt{t-s}%
)C(t,s)^{-1}v\big]_{i}\big|\\
&  (\text{setting $M(t,s):=D_{0}(\sqrt{t-s})C(t,s)^{-1}D_{0}(\sqrt{t-s})$})\\
&  =\frac{c}{(t-s)^{q_{i}/2}}\Big|\Big[M(t,s)\cdot D_{0}\Big(\frac{1}%
{\sqrt{t-s}}\Big)v\Big]_{i}\Big|\\
&  \leq\frac{c}{(t-s)^{q_{i}/2}}\Vert M(t,s)\Vert\cdot\Big|D_{0}\Big(\frac
{1}{\sqrt{t-s}}\Big)v\Big|=:(\bigstar).
\end{align*}
Now, by combining \eqref{eq:estimCCzero} with Lemma \ref{lem:formeqdis}, we
readily infer that
\begin{align*}
\Vert M(t,s)\Vert &  \leq2\nu^{-1}\Vert D_{0}(\sqrt{t-s})C_{0}(t-s)^{-1}%
D_{0}(\sqrt{t-s})\Vert\\
&  (\text{see identity \eqref{C omogenea}})\\
&  =2\nu^{-1}\Vert C_{0}(1)^{-1}\Vert;
\end{align*}
as a consequence, we obtain
\[
(\bigstar)\leq\frac{c}{(t-s)^{q_{i}/2}}\,\Big|D_{0}\Big(\frac{1}{\sqrt{t-s}%
}\Big)(x-E(t-s)y)\Big|.
\]
In particular, since $q_{i}=\omega(\bd\ell)$, we conclude that
\begin{equation}
|D_{(x,y)}^{\bd\ell}\,p_{t,s}(x,y)|\leq\frac{c}{(t-s)^{\omega(\bd\ell)/2}%
}\,\Big|D_{0}\Big(\frac{1}{\sqrt{t-s}}\Big)(x-E(t-s)y)\Big|.
\label{eq:mainestimDexi}%
\end{equation}
\medskip

\noindent\textbf{Case II: $\bd\ell=(\bd0,\bd e_{i})$}. In this case, we first
rewrite $p_{t,s}$ as follows:
\begin{equation}%
\begin{split}
p_{t,s}(x,y)  &  =-\frac{1}{4}\langle C(t,s)^{-1}%
E(t-s)(y-E(s-t)x),E(t-s)(y-E(s-t)x)\rangle\\
&  (\text{setting $\widehat{C}(t,s)=E(t-s)^{T}C(t,s)^{-1}E(t-s)$})\\
&  =-\frac{1}{4}\langle\widehat{C}(t,s)(y-E(s-t)x),y-E(s-t)x\rangle;
\end{split}
\label{eq:exprptdery}%
\end{equation}
hence, by proceeding {exactly} as in Case I, we get
\begin{align*}
|D_{(x,y)}^{\bd\ell}\,p_{t,s}(x,y)|  &  =|\partial_{y_{i}}p_{t,s}(x,y)|\\
&  \leq\frac{c}{(t-s)^{q_{i}/2}}\Vert\widehat{M}(t,s)\Vert\cdot\Big|D_{0}%
\Big(\frac{1}{\sqrt{t-s}}\Big)w\Big|=:(\bigstar),
\end{align*}
where $w:=y-E(s-t)x$ and
\[
\widehat{M}(t,s):=D_{0}(\sqrt{t-s})\widehat{C}(t,s)D_{0}(\sqrt{t-s}).
\]
Now, using again \eqref{eq:estimCCzero} and Lemma \ref{lem:formeqdis}, we get
\begin{align*}
\Vert\widehat{M}(t,s)\Vert &  =\big\|\big[(D_{0}(\sqrt{t-s})E(t-s)^{T}%
\big]\,C(t,s)^{-1}\big[E(t-s)D_{0}(\sqrt{t-s})\big]\big\|\\
&  \leq2\nu^{-1}\big\|\big[(D_{0}(\sqrt{t-s})E(t-s)^{T}\big]\,C_{0}%
(t-s)^{-1}\big[E(t-s)D_{0}(\sqrt{t-s})\big]\big\|\\
&  (\text{see identities \eqref{LP 2.20} and \eqref{C omogenea}})\\
&  =2\nu^{-1}\Vert E(1)^{T}C_{0}(1)^{-1}E(1)\Vert;
\end{align*}
as a consequence, we obtain
\begin{align*}
(\bigstar)  &  \leq\frac{c}{(t-s)^{q_{i}/2}}\,\Big|D_{0}\Big(\frac{1}%
{\sqrt{t-s}}\Big)(y-E(s-t)x)\Big|\\
&  =\frac{c}{(t-s)^{q_{i}/2}}\,\Big|D_{0}\Big(\frac{1}{\sqrt{t-s}%
}\Big)E(s-t)\cdot(x-E(t-s)y)\Big|\\[0.15cm]
&  (\text{again by \eqref{LP 2.20}})\\
&  =\frac{c}{(t-s)^{q_{i}/2}}\,\Big|E(-1)D_{0}\Big(\frac{1}{\sqrt{t-s}%
}\Big)\cdot(x-E(t-s)y)\Big|\\
&  \leq\frac{c}{(t-s)^{q_{i}/2}}\Big|D_{0}\Big(\frac{1}{\sqrt{t-s}}%
\Big)\cdot(x-E(t-s)y)\Big|
\end{align*}
In particular, since $q_{i}=\omega(\bd\ell)$, we conclude that
\begin{equation}
|D_{(x,y)}^{\bd\ell}\,p_{t,s}(x,y)|\leq\frac{c}{(t-s)^{\omega(\bd\ell)/2}%
}\,\Big|D_{0}\Big(\frac{1}{\sqrt{t-s}}\Big)(x-E(t-s)y)\Big|.
\label{eq:mainestimDeyi}%
\end{equation}
\medskip

\noindent\textbf{Case III: $\bd\ell=(\bd e_{i}+\bd e_{j},\bd0)$}. In this
case, a direct computation gives
\[
D_{(x,y)}^{\bd\ell}\,p_{t,s}(x,y)=\partial_{x_{i}x_{j}}^{2}p_{t,s}%
(x,y)=-\frac{1}{2}C(t,s)_{ij}^{-1};
\]
hence, setting $C(t,s)^{-1}:=\big(\gamma_{hk}(t,s)\big)_{h,k=1}^{N}$, we get
\begin{equation}
|D_{(x,y)}^{\bd\ell}\,p_{t,s}(x,y)|\leq c\,|\gamma_{ij}(t,s)|.
\label{eq:estimDellp}%
\end{equation}
Now, taking into account \eqref{eq:estimCCzero}, for every $\varepsilon>0$ we
have
\begin{align*}
&  \gamma_{ii}(t,s)+\varepsilon^{2}\gamma_{jj}(t,s)\pm2\varepsilon\gamma
_{ij}(t,s)=\langle C(t,s)^{-1}(e_{i}\pm\varepsilon e_{j}),e_{i}\pm\varepsilon
e_{j}\rangle\\
&  \qquad\leq\nu^{-1}\big(\theta_{ii}(t-s)+\varepsilon^{2}\theta_{jj}%
(t-s)\pm2\varepsilon\theta_{ij}(t-s)\big),
\end{align*}
where we have used the notation
\[
C_{0}(\tau)^{-1}=\big(\theta_{hk}(\tau)\big)_{h,k=1}^{N}.
\]
From this, since $C(t,s)^{-1}$ and $C_{0}(t-s)^{-1}$ are positive definite, we
obtain
\begin{equation}
|\gamma_{ij}(t,s)|\leq\frac{1}{2\nu}\Big(\frac{1}{\varepsilon}\theta
_{ii}(t-s)+\varepsilon\theta_{jj}(t-s)+2|\theta_{ij}(t-s)|\Big).
\label{eq:estigammaij}%
\end{equation}
To estimate the rhs of \eqref{eq:estigammaij} we remind that, by
\eqref{C omogenea}, one has
\[
C_{0}(t-s)^{-1}=D_{0}\Big(\frac{1}{\sqrt{t-s}}\Big)C_{0}(1)^{-1}%
D_{0}\Big(\frac{1}{\sqrt{t-s}}\Big);
\]
as a consequence, we obtain
\begin{equation}
|\theta_{hk}(t-s)|=\frac{\theta_{hk}(1)}{(t-s)^{(q_{h}+q_{k})/2}}\qquad
\forall\,\,1\leq h,k\leq N. \label{eq:exprgammazeroij}%
\end{equation}
Gathering \eqref{eq:estigammaij}-\eqref{eq:exprgammazeroij}, and choosing
$\varepsilon:=(t-s)^{(q_{j}-q_{i})/2}$, we then derive
\begin{equation}
|\gamma_{ij}(t,s)|\leq\frac{c}{(t-s)^{(q_{i}+q_{j})/2}}.
\label{eq:estimgammaijlast}%
\end{equation}
Finally, since $q_{i}+q_{j}=\omega(\bd\ell)$, from \eqref{eq:estimDellp} and
\eqref{eq:estimgammaijlast} we conclude that
\begin{equation}
|D_{(x,y)}^{\bd\ell}\,p_{t,s}(x,y)|\leq\frac{c}{(t-s)^{\omega(\bd\ell)/2}}.
\label{eq:mainestimDexixj}%
\end{equation}
\medskip

\noindent\textbf{Case IV:} $\bd\ell=(\bd0,\bd e_{i}+\bd e_{j})$. In this case,
using the expression of $p_{t,s}$ given in \eqref{eq:exprptdery} (where
$\widehat{C}(t,s)=E(t-s)^{T}C(t,s)^{-1}E(t-s)$), we readily infer that
\[
D_{(x,y)}^{\bd\ell}\,p_{t,s}(x,y)=\partial_{y_{i}y_{j}}^{2}p_{t,s}%
(x,y)=-\frac{1}{2}\widehat{C}(t,s)_{ij};
\]
hence, setting $\widehat{C}(t,s):=\big(\widehat{\gamma}_{hk}(t,s)\big)_{h,k=1}%
^{N},$ we get
\begin{equation}
|D_{(x,y)}^{\bd\ell}\,p_{t,s}(x,y)|\leq c\,|\widehat{\gamma}_{ij}(t,s)|.
\label{eq:estimDellpy}%
\end{equation}
Now, taking into account \eqref{eq:estimCCzero}, it is easy to see that
\begin{align*}
\widehat{C}(t,s)  &  =E(t-s)^{T}C(t,s)^{-1}E(t-s)\\
&  \leq\nu^{-1}E(t-s)^{T}C_{0}(t-s)^{-1}E(t-s)\equiv\nu^{-1}\widehat{C}%
_{0}(t-s);
\end{align*}
from this, by arguing \emph{exactly} as in Case III, for every $\varepsilon>0$
we obtain
\begin{equation}
|\widehat{\gamma}_{ij}(t,s)|\leq\frac{1}{2\nu}\Big(\frac{1}{\varepsilon
}\widehat{\theta}_{ii}(t-s)+\varepsilon\widehat{\theta}_{jj}%
(t-s)+2|\widehat{\theta}_{ij}(t-s)|\Big), \label{eq:estigammaprimeij}%
\end{equation}
where we have used the notation
\[
\widehat{C}_{0}(\tau)=E(\tau)^{T}C_{0}(\tau)^{-1}E(\tau)=\big(\widehat{\theta
}_{hk}(\tau)\big)_{h,k=1}^{N}.
\]
In order to estimate the rhs of \eqref{eq:estigammaprimeij}, we observe that
\begin{align*}
\widehat{C}_{0}(t-s)  &  =E(t-s)^{T}C_{0}(t-s)^{-1}E(t-s)\\
&  (\text{see \eqref{C omogenea}})\\
&  =E(t-s)^{T}\Big[D_{0}\Big(\frac{1}{\sqrt{t-s}}\Big)C_{0}(1)^{-1}%
D_{0}\Big(\frac{1}{\sqrt{t-s}}\Big)\Big]E(t-s)\\
&  (\text{see \eqref{LP 2.20}})\\
&  =D_{0}\Big(\frac{1}{\sqrt{t-s}}\Big)\widehat{C}_{0}(1)D_{0}\Big(\frac
{1}{\sqrt{t-s}}\Big);
\end{align*}
as a consequence, we obtain
\begin{equation}
|\widehat{\theta}_{hk}(t-s)|=\frac{\widehat{\theta}_{hk}(1)}{(t-s)^{(q_{h}%
+q_{k})/2}}\qquad\forall\,\,1\leq h,k\leq N. \label{eq:exprthetazeroij}%
\end{equation}
Gathering \eqref{eq:estigammaprimeij}-\eqref{eq:exprthetazeroij}, and choosing
$\varepsilon:=(t-s)^{(q_{j}-q_{i})/2}$, we then derive
\begin{equation}
|\widehat{\gamma}_{ij}(t,s)|\leq\frac{c}{(t-s)^{(q_{i}+q_{j})/2}}.
\label{eq:estimgammaprimeijlast}%
\end{equation}
Finally, since $q_{i}+q_{j}=\omega(\bd\ell)$, from \eqref{eq:estimDellpy} and
\eqref{eq:estimgammaprimeijlast} we conclude that
\begin{equation}
|D_{(x,y)}^{\bd\ell}\,p_{t,s}(x,y)|\leq\frac{c}{(t-s)^{\omega(\bd\ell)/2}}.
\label{eq:mainestimDeyiyj}%
\end{equation}
\medskip

\noindent\textbf{Case V:} $\bd\ell=(\bd e_{i},\bd e_{j})$. In this last case,
a direct computation gives
\begin{equation}%
\begin{split}
D_{(x,y)}^{\bd\ell}\,p_{t,s}(x,y)  &  =\partial_{x_{i}y_{j}}^{2}%
p_{t,s}(x,y)=\partial_{y_{j}}\big(\partial_{x_{i}}p_{t,s}\big)(x,y)\\
&  =\partial_{y_{j}}\Big(\frac{1}{2}\big[C(t,s)^{-1}(x-E(t-s)y)\big]_{i}%
\Big)\\
&  =\frac{1}{2}\big[C(t,s)^{-1}E(t-s)\big]_{ij}\\
&  =\frac{1}{2}\sum_{k=1}^{n}\gamma_{ik}(t,s)\,e_{kj}(t-s),
\end{split}
\end{equation}
where we have used the notation
\[
C(t,s)=\big(\gamma_{hk}(t,s)\big)_{h,k=1}^{N}\qquad\text{and}\qquad
E(\tau)=\big(e_{hk}(\tau)\big)_{h,k=1}^{N}.
\]
We now observe that, on account of \eqref{LP 2.20}, we have
\begin{equation}%
\begin{split}
e_{hk}(t-s)  &  =\big[E(t-s)\big]_{hk}=\Big[D_{0}(\sqrt{t-s})E(1)D_{0}%
\Big(\frac{1}{\sqrt{t-s}}\Big)\Big]_{hk}\\
&  =(t-s)^{(q_{h}-q_{k})/2}e_{hk}(1)\qquad\forall\,\,1\leq h,k\leq N;
\end{split}
\label{eq:estimehk}%
\end{equation}
thus, by combining \eqref{eq:estimehk} with \eqref{eq:estimgammaijlast}, we
get
\begin{align*}
&  \Big|\sum_{k=1}^{n}\gamma_{ik}(t,s)\,e_{kj}(t-s)\Big|\leq\sum_{k=1}%
^{n}|\gamma_{ik}(t,s)|\,|e_{kj}(t-s)|\\
&  \qquad\leq c\sum_{k=1}^{n}\frac{1}{(t-s)^{(q_{i}+q_{k})/2}}\cdot
(t-s)^{(q_{k}-q_{j})/2}e_{kj}(1)\\
&  \qquad\leq\frac{c}{(t-s)^{(q_{i}+q_{j})/2}}.
\end{align*}
From this, since $q_{i}+q_{j}=\omega(\bd\ell)$, we immediately conclude that
\begin{equation}
|D_{(x,y)}^{\bd\ell}\,p_{t,s}(x,y)|\leq\frac{1}{2}\Big|\sum_{k=1}^{n}%
\gamma_{ik}(t,s)\,e_{kj}(t-s)\Big|\leq\frac{c}{(t-s)^{\omega(\bd\ell)/2}}.
\label{eq:mainestimDxiyj}%
\end{equation}
Now we have estimated all the non-vanishing derivatives of $p_{t,s}$ (with
respect to both $x$ and $y$), we are ready to complete the proof. Namely, by
combining estimate \eqref{eq:DGammater} with \eqref{eq:mainestimDexi},
\eqref{eq:mainestimDeyi}, \eqref{eq:mainestimDexixj},
\eqref{eq:mainestimDeyiyj} and \eqref{eq:mainestimDxiyj}, we get
\begin{align*}
&  |D_{(x,y)}^{\bd\alpha}\,\Gamma(x,t;y,s)|\\
&  \qquad\leq c\,\Gamma_{\nu^{-1}}(x,t;y,s)\sum_{(\lambda,m)\in\mathcal{S}%
}\,\sum_{p_{m}(\lambda,\bd\alpha)}\prod_{i=1}^{m}\frac{1}{(t-s)^{k_{i}%
\omega(\bd\ell_{i})/2}}\,|v|^{k_{i}(2-|\bd\ell_{i}|)}\\
&  \qquad\leq c\,\Gamma(x,t;y,s)\sum_{(\lambda,m)\in\mathcal{S}}\,\sum
_{p_{m}(\lambda,\bd\alpha)}\frac{1}{(t-s)^{\sum_{i=1}^{m}k_{i}\omega
(\bd\ell_{i})/2}}\,|v|^{\sum_{i=1}^{m}(2k_{i}-k_{i}|\bd\ell_{i}|)},
\end{align*}
where we have used the simplified notation
\[
v:=D_{0}\Big(\frac{1}{\sqrt{t-s}}\Big)(x-E(t-s)y).
\]
On the other hand, owing to the very definition of $p_{m}(\lambda,\bd\alpha)$,
we have
\begin{align*}
&  \mathrm{(a)}\,\,\sum_{i=1}^{m}k_{i}\omega(\bd\ell_{i})/2=\frac{1}%
{2}\,\omega\Big(\sum_{i=1}^{m}k_{i}\bd\ell_{i}\Big)=\omega(\bd\alpha)/2;\\
&  \mathrm{(b)}\,\,\sum_{i=1}^{m}(2k_{i}-k_{i}|\bd\ell_{i}|)=2\sum_{i=1}%
^{m}k_{i}-\Big|\sum_{i=1}^{m}k_{i}\bd\ell_{i}\Big|=2\lambda-|\bd\alpha|.
\end{align*}
As a consequence, we obtain
\begin{equation}%
\begin{split}
&  |D_{(x,y)}^{\bd\alpha}\,\Gamma(x,t;y,s)|\leq\frac{c}{(t-s)^{\omega
(\bd\alpha)/2}}\times\\
&  \qquad\times\Gamma_{\nu^{-1}}(x,t;y,s)\sum_{(\lambda,m)\in\mathcal{S}%
}\Big|D_{0}\Big(\frac{1}{\sqrt{t-s}}\Big)(x-E(t-s)y)\Big|^{2\lambda
-|\bd\alpha|}.
\end{split}
\label{eq:estimDGammalast}%
\end{equation}
We explicitly stress that, if $(\lambda,m)\in\mathcal{S}$, one has
$2\lambda-|\bd\alpha|\geq0$. In fact, taking into account the very definition
of $\mathcal{S}$, we know that
\[
\text{$k_{i}>0$ and $0<|\bd\ell_{i}|\leq2$}\quad\forall\,\,(k_{1},\ldots
,k_{m};\bd\ell_{1},\ldots,\bd\ell_{m})\in p_{m}(\lambda,\bd\alpha);
\]
this, together with identity (b), immediately implies that $2\lambda
-|\bd\alpha|\geq0$. \vspace*{0.1cm}

Now, using the explicit expression of $\Gamma_{\rho}$ given in Theorem
\ref{Thm fund sol cost coeff}, together with the fact that the matrix
$C_{0}(1)^{-1}$ is positive definite, we easily see that
\begin{equation}
\Gamma_{\nu^{-1}}(x,t;y,s)\cdot\Big|D_{0}\Big(\frac{1}{\sqrt{t-s}%
}\Big)(x-E(t-s)y)\Big|^{2\lambda-|\bd\alpha|}\leq c\,\Gamma_{c_{1}\nu^{-1}%
}(x,t;y,s), \label{eq:rimangio}%
\end{equation}
where $c_{1}>0$ is an absolute constant independent of $\nu$ and $\bd\alpha$.
Then, by gathering \eqref{eq:estimDGammalast} and \eqref{eq:rimangio}, we
obtain the first inequality in \eqref{eq:mainestim}.

To prove the second inequality in \eqref{eq:mainestim} we will show that for
every $\alpha>0$ and $\omega\geq0$ there exists a constant $c>0$ such that,
for every $\left(  x,t\right)  ,\left(  y,s\right)  $ with $t\neq s$ one has
\[
\frac{1}{(t-s)^{\omega/2}}\Gamma_{\alpha}(x,t;y,s)\leq\frac{c}%
{d((x,t),(y,s))^{\omega+Q}},
\]
where $Q>0$ is the homogeneous dimension of $\mathbb{R}^{N}$, see \eqref{eq:defQhomdim}.

To this aim, we first observe that, since the matrix $C_{0}(1)^{-1}$ is
(symmetric and) positive definite, by combining
\eqref{eq:Gammaalfaconvolution} with \eqref{eq.exprGammaalfa} we get
\begin{align*}
\Gamma_{\alpha}(x,t;y,s)  &  =\Gamma_{\alpha}\big(x-E(t-s)y,t-s;0,0)\\
&  \leq\frac{c_{0}}{(t-s)^{Q/2}}\exp\Big(-c_{0}\,\Big|D_{0}\Big(\frac{1}%
{\sqrt{t-s}}\Big)(x-E(t-s)y)\Big|^{2}\Big),
\end{align*}
where $c_{0}>0$ is a suitable constant depending on $\alpha$; as a
consequence, taking into account the explicit expression of $d$ provided in
\eqref{eq:explicitd} (and since $\Vert\cdot\Vert$ is $D_{0}$%
\--ho\-mo\-ge\-neous of degree $1$), we obtain the following estimate
\begin{align*}
&  \frac{d((x,t),(y,s))^{\omega+Q}}{(t-s)^{\omega/2}}\cdot\Gamma_{\alpha
}(x,t;y,s)\\
&  \qquad=\frac{\big(\Vert x-E(t-s)y\Vert+\sqrt{|t-s|}\big)^{\omega+Q}%
}{(t-s)^{\omega/2}}\cdot\Gamma_{\alpha}(x,t;y,s)\\
&  \qquad=(t-s)^{Q/2}\,\Big(\Big\|D_{0}\Big(\frac{1}{\sqrt{t-s}}%
\Big)(x-E(t-s)y)\Big\|+1\Big)^{\omega+Q}\Gamma_{\alpha}(x,t;y,s)\\
&  \qquad\leq c_{0}\,\mathcal{U}\Big(D_{0}\Big(\frac{1}{\sqrt{t-s}%
}\Big)(x-E(t-s)y)\Big),
\end{align*}
where we have introduced the notation
\[
\mathcal{U}(z):=(\Vert z\Vert+1)^{\omega+Q}e^{-c_{0}|z|^{2}}\qquad
(z\in\mathbb{R}^{N}).
\]
To complete the proof it suffices to show that the function $\mathcal{U}$ is
\emph{globally bounded} in $\mathbb{R}^{N}$. To this end, bearing in mind the
explicit definition of $\Vert\cdot\Vert$, we notice that
\begin{align*}
0\leq\mathcal{U}(z)  &  \leq\Big(\sum_{i=1}^{N}|z|^{1/q_{i}}+1\Big)^{\omega
+Q}e^{-c_{0}|z|^{2}}\\
&  =\Big[\Big(\sum_{i=1}^{N}|z|^{1/q_{i}}+1\Big)e^{-\frac{c_{0}}{\omega
+Q}|z|^{2}})\Big]^{\omega+Q};
\end{align*}
from this, since the map $\tau\mapsto\tau^{\alpha}e^{-\beta\tau^{2}}$ is
globally bounded on $[0,+\infty)$ for every choice of $\alpha\geq0$ and
$\beta>0$, we conclude that $\mathcal{U}\in L^{\infty}(\mathbb{R}^{N})$, as desired.

Finally, combining the two inequalities in \eqref{eq:mainestim} we get%
\[
\left\vert D_{\left(  x,y\right)  }^{\bd\alpha}\Gamma(x,t;y,s)\right\vert
\leq\frac{c}{d(x,t;y,s)^{Q+\omega(\bd\alpha)}}%
\]
for every $\left(  x,t\right)  ,\left(  y,s\right)  $ with $t\neq s$. However,
for $x\neq y$ and $s\rightarrow t^{-}$, the first bound in
\eqref{eq:mainestim} shows that $D_{\left(  x,y\right)  }^{\bd\alpha}%
\Gamma(x,t;y,t)=0$, hence the above inequality actually holds for every
$\left(  x,t\right)  \neq\left(  y,s\right)  $, and we are done.
\end{proof}

We highlight a simple consequence of Theorem \ref{Thm bound derivatives} and
of \eqref{eq:integralGamma1} which will be repeatedly exploited in the sequel.

\begin{lemma}
\label{lem:integralvanishinggamma} Let $\Gamma$ be as in Theorem
\ref{Thm bound derivatives}, and let $\bd\alpha=(\alpha_{1},\ldots,\alpha
_{N})\in\N^{N}$ be a fixed non-zero multi-index. Then, we have
\begin{equation}
\int_{\mathbb{R}^{N}}D_{x}^{\bd\alpha}\Gamma(x,t;y,s)\,dy=0\qquad\text{for
every $x\in\mathbb{R}^{N}$ and every $s<t$}. \label{eq:intDxGammazero}%
\end{equation}

\end{lemma}

\begin{proof}
Let $x,s,t$ be as in the statement. Using the global estimates for
$D_{x}^{\bd\alpha}\Gamma$ gi\-ven in Theorem \ref{Thm bound derivatives}, and
taking into account identity \eqref{eq:integralGamma1}, we can perform a
standard do\-mi\-na\-ted-convergence argument, yielding
\[
\int_{\mathbb{R}^{N}}D_{x}^{\bd\alpha}\Gamma(x,t;y,s)\,dy=D_{x}^{\bd\alpha
}\Big(x\mapsto\int_{\mathbb{R}^{N}}\Gamma(x,t;y,s)\,dy\Big)=0.
\]
This ends the proof.
\end{proof}

The next theorem will also be a key tool in our a-priori estimates.

\begin{theorem}
[Mean value inequality for fractional and singular kernels]%
\label{Thm mean value} Let $\Gamma$ be as in Theorem
\ref{Thm fund sol coeff t dip}, and let $\eta=(y,s)\in\mathbb{R}^{N+1}$ be
fixed. Moreover, let
\[
\bd\alpha=(\alpha_{1},\ldots,\alpha_{N})
\]
be a fixed multi-index. Then, there exists a constant $c=c(\bd\alpha)>0$ such
that
\[
|D_{x}^{\bd{\alpha}}\Gamma(\xi_{1},\eta)-D_{x}^{\bd{\alpha}}\Gamma(\xi
_{2},\eta)|\leq c\frac{d(\xi_{1},\xi_{2})}{d(\xi_{1},\eta)^{Q+\omega
(\bd\alpha)+1}}%
\]
for every $\xi_{1}=(x_{1},t_{1}),\xi_{2}=(x_{2},t_{2})\in\mathbb{R}^{N+1}$
such that
\[
d(\xi_{1},\eta)\geq4\bd{\kappa}d(\xi_{1},\xi_{2})>0.
\]

\end{theorem}

\begin{proof}
Let $\xi_{1},\xi_{2}\in\mathbb{R}^{N+1}$ be as in the statement, and let
$r:=2d(\xi_{1},\xi_{2})>0$. Owing to \eqref{quasitriangle}, one can easily
recognize that $\eta\notin\overline{B}_{r}(\xi_{2})$; thus, taking into
account the \emph{regularity of $\Gamma$} stated in Theorem
\ref{Thm fund sol coeff t dip}-(1) and Remark \ref{rem:regulGammaMV}, we are
entitled to apply Theorem \ref{Lagrange} to the function $f:=D_{x}^{\bd\alpha
}\Gamma(\cdot;\eta)$ on the ball $\overline{B}_{r}(\xi_{2})\ni\xi_{1}$,
obtaining
\begin{equation}%
\begin{split}
|D_{x}^{\bd\alpha}  &  \Gamma(\xi_{1},\eta)-D_{x}^{\bd\alpha}\Gamma(\xi
_{2},\eta)|=|f(\xi_{1})-f(\xi_{2})|\\[0.1cm]
&  \leq c\Big(d(\xi_{1},\xi_{2})\cdot\sup_{B_{r}(\xi_{2})}\sqrt{\sum_{k=1}%
^{q}|\partial_{x_{k}}D_{x}^{\bd\alpha}\Gamma(\cdot;\eta)|^{2}}\\
&  \qquad\qquad+d(\xi_{1},\xi_{2})^{2}\cdot\sup_{B_{r}(\xi_{2})}%
|YD_{x}^{\bd\alpha}\Gamma(\cdot;\eta)|\Big).
\end{split}
\label{eq:lagrangefGamma}%
\end{equation}
Now, since $\xi_{1}\in B_{r}(\xi_{2})$ and $q_{k}=1$ for $1\leq k\leq q$, by
Theorem \ref{Thm bound derivatives} we have
\begin{equation}%
\begin{split}
\sup_{B_{r}(\xi_{2})}  &  \sqrt{\sum_{k=1}^{q}|\partial_{x_{k}}D_{x}%
^{\bd\alpha}\Gamma(\cdot;\eta)|^{2}}=\sup_{B_{r}(\xi_{2})}\sqrt{\sum_{k=1}%
^{q}|D_{x}^{\bd\alpha+\bd{e}_{k}}\Gamma(\cdot;\eta)|^{2}}\\
&  \qquad\qquad\leq c\sup_{\zeta\in B_{r}(\xi_{2})}\frac{1}{d(\zeta
,\eta)^{Q+\omega(\bd\alpha)+1}}\leq\frac{c}{d(\xi_{1},\eta)^{Q+\omega
(\bd\alpha)+1}}.
\end{split}
\label{eq:estimNablaf}%
\end{equation}
We then claim that we also have
\begin{equation}
\sup_{B_{r}(\xi_{2})}|YD_{x}^{\bd\alpha}\Gamma(\cdot;\eta)|\leq\frac{c}%
{d(\xi_{1},\eta)^{Q+\omega(\bd\alpha)+2}}. \label{eq:estimYf}%
\end{equation}
Taking this claim for granted for a moment, we can conclude the proof of
theorem: indeed, by combining \eqref{eq:lagrangefGamma},
\eqref{eq:estimNablaf} and \ref{eq:estimYf} we immediately obtain
\begin{align*}
|D_{x}^{\bd\alpha}\Gamma(\xi_{1},\eta)-D_{x}^{\bd\alpha}\Gamma(\xi_{2},\eta)|
&  \leq cd(\xi_{1},\xi_{2})\Big(\frac{1}{d(\xi_{1},\eta)^{Q+\omega
(\bd\alpha)+1}}+\frac{d(\xi_{1},\xi_{2})}{d(\xi_{1},\eta)^{Q+\omega
(\bd\alpha)+2}}\Big)\\[0.1cm]
&  (\text{since $d(\xi_{1},\eta)\geq4\bd\kappa d(\xi_{1},\xi_{2})$})\\
&  \leq c\frac{d(\xi_{1},\xi_{2})}{d(\xi_{1},\eta)^{Q+\omega(\bd\alpha)+1}},
\end{align*}
which is exactly what we wanted to prove.

Hence, we are left to prove the claimed \eqref{eq:estimYf}. To this end we
first notice that, since $\mathcal{L}\Gamma(\cdot;\eta)=0$ a.e.\thinspace in
$\mathbb{R}^{N+1}\setminus\{\eta\}$, we can write
\begin{equation}%
\begin{split}
YD_{x}^{\bd\alpha}\Gamma(\cdot;\eta)  &  =D_{x}^{\bd\alpha}\big(Y\Gamma
(\cdot;\eta)\big)+[Y,D_{x}^{\bd\alpha}]\Gamma(\cdot;\eta)\\
&  =-\sum_{i,j=1}^{q}a_{ij}(t)D_{x}^{\bd\alpha+\bd{e}_{i}+\bd{e}_{j}}%
\Gamma(\cdot;\eta)+[Y,D_{x}^{\bd\alpha}]\Gamma(\cdot;\eta),
\end{split}
\label{eq:YGammacommut}%
\end{equation}
where $[Y,D_{x}^{\bd\alpha}]=YD_{x}^{\bd\alpha}-D_{x}^{\bd\alpha}Y$. Moreover,
since the coefficients $a_{ij}$ are globally bounded (and $q_{k}=1$ for every
$1\leq k\leq q$), again by Theorem \ref{Thm bound derivatives} we get
\begin{equation}
\left\vert -\sum_{i,j=1}^{q}a_{ij}(t)D_{x}^{\bd\alpha+\bd{e}_{i}+\bd{e}_{j}%
}\Gamma(\zeta;\eta)\right\vert \leq\frac{c}{d(\zeta,\eta)^{Q+\omega
(\bd\alpha)+2}}\qquad\forall\,\,\zeta\in B_{r}(\xi_{2}).
\label{eq:estimYStepI}%
\end{equation}
We now turn to estimate the term $[Y,D_{x}^{\bd\alpha}]\Gamma(\cdot;\eta)$.
First of all, using the explicit expression of the vector field $Y$ in
\eqref{eq:driftY}, it is easy to see that
\[
\lbrack Y,D_{x}^{\bd\alpha}]=YD_{x}^{\bd\alpha}-D_{x}^{\bd\alpha}%
Y=\sum_{j,k=1}^{N}b_{jk}\alpha_{k}D_{x}^{\bd\alpha+\bd{e}_{j}-\bd{e}_{k}},
\]
where $\bd\alpha=\left(  \alpha_{1},\ldots,\alpha_{N}\right)  $ and the
$b_{jk}$'s are the entries of the matrix $B$. On the other hand, taking into
account the specific block form of $B$ in assumption (H2), it is not difficult
to recognize that
\[
q_{j}-q_{k}=2\qquad\text{for every $1\leq j,k\leq N$ such that $b_{jk}\neq0$%
}.
\]
As a consequence, using once again Theorem \ref{Thm bound derivatives}, we
get
\begin{equation}%
\begin{split}
\big|[Y,D_{x}^{\bd\alpha}]\Gamma(\zeta;\eta)\big|  &  \leq c\sum_{j,k=1}%
^{N}|b_{jk}|\cdot\frac{1}{d(\zeta,\eta)^{Q+\omega(\bd\alpha)+q_{j}-q_{k}}}\\
&  \leq\frac{c}{d(\zeta,\eta)^{Q+\omega(\bd\alpha)+2}}\qquad\forall
\,\,\zeta\in B_{r}(\xi_{2}).
\end{split}
\label{eq:estimYStepII}%
\end{equation}
Finally, by combining \eqref{eq:YGammacommut}, \eqref{eq:estimYStepI} and
\eqref{eq:estimYStepII} we obtain
\[
\sup_{\zeta\in B_{r}(\xi_{2})}|YD_{x}^{\bd\alpha}\Gamma(\zeta;\eta)|\leq
c\sup_{\zeta\in B_{r}(\xi_{2})}\frac{1}{d(\zeta,\eta)^{Q+\omega(\bd\alpha)+2}%
}\leq\frac{c}{d(\xi_{1},\eta)^{Q+\omega(\bd\alpha)+2}},
\]
which is precisely the claimed \ref{eq:estimYf}. This ends the proof.
\end{proof}

\subsection{Representation formulas for $u$ and $\partial_{x_{i}x_{j}}^{2}u$
in terms of $\mathcal{L}u$\label{sec repr formulas}}

We continue to consider an operator $\mathcal{L}$ with coefficients
$a_{ij}(t)$ satisfying (H1)-(H2), and its fun\-da\-men\-tal solution $\Gamma$
(see Theorem \ref{Thm fund sol coeff t dip}). Here, we are going to establish
some representation formulas for $u$ and for its derivatives in terms of
$\mathcal{L}u$. \medskip

We start with the following proposition.

\begin{proposition}
\label{Thm limit v epsi} Let $T\in\mathbb{R}$ be fixed, and let $g:S_{T}%
\rightarrow\mathbb{R}$ be \emph{continuous and bounded}. For every $\e>0$, we
consider the function
\[
v_{\varepsilon}:S_{T}\rightarrow\mathbb{R},\qquad v_{\varepsilon}%
(x,t):=\int_{\mathbb{R}^{N}}\Gamma(x,t;y,t-\e)\,g(y,t-\e)\,dy.
\]
Then, $v_{\varepsilon}\rightarrow g$ \emph{pointwise in $S_{T}$} as
$\e\rightarrow0^{+}$.
\end{proposition}

\begin{proof}
Let $(x,t)\in S_{T}$. By combining \eqref{eq:integralGamma1} with
\eqref{G G_nu}, we can write
\begin{align*}
&  |v_{\varepsilon}(x,t)-g(x,t)|=\bigg|\int_{\mathbb{R}^{N}}\Gamma
(x,t;y,t-\e)\big(g(y,t-\e)-g(x,t)\big)\,dy\bigg|\\[0.1cm]
&  \qquad\leq\frac{1}{\nu^{N}}\int_{\mathbb{R}^{N}}\Gamma_{\nu^{-1}%
}(x,t;y,t-\e)\cdot|g(y,t-\e)-g(x,t)|\,dy\\[0.1cm]
&  \qquad\leq\frac{c_{0}}{\e^{Q/2}}\int_{\mathbb{R}^{N}}e^{-c_{0}%
\,\big|D_{0}\big(\frac{1}{\sqrt{\e}}\big)(x-E(\e)y)\big|^{2}}\cdot
|g(y,t-\e)-g(x,t)|\,dy=(\bigstar),
\end{align*}
where $c_{0}>0$ is a suitable constant only depending on $\nu>0$. On the other
hand, taking into account \eqref{LP 2.20} and performing the change of
variables
\[
y=E(-\e)x-D_{0}(\sqrt{\e})E(-1)z,
\]
we derive
\[
(\bigstar)=c_{0}\int_{\mathbb{R}^{N}}e^{-c_{0}\,|z|^{2}}\big|g(E(-\e)x-D_{0}%
(\sqrt{\e})E(-1)z,t-\e)-g(x,t)\big|\,dz,
\]
since $\det(E(-1))=1$. Summing up, we obtain the estimate
\begin{equation}
|v_{\varepsilon}(x,t)-g(x,t)|\leq c_{0}\int_{\mathbb{R}^{N}}e^{-c_{0}%
\,|z|^{2}}h_{\varepsilon}(z)\,dz, \label{eq:estimsupinthe}%
\end{equation}
where we have introduced the simplified notation
\[
h_{\varepsilon}(z):=\big|g(E(-\e)x-D_{0}(\sqrt{\e})E(-1)z,t-\e)-g(x,t)\big|.
\]
Now, since $E(-\e)\rightarrow E(0)=\mathrm{Id}_{N}$ and $D_{0}(\sqrt
{\e})\rightarrow\mathbb{O}_{N}$ as $\e\rightarrow0^{+}$ (see
\eqref{dilations}), from the continuity of $g$ we immediately derive that
\[
\lim_{\e\rightarrow0^{+}}h_{\varepsilon}(z)\rightarrow0\qquad\text{for every
fixed $z\in\mathbb{R}^{N}$}.
\]
Moreover, since $g$ is globally bounded on $S_{T}$, we have
\[
0\leq|h_{\varepsilon}(z)|\leq2\,\Vert g\Vert_{L^{\infty}(S_{T})}.
\]
Gathering these two facts, we can apply the dominated convergence the\-o\-rem
in the right-hand side of \eqref{eq:estimsupinthe}, yielding
\[
\text{$|v_{\varepsilon}(x,t)-g(x,t)|\rightarrow0$ as $\e\rightarrow0^{+}.$}%
\]
By the arbitrariness of $(x,t)\in S_{T}$, this completes the proof.
\end{proof}

Thanks to Proposition \ref{Thm limit v epsi}, we can now prove the next key
result. Throughout the sequel, when dealing with integral over strips we
tacitly understand that
\[
\int_{\mathbb{R}^{N}\times(a,b)}{\cdots}=-\int_{\mathbb{R}^{N}\times
(b,a)}\{\cdots\}\quad\text{when $b<a$}.
\]

\begin{theorem}
\label{Thm repr formula u} Let $T\in\mathbb{R}$ be fixed, and let $\tau<T$.
Moreover, let $u\in\mathcal{S}^{0}(\tau;T)$. Then, we have the following
\emph{representation formula}
\begin{equation}
\label{repr formula u}u(x,t)=-\int_{\mathbb{R}^{N}\times(\tau,t)}%
\Gamma(x,t;y,s)\mathcal{L}u(y,s)\,dy\,ds,
\end{equation}
for every point $(x,t)\in S_{T}$.
\end{theorem}

\begin{proof}
Since $u\in\mathcal{S}^{0}(\tau;T)$, then $\mathcal{L}u\in L^{\infty}(S_{T})$.
Thus, ta\-king into account \eqref{eq:integralGamma1} in Theorem
\ref{Thm fund sol coeff t dip}, for every $(x,t)\in S_{T}$ we get
\begin{equation}%
\begin{split}
&  \bigg|\int_{\mathbb{R}^{N}\times(\tau,t)}\big|\Gamma(x,t;y,s)\mathcal{L}%
u(y,s)\big|\,dy\,ds\bigg|\\
&  \qquad\leq\Vert\mathcal{L}u\Vert_{L^{\infty}(S_{T})}\bigg|\int_{\tau}%
^{t}\bigg(\int_{\mathbb{R}^{N}}\Gamma(x,t;y,s)\,dy\bigg)ds\bigg|=|t-\tau
|<\infty,
\end{split}
\label{eq:summGammaLLRF}%
\end{equation}
and this proves that the right-hand side of \ref{repr formula u} is
\emph{finite}. Now, in order to e\-sta\-bli\-sh the representation formula
\ref{repr formula u}, we proceed by steps. \medskip

\textsc{Step I.} Let us first prove \ref{repr formula u} by assuming that
$u\in\mathcal{S}^{0}(\tau;T)$ satisfies the following \emph{additional
properties}:

\begin{itemize}
\item[(i)] $u\in C^{\infty}(S_{T})$;

\item[(ii)] there exists $r > 0$ such that
\[
\text{$u(x,t) = 0$ for every $(x,t)\in S_{T}$ with $|x| > r$}.
\]

\end{itemize}

Then, owing to \eqref{eq:summGammaLLRF} we have
\begin{equation}
\int_{\mathbb{R}^{N}\times(\tau,t)}\Gamma(x,t;\cdot)\,\mathcal{L}%
u\,dy\,ds=\lim_{\e\rightarrow0^{+}}\int_{\mathbb{R}^{N}\times(\tau
,t-\e)}\Gamma(x,t;\cdot)\,\mathcal{L}u\,dy\,ds; \label{eq:topasslimitRFStepI}%
\end{equation}
moreover, since we are assuming that $u\in C^{\infty}(S_{T})$, we can write
\begin{equation}%
\begin{split}
&  \int_{\mathbb{R}^{N}\times(\tau,t-\e)}\Gamma(x,t;\cdot)\,\mathcal{L}%
u\,dy\,ds\\[0.1cm]
&  \qquad=\int_{\tau}^{t-\e}\bigg(\int_{\mathbb{R}^{N}}\Gamma(x,t;\cdot
)\,\mathcal{L}_{0}u\,dy\bigg)ds-\int_{\mathbb{R}^{N}}\bigg(\int_{\tau}%
^{t-\e}\Gamma(x,t;\cdot)\,\partial_{s}u\,ds\bigg)dy,
\end{split}
\label{eq:whereinsertFE}%
\end{equation}
where we have written $\mathcal{L}=\mathcal{L}_{0}-\partial_{s}$, that is,
\[
\textstyle\mathcal{L}_{0}=\sum_{i,j=1}^{q}a_{ij}(s)\partial_{y_{i}y_{j}}%
+\sum_{j,k=1}^{N}b_{jk}y_{k}\partial_{y_{j}}.
\]
Now, owing to Theorem \ref{Thm fund sol coeff t dip}-(1), we readily see that
$y\mapsto\Gamma(x,t;y,s)\in C^{\infty}(\mathbb{R}^{N})$ for every {fixed}
point $(x,t)\in S_{T}$ and every $s<t-\e$; as a consequence, taking into
account the additional assumptions (i)-(ii), we have
\begin{equation}
\int_{\mathbb{R}^{N}}\Gamma(x,t;\cdot)\,\mathcal{L}_{0}u\,dy=\int%
_{\mathbb{R}^{N}}(\mathcal{L}_{0})^{\ast}\Gamma(x,t;\cdot)\,u\,dy,
\label{eq:intbypartsRF}%
\end{equation}
where $\mathcal{L}_{0}^{\ast}$ denotes the formal adjoint of $\mathcal{L}_{0}%
$, that is,
\[
\textstyle\mathcal{L}_{0}^{\ast}=\sum_{i,j=1}^{q}a_{ij}(s)\partial_{y_{i}%
y_{j}}-\sum_{j,k=1}^{N}b_{jk}y_{k}\partial_{y_{j}}.
\]
On the other hand, since from Theorem \ref{Thm fund sol coeff t dip}-(1) we
also derive that $s\mapsto\Gamma(x,t;y,s)$ is Lipschitz-continuous on
$(\tau,t-\e)$, again by (i)-(ii) we have
\begin{equation}
\int_{\tau}^{t-\e}\Gamma(x,t;\cdot)\,\partial_{s}u\,ds=\Gamma
(x,t;y,t-\e)u(y,t-\e)-\int_{\tau}^{t-\e}\partial_{s}\Gamma(x,t;\cdot)\,u\,ds,
\label{eq:TFCRF}%
\end{equation}
where we have also used the fact that $u\in\mathcal{S}^{0}(\tau;T)$. Gathering
\eqref{eq:intbypartsRF}-\eqref{eq:TFCRF}, from the above
\eqref{eq:whereinsertFE} we then obtain the following identity
\begin{align*}
\int_{\mathbb{R}^{N}\times(\tau,t-\e)}\Gamma(x,t;\cdot)\,  &  \mathcal{L}%
u\,dy\,ds=-\int_{\mathbb{R}^{N}}\Gamma(x,t;y,t-\e)u(y,t-\e)\,dy\\
&  +\int_{\mathbb{R}^{N}\times(\tau,t-\e)}(\mathcal{L}_{0}^{\ast}+\partial
_{s})\Gamma(x,t;\cdot)\,u\,dy\,ds.
\end{align*}
Then, taking into account \eqref{eq:topasslimitRFStepI}, in order to establish
formula \eqref{repr formula u} it is enough to prove the following fact:
\begin{equation}%
\begin{split}
\bigg(\int_{\mathbb{R}^{N}\times(\tau,t-\e)}\!\!\!(\mathcal{L}_{0}^{\ast
}\,+\,  &  \partial_{s})\Gamma(x,t;\cdot)\,u\,dy\,ds-\int_{\mathbb{R}^{N}%
}\Gamma(x,t;y,t-\e)u(y,t-\e)\,dy\bigg)\\
&  \rightarrow\,\,-u(x,t)\qquad\text{for every $(x,t)\in S_{T}$ as
$\e\rightarrow0^{+}$}.
\end{split}
\label{eq:toprovelungaRF}%
\end{equation}
To this end we first notice that, owing to Theorem
\ref{Thm fund sol coeff t dip}, we have
\begin{equation}
(\mathcal{L}_{0}^{\ast}+\partial_{s})\Gamma(x,t;\cdot)=\mathcal{L}^{\ast
}\Gamma(x,t;\cdot)=0\quad\text{a.e.\thinspace on $\mathbb{R}^{N}\times
(\tau,t-\e)$}; \label{eq:LLzerostarGammazeroRF}%
\end{equation}
moreover, since $u\in\mathcal{S}^{0}(\tau;T)$ (hence, in particular, $u$ is
continuous and bounded on the strip $S_{T}$), from Proposition
\ref{Thm limit v epsi} we infer that
\begin{equation}
\lim_{\e\rightarrow0^{+}}\int_{\mathbb{R}^{N}}\Gamma
(x,t;y,t-\e)u(y,t-\e)\,dy=u(x,t)\quad\text{pointwise on $S_{T}$}.
\label{eq:limitbyPropveRF}%
\end{equation}
By combining \eqref{eq:LLzerostarGammazeroRF}-\eqref{eq:limitbyPropveRF}, we
immediately obtain \eqref{eq:toprovelungaRF}. \medskip

\textsc{Step II.} Let us now prove the representation formula
\eqref{repr formula u} by dropping the additional assumption (ii) on $u$, that
is, we only suppose that
\[
u\in\mathcal{S}^{0}(\tau;T)\cap C^{\infty}(S_{T}).
\]
To begin with, we fix a cut-off function $\phi_{0}\in C_{0}^{\infty
}(\mathbb{R}^{N})$ such that

\begin{itemize}
\item[(a)] $0\leq\phi_{0}\leq1$ in $\mathbb{R}^{N}$; \vspace*{0.1cm}

\item[(b)] $\phi_{0}\equiv1$ on $\{|x| < 1\}$ and $\phi_{0}\equiv0$ on $\{|x|
> 2\}$.
\end{itemize}

Moreover, for every $n\geq1$ we set $\phi_{n}(x):=\phi_{0}(x/n)$, and we
define
\[
u_{n}:=u\cdot\phi_{n}.
\]
Owing to (a)-(b), it is readily seen that $u_{n}\in\mathcal{S}^{0}(\tau;T)\cap
C^{\infty}(S_{T})$ and $u_{n}(x,t)=0$ for every $(x,t)\in S_{T}$ with $|x|>n$;
hence, by Step I we can write
\begin{equation}
u_{n}(x,t)=-\int_{\mathbb{R}^{N}\times(\tau,t)}\Gamma(x,t;\cdot)\,\mathcal{L}%
u_{n}\,dy\,ds\quad\text{for every $(x,t)\in S_{T}$}.
\label{eq:topasslimitnRFStepII}%
\end{equation}
We now aim to pass to the limit as $n\rightarrow\infty$ in the above
\eqref{eq:topasslimitnRFStepII}. By definition of $\phi_{n}$, we have
\begin{equation}
\lim_{n\rightarrow\infty}u_{n}(x,t)=u(x,t)\quad\text{for every fixed $(x,t)\in
S_{T}$}. \label{eq:pointwiselimunRF}%
\end{equation}
As to the right-hand side, instead, we rely on the dominated convergence
th\-e\-o\-rem. First of all, since $u\in C^{\infty}(S_{T})$ and $\phi_{n}\in
C_{0}^{\infty}(\mathbb{R}^{N})$, we have
\[
\mathcal{L}u_{n}=\mathcal{L}(u\phi_{n})=(\mathcal{L}u)\cdot\phi_{n}%
+u\cdot(\mathcal{L}\phi_{n})+2\sum_{i,j=1}^{q}a_{ij}(t)\partial_{x_{i}%
}u\,\partial_{x_{j}}\phi_{n};
\]
moreover, since $u\in\mathcal{S}^{0}(\tau;T)$ and $\phi_{n}=\phi_{0}(\cdot
/n)$, there exists a constant $\mathbf{c}>0$, depending on $u$ and $\phi_{0}$
but {independent of $n$}, such that
\[
\Big|u\cdot(\mathcal{L}\phi_{n})+2\sum_{i,j=1}^{q}a_{ij}(t)\partial_{x_{i}%
}u\,\partial_{x_{j}}\phi_{n}\Big|\leq\frac{\mathbf{c}}{n}\qquad\text{pointwise
on $S_{T}$}.
\]
This, together with the fact that $\phi_{n}\equiv1$ on $\{|x|<n\}$, implies
\[
\lim_{n\rightarrow\infty}\mathcal{L}u_{n}=\mathcal{L}u\qquad\text{pointwise on
$S_{T}$}.
\]
On the other hand, since $\mathcal{L}u\in L^{\infty}(S_{T})$ and $0\leq
\phi_{n}\leq1$, we also have
\[
|\mathcal{L}u_{n}|\leq\Vert\mathcal{L}u\Vert_{L^{\infty}(S_{T})}%
+\frac{\mathbf{c}}{n}\leq\Vert\mathcal{L}u\Vert_{L^{\infty}(S_{T})}%
+\mathbf{c}=:\mathbf{c}^{\prime}\qquad\text{for every $n\geq1$};
\]
gathering these facts, and taking into account \eqref{eq:integralGamma1}, we
can then apply the dominated convergence theorem in the right-hand side of
\eqref{eq:topasslimitnRFStepII}, yielding
\begin{equation}
\lim_{n\rightarrow\infty}\int_{\mathbb{R}^{N}\times(\tau,t)}\Gamma
(x,t;\cdot)\,\mathcal{L}u_{n}\,dy\,ds=\int_{\mathbb{R}^{N}\times(\tau
,t)}\Gamma(x,t;\cdot)\,\mathcal{L}u\,dy\,ds. \label{eq:limitintegralRF}%
\end{equation}
Finally, by combining \eqref{eq:pointwiselimunRF} and
\eqref{eq:limitintegralRF} we can let $n\rightarrow\infty$ in
\eqref{eq:topasslimitnRFStepII}, thus obtaining the desired representation
formula \eqref{repr formula u} for $u$. \medskip

\textsc{Step III:} Let us finally prove the representation formula
\eqref{repr formula u} for every $u\in\mathcal{S}^{0}(\tau;T)$.

To begin with, we fix a point $\xi_{0}=(x_{0},t_{0})\in S_{T}$ and we choose
$0<\e_{0}<1$ in such a way that $\xi_{0}\in S_{T-\e_{0}}$. Moreover, we choose
a function $J\in C_{0}^{\infty}(\mathbb{R}^{N+1})$ such that $J\geq0$
pointwise in $\mathbb{R}^{N+1}$, $\mathrm{supp}(J)\subseteq B_{1}(0)$ and
\begin{equation}
\int_{\mathbb{R}^{N+1}}J(\eta)\,d\eta=\int_{B_{1}}J(\eta)\,d\eta=1,
\label{eq:integralJone}%
\end{equation}
where $B_{1}(0)=\{\eta:\,d(\eta,0)<1\}$ is the $d$-ball with centre $0$ and
radius $1$. We then define, for every fixed $0<\e<\e_{0}$, the \emph{$(\e,\G)$%
-convolution kernel}
\[
J_{\varepsilon}(\eta):=\e^{-Q-2}J\big(D(1/\e)\eta\big)
\]
(where $D(\cdot)$ and $Q>0$ are as in \eqref{dilations} and
\eqref{eq:defQhomdim}, respectively), and we consider the so-called
\emph{mollifier of $u$ related to the kernel $J_{\varepsilon}$}, that is,
\begin{align*}
&  u_{\varepsilon}:S_{T-\e_{0}}\rightarrow\mathbb{R},\\[0.1cm]
&  u_{\varepsilon}(\xi):=\int_{S_{T}}J_{\varepsilon}(\xi\circ\eta
^{-1})\,u(\eta)\,d\eta=\int_{B_{1}(0)}J(\zeta)\,u\big((D(\e)\zeta^{-1}%
)\circ\xi\big)\,d\zeta.
\end{align*}
We explicitly point out, for the sake of completeness, that the definition of
$u_{\varepsilon}$ is \emph{meaningful}: in fact, using
\eqref{eq:convolutionG}, \eqref{dilations} and \eqref{d} we easily see that

\begin{itemize}
\item[(a)] for every fixed $\xi=(x,t)\in S_{T-\e_{0}}$, one has
\begin{equation}
\mathrm{supp}\big(\eta\mapsto J_{\varepsilon}(\xi\circ\eta^{-1})\big)\subseteq
\{\eta=(y,s):\,|t-s|<\e\}\subseteq S_{T}; \label{eq:supportJe}%
\end{equation}

\item[(b)] for every $\zeta\in B_{1}(0)$ and $\xi\in S_{T-\e_{0}}$, one has
$(D(\e)\zeta^{-1})\circ\xi\in S_{T}$.
\end{itemize}

We now claim that:%
\begin{equation}
u_{\varepsilon}\in\mathcal{S}^{0}(\tau-\e_{0};T-\e_{0})\cap C^{\infty
}(S_{T-\e_{0}}). \label{eq:toproveueRFStepIII}%
\end{equation}
Indeed, since $J\in C_{0}^{\infty}(\mathbb{R}^{N+1})$, by a standard
dominated-convergence argument we easily infer that $u_{\varepsilon}\in
C^{\infty}(S_{T-\e_{0}})$; moreover, taking into account that $u(x,t)\equiv0$
for every $(x,t)\in S_{T}$ with $t\leq\tau$, by \eqref{eq:supportJe} we derive
that
\[
u_{\varepsilon}(\xi)=\int_{\{|t-s|<\e\}}J_{\varepsilon}((x,t)\circ
(y,s)^{-1})\,u(y,s)\,dy\,ds=0
\]
for every $\xi=(x,t)\in S_{T-\e_{0}}$ with $t\leq\tau-\e_{0}$. Hence, to prove
the claimed \eqref{eq:toproveueRFStepIII} we are left to show that the
derivatives $\partial_{x_{i}x_{j}}^{2}u_{\varepsilon},Yu_{\varepsilon}$, which
exist {pointwise and in the classical sense on $S_{T-\e_{0}}$}, are globally
bounded in $S_{T-\e_{0}}$ (for $1\leq i,j\leq q$).

To this end it suffices to observe that, since $u\in\mathcal{S}^{0}(\tau;T)$
and since the vector fields $\partial_{x_{1}},\ldots,\partial_{x_{q}},Y$ are
{left\--in\-va\-riant} with respect to $\circ$, we can write
\begin{equation}%
\begin{split}
\partial_{x_{i}x_{j}}^{2}u_{\varepsilon}\left(  \xi\right)   &  =\int%
_{B_{1}(0)}J(\zeta)\,(\partial_{x_{i}x_{j}}^{2}u)\big(D(\e)\zeta^{-1})\circ
\xi\big)\,d\zeta\qquad(\text{for $i=1,\ldots,q$}),\\
Yu_{\varepsilon}\left(  \xi\right)   &  =\int_{B_{1}(0)}J(\zeta
)\,(Yu)\big(D(\e)\zeta^{-1})\circ\xi\big)\,d\zeta,
\end{split}
\label{eq:deprdexiuYuRFStepIII}%
\end{equation}
thus, since $\partial_{x_{i}x_{j}}^{2}u,Yu\in L^{\infty}(S_{T})$, from
\eqref{eq:integralJone} we obtain
\begin{equation}%
\begin{split}
\Vert\partial_{x_{i}x_{j}}^{2}u_{\varepsilon}\Vert_{L^{\infty}(S_{T-\e_{0}})}
&  \leq\Vert\partial_{x_{i}x_{j}}^{2}u\Vert_{L^{\infty}(S_{T})}\quad(\text{for
$1\leq i,j\leq q$});\\
\Vert Yu_{\varepsilon}\Vert_{L^{\infty}(S_{T-\e_{0}})}  &  \leq\Vert
Yu\Vert_{L^{\infty}(S_{T})},
\end{split}
\label{eq:boundederue}%
\end{equation}
and this completes the proof of \eqref{eq:toproveueRFStepIII}.

Now we have established \eqref{eq:toproveueRFStepIII}, thanks to Step II we
know that the representation formula \eqref{repr formula u} holds {for the
function $u_{\varepsilon}$ on the strip $S_{T-\e_{0}}$}: in particular, since
we have that $\xi_{0}=(x_{0},t_{0})\in S_{T-\e_{0}}$, we can write
\begin{equation}
u_{\varepsilon}(x_{0},t_{0})=-\int_{\mathbb{R}^{N}\times(\tau-\e_{0},t_{0}%
)}\Gamma(x_{0},t_{0};\cdot)\,\mathcal{L}u_{\varepsilon}\,dy\,ds.
\label{eq:RFforueStepIII}%
\end{equation}
We then pass to the limit as $\e\rightarrow0^{+}$ in
\eqref{eq:RFforueStepIII}. As to the left-hand, since $u$ is continuous and
bounded on $S_{T}$, it is easily seen that
\begin{equation}
\lim_{\varepsilon\rightarrow0^{+}}u_{\varepsilon}(x_{0},t_{0})=u(x_{0},t_{0}).
\label{eq:limitueuRFStepIII}%
\end{equation}
As to the right-hand side, taking into account \eqref{eq:deprdexiuYuRFStepIII}
and the fact that
\[
\partial_{x_{i}x_{j}}^{2}u,Yu\in L^{\infty}(S_{T}),
\]
we can use a classical approximation argument to prove that $\partial
_{x_{i}x_{j}}^{2}u_{\varepsilon}\rightarrow\partial_{x_{i}x_{j}}^{2}u$ (for
e\-ve\-ry $1\leq i,j\leq q$) and $Yu_{\varepsilon}\rightarrow Y_{u}$ in
$L_{\mathrm{loc}}^{1}(S_{T-\e_{0}})$ as $\e\rightarrow0^{+}$; as a
consequence, by possibly choosing a sequence $\e_{n}\rightarrow0$ as
$n\rightarrow\infty$, we get
\[
\lim_{\e\rightarrow0^{+}}\mathcal{L}u_{\varepsilon}=\lim_{\e\rightarrow0^{+}%
}\bigg(\sum_{i,j=1}^{q}a_{ij}(\cdot)\partial_{x_{i}x_{j}}^{2}u_{\varepsilon
}+Yu_{\varepsilon}\bigg)=\mathcal{L}u\quad\text{a.e.\thinspace in
$S_{T-\e_{0}}$}.
\]
On the other hand, using \eqref{eq:boundederue} and the fact that the
coefficients $a_{ij}$ are globally bounded, we also have the following
estimate
\[%
\begin{split}
|\mathcal{L}u_{\varepsilon}|  &  \leq\sum_{i,j=1}^{q}\Vert a_{ij}%
\Vert_{L^{\infty}(\mathbb{R})}\cdot\Vert\partial_{x_{i}x_{j}}^{2}%
u\Vert_{L^{\infty}(S_{T})}+\Vert Yu\Vert_{L^{\infty}(S_{T})}\\
&  =:\mathbf{c},\qquad\text{for every $0<\e<\e_{0}$}.
\end{split}
\]
Gathering these facts, and recalling that $J\in C_{0}^{\infty}(\mathbb{R}%
^{N+1})$, we can then apply the dominated con\-ver\-gen\-ce theorem in the
right-hand side of \eqref{eq:RFforueStepIII}, getting
\begin{equation}%
\begin{split}
&  \lim_{\varepsilon}\int_{\mathbb{R}^{N}\times(\tau-\e_{0},t_{0})}%
\Gamma(x_{0},t_{0};\cdot)\,\mathcal{L}u_{\varepsilon}\,dy\,ds\\
&  \qquad\qquad=\int_{\mathbb{R}^{N}\times(\tau-\e_{0},t_{0})}\Gamma
(x_{0},t_{0};\cdot)\,\mathcal{L}u\,dy\,ds
\end{split}
\label{eq:limitLLueRFStepIII}%
\end{equation}
Finally, by combining
\eqref{eq:limitueuRFStepIII}-\eqref{eq:limitLLueRFStepIII} and by taking into
account that
\[
\text{$u=\mathcal{L}u\equiv0$ a.e.\thinspace on $\mathbb{R}^{N}\times
(-\infty,\tau)$},
\]
we can pass to the limit as $\e\rightarrow0^{+}$ in \eqref{eq:RFforueStepIII},
thus obtaining the desired representation formula \eqref{repr formula u} for
$u$. This completes the proof.
\end{proof}

Starting from the representation formula \eqref{repr formula u}, we easily
obtain the following representation formula for the \emph{first-order
derivatives} of $u$.

\begin{corollary}
\label{cor:firtderivrepr} Let $T\in\mathbb{R}$ be fixed, and let $\tau<T$.
Moreover, let $u\in\mathcal{S}^{0}(\tau;T)$ and let $1\leq i\leq q$ be fixed.
Then, we have the representation formula
\begin{equation}
\partial_{x_{i}}u(x,t)=-\int_{\mathbb{R}^{N}\times(\tau,t)}\partial_{x_{i}%
}\Gamma(x,t;\cdot)\,\mathcal{L}u\,dy\,ds, \label{eq:reprudexi}%
\end{equation}
for every $(x,t)\in S_{T}$. Moreover,
\begin{equation}
\Vert\partial_{x_{i}}u\Vert_{L^{\infty}(S_{T})}\leq c\,\Vert\mathcal{L}%
u\Vert_{L^{\infty}(S_{T})}\cdot\sqrt{T-\tau}. \label{eq:dexiuStatement}%
\end{equation}

\end{corollary}

\begin{proof}
We start noting that, combining the global estimates for $\partial_{x_{i}%
}\Gamma$ contained in Theorem \ref{Thm bound derivatives}, see
\eqref{eq:mainestim}, with identity \eqref{eq:integralGamma1}, we have, for
every $x\in\mathbb{R}^{N}$ and every $\tau<t$,%
\begin{align}
\int_{\mathbb{R}^{N}\times(\tau,t)}|\partial_{x_{i}}\Gamma(x,t;\cdot
)|\,dy\,ds  &  \leq c\int_{\tau}^{t}\frac{1}{\sqrt{t-s}}\bigg(\int%
_{\mathbb{R}^{N}}\Gamma_{c_{1}\nu^{-1}}(x,t;\cdot)\,dy\bigg)ds\nonumber\\
&  =c\int_{\tau}^{t}\frac{1}{\sqrt{t-s}}\,ds=2c\sqrt{t-\tau}.
\label{eq:integraldexi}%
\end{align}

Let us now prove formula \eqref{eq:reprudexi}. To begin with, since
$u\in\mathcal{S}^{0}(\tau;T)$, we have $\mathcal{L}u\in L^{\infty}(S_{T})$;
thus, from \eqref{eq:integraldexi} we get
\begin{equation}%
\begin{split}
&  \bigg|\int_{\mathbb{R}^{N}\times(\tau,t)}|\partial_{x_{i}}\Gamma
(x,t;\cdot)|\,|\mathcal{L}u|dy\,ds\bigg|\\
&  \qquad\leq c\,\Vert\mathcal{L}u\Vert_{L^{\infty}(S_{T})}\cdot\sqrt
{|t-\tau|}\qquad\forall\,\,(x,t)\in S_{T}%
\end{split}
\label{eq:dexiuglobalbd}%
\end{equation}
(where $c>0$ only depends on $\nu$), and this shows that the function
\[
g(x,t):=-\int_{\mathbb{R}^{N}\times(\tau,t)}\partial_{x_{i}}\Gamma
(x,t;\cdot)\,\mathcal{L}u\,dy\,ds,
\]
is {well-defined on $S_{T}$}. We then turn to prove that $\partial_{x_{i}%
}u\equiv g$ pointwise in $S_{T}$ by an approximation argument. To this end, we
fix $0<\e\ll1$ and we define
\[
u_{\varepsilon}(x,t):=-\int_{\mathbb{R}^{N}\times(\tau,t-\e)}\Gamma
(x,t;\cdot)\,\mathcal{L}u\,dy\,ds.
\]
Owing to the representation formula \eqref{repr formula u}, it is readily seen
that $u_{\varepsilon}\rightarrow u$ point\-wi\-se on $S_{T}$ as $\e\rightarrow
0^{+}$; moreover, since $t-s\geq\e>0$ when $s<t-\e$, by simple
do\-mi\-na\-ted-con\-ver\-gen\-ce arguments based on \eqref{eq:mainestim} (and
on the regularity of $\Gamma$, see Theorem \ref{Thm fund sol coeff t dip}-(1))
we easily infer that

\begin{itemize}
\item[(i)] $u_{\varepsilon}\in C(S_{T})$;

\item[(ii)] $u_{\varepsilon}$ is continuously differentiable w.r.t.\thinspace
$x_{i}$ on $S_{T}$, and
\[
\partial_{x_{i}}u_{\varepsilon}(x,t)=-\int_{\mathbb{R}^{N}\times(\tau
,t-\e)}\partial_{x_{i}}\Gamma(x,t;\cdot)\,\mathcal{L}u\,dy\,ds\quad
\forall\,\,(x,t)\in S_{T}.
\]

\end{itemize}

Finally, by \eqref{eq:integraldexi} we also have
\begin{align*}
|\partial_{x_{i}}u_{\varepsilon}(x,t)-g(x,t)|  &  =\int_{\mathbb{R}^{N}%
\times(t-\e,t)}|\partial_{x_{i}}\Gamma(x,t;\cdot)|\,\Vert\mathcal{L}%
u\Vert_{L^{\infty}(S_{T})}\,dy\,ds\\
&  \leq c\,\Vert\mathcal{L}u\Vert_{L^{\infty}(S_{T})}\sqrt{\e}\qquad
\text{uniformly for $(x,t)\in S_{T}$},
\end{align*}
from which we derive that $\partial_{x_{i}}u_{\varepsilon}\rightarrow g$
{uniformly on $S_{T}$ as $\e\rightarrow0^{+}$}. As is well-know, all the above
facts are enough to conclude that
\[
\text{$\partial_{x_{i}}u\equiv g$ on $S_{T}$},
\]
and this is precisely \eqref{eq:reprudexi}. By \eqref{eq:dexiuglobalbd}, this
also implies \eqref{eq:dexiuStatement}.
\end{proof}

With the representation formula \eqref{eq:reprudexi} at hand, we now aim to
prove a representation formula for the \emph{de\-ri\-va\-ti\-ves}
$\partial_{x_{i}x_{j}}u$ of a function $u\in\mathcal{S}^{0}(\tau;T)$. \medskip

To this end, we first establish the following proposition.

\begin{proposition}
\label{Proposition Claim} Let $\alpha\in(0,1)$ be fixed, and let $1\leq
i,j\leq q$. Then, there exists a constant ${c}={c}(\alpha)>0$ such that, for
every $x\in\mathbb{R}^{N}$ and every $\tau<t$, one has
\begin{equation}
\int_{\mathbb{R}^{N}\times(\tau,t)}|\partial_{x_{i}x_{j}}^{2}\Gamma
(x,t;y,s)|\cdot\Vert E(s-t)x-y\Vert^{\alpha}\,dy\,ds\leq{c}(t-\tau)^{\alpha
/2}. \label{eq:integraldexixjconv}%
\end{equation}
As a consequence, we have
\begin{equation}%
\begin{split}
\int_{\mathbb{R}^{N}\times(t-\e,t)}  &  |\partial_{x_{i}x_{j}}^{2}%
\Gamma(x,t;y,s)|\cdot\Vert E(s-t)x-y\Vert^{\alpha}\,dy\,ds\rightarrow0\\
&  \text{\emph{uniformly w.r.t.\thinspace$(x,t)\in\mathbb{R}^{N+1}$} as
$\e\rightarrow0^{+}$}.
\end{split}
\label{eq:integraldexixjunifconv}%
\end{equation}

\end{proposition}

\begin{proof}
Let $x,\tau,t$ be as in the statement. Owing to the global estimates for
$\partial_{x_{i}x_{j}}^{2}\Gamma$ in Theorem \ref{Thm bound derivatives}, see
\eqref{eq:mainestim}, and taking into account
\eqref{eq:Gammaalfaconvolution}-\eqref{eq.exprGammaalfa}, we have
\begin{equation}%
\begin{split}
|\partial_{x_{i}x_{j}}^{2}\Gamma(x,t;y,s)|  &  \leq\frac{c}{t-s}\Gamma
_{c_{1}\nu^{-1}}(x,t;y,s)\\
&  \leq\frac{c_{0}}{(t-s)^{Q/2+1}}e^{-c_{0}\big|D_{0}\big(\frac{1}{\sqrt{t-s}%
}\big)(x-E(t-s)y)\big|^{2}},
\end{split}
\label{eq:estimdexixjGamma}%
\end{equation}
where $c_{0}>0$ is a suitable constant only depending on the number $\nu>0$.
On the other hand, taking into account \eqref{LP 2.20}, for every $s<t$ we can
write
\begin{align*}
D_{0}\Big(\frac{1}{\sqrt{t-s}}\Big)(x-E(t-s)y)  &  =\Big[D_{0}\Big(\frac
{1}{\sqrt{t-s}}\Big)E(t-s)\Big](E(s-t)x-y)\\
&  =E(1)\Big[D_{0}\Big(\frac{1}{\sqrt{t-s}}\Big)(E(s-t)x-y)\Big].
\end{align*}
As a consequence, since $E(1)$ is non-singular, we get
\begin{equation}
e^{-c_{0}\big|D_{0}\big(\frac{1}{\sqrt{t-s}}\big)(x-E(t-s)y)\big|^{2}}\leq
e^{-c_{0}^{\prime}\big|D_{0}\big(\frac{1}{\sqrt{t-s}}\big)(E(s-t)x-y)\big|^{2}%
}, \label{eq:estimexpLP}%
\end{equation}
where $c_{0}^{\prime}>0$ is another constant only depending on $\nu$. Then, by
combining \eqref{eq:estimdexixjGamma} with \eqref{eq:estimexpLP}, we obtain
the following estimate:
\begin{align*}
&  \int_{\mathbb{R}^{N}\times(\tau,t)}|\partial_{x_{i}x_{j}}^{2}%
\Gamma(x,t;y,s)|\cdot\Vert E(s-t)x-y\Vert^{\alpha}\,dy\,ds\\
&  \quad\leq c_{0}\int_{\mathbb{R}^{N}\times(\tau,t)}\frac{1}{(t-s)^{Q/2+1}%
}e^{-c_{0}^{\prime}\big|D_{0}\big(\frac{1}{\sqrt{t-s}}%
\big)(E(s-t)x-y)\big|^{2}}\Vert E(s-t)x-y\Vert^{\alpha}\,dy\,ds\\
&  \quad=c_{0}\int_{\tau}^{t}\frac{1}{(t-s)^{Q/2+1}}\bigg(\int_{\mathbb{R}%
^{N}}e^{-c_{0}^{\prime}\big|D_{0}\big(\frac{1}{\sqrt{t-s}}%
\big)(E(s-t)x-y)\big|^{2}}\Vert E(s-t)x-y\Vert^{\alpha}\,dy\bigg)ds\\[0.1cm]
&  \quad=:(\bigstar).
\end{align*}
To proceed further, we perform in the $dy$-integral the change of variables
\[
y=E(s-t)x-D_{0}(\sqrt{t-s})z.
\]
Reminding that $\det(D_{0}(\lambda))=\lambda^{Q}$ for every $\lambda>0$ (see
\eqref{dilations}-\eqref{eq:defQhomdim}), and since the norm $\Vert\cdot\Vert$
is $D_{0}$-homogeneous of degree $1$, we get
\begin{align*}
(\bigstar)  &  =c_{0}\int_{\tau}^{t}\frac{1}{(t-s)^{1-\frac{\alpha}{2}}%
}\bigg(\int_{\mathbb{R}^{N}}e^{-c_{0}^{\prime}\left\vert z\right\vert ^{2}%
}\Vert z\Vert^{\alpha}\,dz\bigg)ds\\
&  =\frac{c_{0}}{\alpha}(t-\tau)^{\alpha/2}\bigg(\int_{\mathbb{R}^{N}%
}e^{-c_{0}^{\prime}\left\vert z\right\vert ^{2}}\Vert z\Vert^{\alpha
}\,dz\bigg).
\end{align*}
To complete the proof of \eqref{eq:integraldexixjconv} we only need to show
that the $dz$-integral is finite. To this end we observe that, by definition
of $\Vert\cdot\Vert$, we have
\begin{align*}
I  &  :=\int_{\mathbb{R}^{N}}e^{-c_{0}^{\prime}\left\vert z\right\vert ^{2}%
}\Vert z\Vert^{\alpha}\,dz=\int_{\mathbb{R}^{N}}e^{-c_{0}^{\prime}\left\vert
z\right\vert ^{2}}\Big(\sum_{i=1}^{N}|z_{i}|^{1/q_{i}}\Big)^{\alpha}\,dz\\
&  \leq c(\alpha)\sum_{i=1}^{N}\int_{\mathbb{R}^{N}}e^{-c_{0}^{\prime
}\left\vert z\right\vert ^{2}}|z_{i}|^{\alpha/q_{i}}\,dz\leq c(\alpha
)\sum_{i=1}^{N}\int_{\mathbb{R}^{N}}e^{-c_{0}^{\prime}\left\vert z\right\vert
^{2}}|z|^{\alpha/q_{i}}\,dz;
\end{align*}
from this, we immediately see that $I<\infty$, and the proof is complete.
\end{proof}

With Proposition \ref{Proposition Claim} at hand, we can now prove the
following theorem.

\begin{theorem}
\label{Thm repr formula derivatives}For $T>\tau>-\infty\ $and $\alpha\in
(0,1)$, let $u\in\mathcal{S}^{0}(\tau;T)$ be such that $\mathcal{L}u\in
C_{x}^{\alpha}(S_{T})$. Then, we have
\begin{equation}
\partial_{x_{i}x_{j}}^{2}u(x,t)=\int_{\mathbb{R}^{N}\times(\tau,t)}%
\partial_{x_{i}x_{j}}^{2}\Gamma(x,t;y,s)\big[\mathcal{L}%
u(E(s-t)x,s)-\mathcal{L}u(y,s)\big]\,dy\,ds, \label{repr formula u_xx}%
\end{equation}
for {every} $(x,t)\in S_{T}$ and every $1\leq i,j\leq q$.
\end{theorem}

\begin{proof}
We first observe that, since $\mathcal{L}u\in C_{x}^{\alpha}(S_{T})$, by
definition we have
\begin{equation}
|\mathcal{L}u(E(s-t)x,s)-\mathcal{L}u(y,s)|\leq|\mathcal{L}u|_{C_{x}^{\alpha
}(S_{T})}\cdot\Vert E(s-t)x-y\Vert^{\alpha}, \label{eq:LLuholderestim}%
\end{equation}
for every $x,y\in\mathbb{R}^{N}$ and every $s,t<T$. Thus, by Proposition
\ref{Proposition Claim} we get
\begin{align*}
&  \bigg|\int_{\mathbb{R}^{N}\times(\tau,t)}|\partial_{x_{i}x_{j}}^{2}%
\Gamma(x,t;y,s)|\cdot|\mathcal{L}u(E(s-t)x,s)-\mathcal{L}%
u(y,s)|\,dy\,ds\bigg|\\
&  \quad\leq|\mathcal{L}u|_{C_{x}^{\alpha}(S_{T})}\bigg|\int_{\mathbb{R}%
^{N}\times(\tau,t)}|\partial_{x_{i}x_{j}}^{2}\Gamma(x,t;y,s)|\cdot\Vert
E(s-t)x-y\Vert^{\alpha}\,dy\,ds\bigg|\\
&  \quad\leq\mathbf{c}\,|\mathcal{L}u|_{C_{x}^{\alpha}(S_{T})}\cdot
|t-\tau|^{\alpha/2}\quad\forall\,\,(x,t)\in S_{T}%
\end{align*}
(where $\mathbf{c}>0$ only depends on $\alpha$), and this shows that the
function
\[
g(x,t):=\int_{\mathbb{R}^{N}\times(\tau,t)}\partial_{x_{i}x_{j}}^{2}%
\Gamma(x,t;y,s)\big[\mathcal{L}u(E(s-t)x,s)-\mathcal{L}u(y,s)\big]\,dy\,ds
\]
is {well-defined} on $S_{T}$. We then turn to prove that $\partial_{x_{i}%
x_{j}}^{2}u=g$ {po\-int\-wi\-se in $S_{T}$} by an approximation argument. To
this end, we fix $0<\e\ll1$ and we define
\[
v_{\varepsilon}(x,t):=-\int_{\mathbb{R}^{N}\times(\tau,t-\e)}\partial_{x_{j}%
}\Gamma(x,t;\cdot)\,\mathcal{L}u\,dy\,ds.
\]
Now, arguing as in the proof of Corollary \ref{cor:firtderivrepr} and taking
into ac\-cou\-nt \eqref{eq:reprudexi}, we see that

\begin{itemize}
\item[(i)] $v_{\varepsilon}\in C(S_{T})$ and $v_{\varepsilon}\rightarrow
\partial_{x_{j}}u$ pointwise in $S_{T}$ as $\e\rightarrow0^{+}$;

\item[(ii)] $v_{\varepsilon}$ is continuously differentiable w.r.t.\thinspace
$x_{i}$ on $S_{T}$, and
\[
\partial_{x_{i}}v_{\varepsilon}(x,t)=-\int_{\mathbb{R}^{N}\times(\tau
,t-\e)}\partial_{x_{i}x_{j}}^{2}\Gamma(x,t;\cdot)\,\mathcal{L}u\,dy\,ds\quad
\forall\,\,(x,t)\in S_{T}.
\]

\end{itemize}

On the other hand, owing to Lemma \ref{lem:integralvanishinggamma}, we have
\begin{align*}
&  \partial_{x_{i}}v_{\varepsilon}(x,t)=-\int_{\mathbb{R}^{N}\times
(\tau,t-\e)}\partial_{x_{i}x_{j}}^{2}\Gamma(x,t;\cdot)\,\mathcal{L}u\,dy\,ds\\
&  \qquad=\int_{\mathbb{R}^{N}\times(\tau,t-\e)}\partial_{x_{i}x_{j}}%
^{2}\Gamma(x,t;y,s)\big[\mathcal{L}u(E(s-t)x,s)-\mathcal{L}%
u(y,s)\big]\,dy\,ds.
\end{align*}
As a consequence, by combining \eqref{eq:LLuholderestim} with Proposition
\ref{Proposition Claim} we obtain
\begin{align*}
&  |\partial_{x_{i}}v_{\varepsilon}(x,t)-g(x,t)|=\int_{\mathbb{R}^{N}%
\times(t-\e,t)}|\partial_{x_{i}x_{j}}^{2}\Gamma(x,t;\cdot)|\,|\mathcal{L}%
u(E(s-t)x,s)-\mathcal{L}u|\,dy\,ds\\
&  \qquad\leq|\mathcal{L}u|_{C_{x}^{\alpha}(S_{T})}\int_{\mathbb{R}^{N}%
\times(t-\e,t)}|\partial_{x_{i}x_{j}}^{2}\Gamma(x,t;\cdot)|\ \cdot\Vert
E(s-t)x-y\Vert^{\alpha}\,dy\,ds\\
&  \qquad\leq c\,|\mathcal{L}u|_{C_{x}^{\alpha}(S_{T})}\e^{\alpha/2}%
\qquad\text{uniformly for $(x,t)\in S_{T}$},
\end{align*}
from which we derive that $\partial_{x_{i}}v_{\varepsilon}\rightarrow g$
{uniformly on $S_{T}$ as $\e\rightarrow0^{+}$}. As in the proof of Corollary
\ref{cor:firtderivrepr} we then conclude that
\[
\partial_{x_{i}x_{j}}^{2}u=\partial_{x_{i}}(\partial_{x_{j}}u)=g\quad
\text{pointwise in $S_{T}$},
\]
and this gives \eqref{repr formula u_xx}.
\end{proof}

\subsection{Schauder estimates in space\label{sec schauder space model op}}

We now want to prove the following result:

\begin{theorem}
[Global Schauder estimates in space]\label{Thm Schauder space} Let
$T>\tau>-\infty$ and $\alpha\in(0,1)$. Then, there exists ${c}>0$, only
depending on $(T-\tau),\alpha,\nu,B$, such that
\begin{align}
&  \sum_{i,j=1}^{q}\Vert\partial_{x_{i}x_{j}}^{2}u\Vert_{C_{x}^{\alpha}%
(S_{T})}\leq c\,|\mathcal{L}u|_{C_{x}^{\alpha}(S_{T})}%
\label{eq:Schauderspaceaijt}\\[0.1cm]
&  \Vert Yu\Vert_{C_{x}^{\alpha}(S_{T})}\leq{c}\,\Vert\mathcal{L}u\Vert
_{C_{x}^{\alpha}(S_{T})}, \label{eq:SchauderspaceaijtDRIFT}%
\end{align}
{for every $u\in\mathcal{S}^{0}(\tau;T)$ with $\mathcal{L}u\in C_{x}^{\alpha
}(S_{T})$}.
\end{theorem}

The estimates in the above theorem will be generalized, in Section
\ref{Sec operators a(x,t)}, in the context of operators with coefficients
$a_{ij}(x,t)$; hence, the core of this section consists more in the
development of the \emph{tools} necessary to prove the above theorem, then in
the result itself. Actually, these tools will be useful also in the following
parts of the paper. Also, it is worth noting that the proof of global Schauder
estimates in the situation considered in this section is much more
straightforward than for coefficients also depending on $x$. However, note
that for the moment we do not prove global estimates on the lower order
derivatives $\partial_{x_{k}}u$ and on $u$ itself. \medskip

To prove Theorem \ref{Thm Schauder space}, we need the following auxiliary results.

\begin{theorem}
[Cancellation property of the singular kernel]%
\label{Thm cancellation property} There exists a con\-stant $c>0$ such that,
for every $1\leq i,j\leq q$, one has the estimate
\begin{equation}
I_{r,\tau}(x,t):=\int_{\tau}^{t}\bigg|\int_{\{y\in\mathbb{R}^{N}%
:\,d((x,t),(y,s))\,\geq\,r\}}\partial_{x_{i}x_{j}}^{2}\Gamma
(x,t;y,s)\,dy\,\bigg|\,ds\leq c, \label{eq:cancelprop}%
\end{equation}
{for every $x\in\mathbb{R}^{N}$, $\tau<t$ and $r>0$}.
\end{theorem}

\begin{proof}
Let $x,\tau,t$ and $r$ be as in the statement. We then distinguish two cases.
\medskip

\textsc{Case I:} $t-\tau>r^{2}$. In this case we first observe that, taking
into account the e\-xpli\-cit expression of the quasi-distance $d$ given in
\eqref{eq:explicitd}, we have
\[
d((x,t),(y,s))=\Vert x-E(t-s)y\Vert+\sqrt{t-s}\geq\sqrt{t-s}\geq r,
\]
for every $\tau<s<t-r^{2}$; thus, by Lemma \ref{lem:integralvanishinggamma} we
can write
\begin{align*}
I_{r,\tau}(x,t)  &  =\int_{\tau}^{t-r^{2}}\bigg|\int_{\mathbb{R}^{N}}%
\partial_{x_{i}x_{j}}^{2}\Gamma(x,t;y,s)\,dy\,\bigg|\,ds\\[0.1cm]
&  \qquad+\int_{t-r^{2}}^{t}\bigg|\int_{\{y\in\mathbb{R}^{N}%
:\,d((x,t),(y,s))\,\geq\,r\}}\partial_{x_{i}x_{j}}^{2}\Gamma
(x,t;y,s)\,dy\,\bigg|\,ds\\[0.1cm]
&  =\int_{t-r^{2}}^{t}\bigg|\int_{\{y\in\mathbb{R}^{N}:\,d((x,t),(y,s))\,\geq
\,r\}}\partial_{x_{i}x_{j}}^{2}\Gamma(x,t;y,s)\,dy\,\bigg|\,ds=:J_{r,\tau
}(x,t).
\end{align*}
In order to prove \eqref{eq:cancelprop}, we then turn to bound the integral
$J_{r}(x,t)$. \vspace*{0.05cm}

First of all, by combining the global upper estimates for $\partial
_{x_{i}x_{j}}^{2}\Gamma$ in Theorem \ref{Thm bound derivatives} with
\eqref{eq:Gammaalfaconvolution}-\eqref{eq.exprGammaalfa} (see also
\eqref{eq:estimdexixjGamma} in the proof of Prop.\thinspace
\ref{Proposition Claim}), we get
\begin{align*}
J_{r,\tau}(x,t)  &  \leq\int_{t-r^{2}}^{t}\bigg(\int_{\{y\in\mathbb{R}%
^{N}:\,d((x,t),(y,s))\geq r\}}|\partial_{x_{i}x_{j}}^{2}\Gamma
(x,t;y,s)|\,dy\bigg)ds\\
&  \leq c_{0}\int_{t-r^{2}}^{t}\frac{ds}{(t-s)^{Q/2+1}}\times\\
&  \qquad\,\,\times\bigg(\int_{\{y\in\mathbb{R}^{N}:\,d((x,t),(y,s))\geq
r\}}e^{-c_{0}\big|D_{0}\big(\frac{1}{\sqrt{t-s}}\big)(x-E(t-s)y)\big|^{2}%
}\,dy\bigg)=:(\bigstar),
\end{align*}
where $c_{0}>0$ is a constant only de\-pen\-ding on $\nu$. From this,
recalling \eqref{eq:explicitd} and using the change of variables
\begin{equation}
y=E(s-t)x-E(s-t)z \label{eq:changeydyCaseI}%
\end{equation}
in the $dy$-integral, we obtain
\begin{align*}
(\bigstar)  &  =c_{0}\int_{t-r^{2}}^{t}\frac{|\det(E(s-t))|}{(t-s)^{Q/2+1}%
}\bigg(\int_{\{z\in\mathbb{R}^{N}:\,\Vert z\Vert+\sqrt{t-s}\geq r\}}%
\!\!\!e^{-c_{0}\big|D_{0}\big(\frac{1}{\sqrt{t-s}}\big)z\big|^{2}%
}\,dz\bigg)ds\\[0.1cm]
&  (\text{since $\det(E(s-t))=e^{(t-s)\det B}=1$, see \eqref{B}})\\
&  =c_{0}\int_{\{t-r^{2}\leq s\leq t,\,\Vert z\Vert+\sqrt{t-s}\geq r\}}%
\frac{1}{(t-s)^{Q/2+1}}e^{-c_{0}\big|D_{0}\big(\frac{1}{\sqrt{t-s}%
}\big)z\big|^{2}}\,dz\,ds=:(2\bigstar).
\end{align*}
To proceed further, we now perform another change of variables, this time
involving {both $z$ and $s$}: taking into account the $D_{0}$-homogeneity of
$\Vert\cdot\Vert$, we set
\begin{equation}
(z,s)=\big(D_{0}(r)w,t-r^{2}\sigma\big). \label{eq:changezsCaseI}%
\end{equation}
Recalling that $\det(D_{0}(r))=r^{Q}$, we then get
\[
(2\bigstar)=c_{0}\int_{0}^{1}\frac{1}{\sigma^{Q/2+1}}\bigg(\int_{\{w\in
\mathbb{R}^{N}:\,\Vert w\Vert+\sqrt{\sigma}\geq1\}}e^{-c_{0}\big|D_{0}%
\big(\frac{1}{\sqrt{\sigma}}\big)w\big|^{2}}\,dw\bigg)d\sigma\equiv
c_{0}\,\mathbf{J}.
\]
Since the integral $\mathbf{J}$ is a constant, to complete the proof of
\eqref{eq:cancelprop} in this case it suffices to show that $\mathbf{J}%
<\infty$. To this end, we perform yet a\-no\-ther change of variables in the
$dw$-integral: setting
\[
w=D_{0}(\sqrt{\sigma})u,
\]
and taking into account that $\det(D_{0}(\sqrt{\sigma}))=\sigma^{Q/2}$, we
obtain
\begin{align*}
\mathbf{J}  &  =\int_{0}^{1}\frac{1}{\sigma}\bigg(\int_{\{u\in\mathbb{R}%
^{N}:\,\Vert u\Vert\geq\frac{1}{\sqrt{\sigma}}-1\}}e^{-c_{0}\,|u|^{2}%
}\,du\bigg)d\sigma\\[0.1cm]
&  =\int_{0}^{1/4}\frac{h(\sigma)}{\sigma}\,d\sigma++\int_{1/4}^{1}%
\frac{h(\sigma)}{\sigma}\,d\sigma=:\mathbf{J}_{1}+\mathbf{J}_{2},
\end{align*}
where we have introduced the shorthand notation
\[
h(\sigma):=\int_{\{u\in\mathbb{R}^{N}:\,\Vert u\Vert\geq\frac{1}{\sqrt{\sigma
}}-1\}}e^{-c_{0}\,|u|^{2}}\,du.
\]
We then turn to show that {both the integrals} $\mathbf{J}_{1},\,\mathbf{J}%
_{2}$ are finite. As to $\mathbf{J}_{1}$ we first notice that, since
{$\theta_{N}\Vert u\Vert\leq|u|$} when $\Vert u\Vert\geq1$ (here, $\theta
_{N}>0$ is a constant only depending on the dimension $N$), and since
\[
\frac{1}{\sqrt{\sigma}}-1\geq1\quad\text{when $0<\sigma\leq\frac{1}{4}$},
\]
we have the following estimate on the function $h$:
\begin{align*}
h(\sigma)  &  \leq\int_{\{u\in\mathbb{R}^{N}:\,|u|\geq\theta_{N}(\frac
{1}{\sqrt{\sigma}}-1)\}}e^{-c_{0}\,|u|^{2}}\,du=\omega_{N}\int_{\theta
_{N}(\frac{1}{\sqrt{\sigma}}-1)}^{+\infty}e^{-c_{0}\rho^{2}}\rho^{N-1}%
\,d\rho\\[0.1cm]
&  (\text{since $e^{-c_{0}\rho^{2}}\rho^{N-1}\leq\gamma\rho e^{-\frac{c_{0}%
}{2}\rho^{2}}$ when $N\geq2$})\\
&  =\gamma\,\omega_{N}\,\int_{\theta_{N}(\frac{1}{\sqrt{\sigma}}-1)}^{+\infty
}\rho e^{-\frac{c_{0}}{2}\rho^{2}}\,d\rho=c_{N}e^{-\frac{c_{0}\theta_{N}^{2}%
}{2}(\frac{1}{\sqrt{\sigma}}-1)^{2}},
\end{align*}
where $c_{N}:=\gamma\,\omega_{N}/c_{0}$. As a consequence, we easily obtain
\[
\mathbf{J}_{1}\leq c_{N}\int_{0}^{1/4}\frac{1}{\sigma}e^{-\frac{c_{0}%
\theta_{N}^{2}}{2}(\frac{1}{\sqrt{\sigma}}-1)^{2}}\,d\sigma<\infty.
\]
As to $\mathbf{J}_{2}$, instead, taking into account that the map $u\mapsto
e^{-c_{0}|u|^{2}}$ is integrable on $\mathbb{R}^{N}$, we immediately get
\[
\mathbf{J}_{2}\leq\int_{1/4}^{1}\frac{1}{\tau}\bigg(\int_{\mathbb{R}^{N}%
}e^{-c_{0}|u|^{2}}\,du\bigg)d\sigma\leq4\int_{\mathbb{R}^{N}}e^{-c_{0}|u|^{2}%
}\,du<\infty.
\]
Gathering these facts, we then conclude that $\mathbf{J}<\infty$, as desired.
\medskip

\textsc{Case II:} $t-\tau\leq r^{2}$. In this case, using once again the
global upper estimates for $\partial_{x_{i}x_{j}}^{2}\Gamma$ in Theorem
\ref{Thm bound derivatives}, and taking into account
\eqref{eq:Gammaalfaconvolution}-\eqref{eq.exprGammaalfa}, we get
\begin{align*}
I_{r,\tau}(x,t)  &  \leq\int_{\tau}^{t}\bigg(\int_{\{y\in\mathbb{R}%
^{N}:\,d((x,t),(y,s))\geq r\}}|\partial_{x_{i}x_{j}}^{2}\Gamma
(x,t;y,s)|\,dy\bigg)ds\\
&  \leq c_{0}\int_{\tau}^{t}\frac{ds}{(t-s)^{Q/2+1}}\times\\
&  \qquad\,\,\times\bigg(\int_{\{y\in\mathbb{R}^{N}:\,d((x,t),(y,s))\geq
r\}}e^{-c_{0}\big|D_{0}\big(\frac{1}{\sqrt{t-s}}\big)(x-E(t-s)y)\big|^{2}%
}\,dy\bigg)=:(\star).
\end{align*}
Starting from this estimate, and performing the change of variables
\eqref{eq:changeydyCaseI}-\eqref{eq:changezsCaseI}, we then obtain
\begin{align*}
(\star)  &  =c_{0}\int_{\tau}^{t}\frac{1}{(t-s)^{Q/2+1}}\bigg(\int%
_{\{z\in\mathbb{R}^{N}:\,\Vert z\Vert+\sqrt{t-s}\geq r\}}\!\!\!e^{-c_{0}%
\big|D_{0}\big(\frac{1}{\sqrt{t-s}}\big)z\big|^{2}}\,dz\bigg)ds\\[0.1cm]
&  =c_{0}\int_{0}^{\frac{t-\tau}{r^{2}}}\frac{1}{\sigma^{Q/2+1}}%
\bigg(\int_{\{w\in\mathbb{R}^{N}:\,\Vert w\Vert+\sqrt{\sigma}\geq1\}}%
e^{-c_{0}\big|D_{0}\big(\frac{1}{\sqrt{\sigma}}\big)w\big|^{2}}%
\,dw\bigg)d\sigma=:(2\star).
\end{align*}
Now, since are assuming that $t-\tau\leq r^{2}$, we have
\begin{equation}
(2\star)\leq c_{0}\int_{0}^{1}\frac{1}{\sigma^{Q/2+1}}\bigg(\int%
_{\{w\in\mathbb{R}^{N}:\,\Vert w\Vert+\sqrt{\sigma}\geq1\}}e^{-c_{0}%
\big|D_{0}\big(\frac{1}{\sqrt{\sigma}}\big)w\big|^{2}}\,dw\bigg)d\sigma
=c_{0}\,\mathbf{J}, \label{eq:touseperconcludere}%
\end{equation}
where $\mathbf{J}$ is {the same integral} considered in the previous case; as
a consequence, since we have already recognized that $\mathbf{J}<\infty$, from
\eqref{eq:touseperconcludere} we immediately derive \eqref{eq:cancelprop} also
in this case, and the proof is complete.
\end{proof}

\begin{theorem}
[H\"{o}lder continuity of singular integrals]%
\label{Thm holder singular integrals} For $T>\tau>-\infty$ and $\alpha
\in(0,1)$, let us introduce the function space
\[
C_{x}^{\alpha}(\tau;T):=\{f\in C_{x}^{\alpha}(S_{T}):\,\text{$f(x,t)=0$ for
every $t\leq\tau$}\},
\]
and define, on this space $C_{x}^{\alpha}(\tau;T)$, the linear operator
\[
f\mapsto T_{ij}f(x,t):=\int_{\mathbb{R}^{N}\times(\tau,t)}\partial_{x_{i}%
x_{j}}^{2}\Gamma(x,t;y,s)\,\big[f(E(s-t)x,s)-f(y,s)\big]\,dy\,ds.
\]
Then, there exists a constant $c>0$, depending on $(T-\tau)$ and $\alpha$,
such that
\begin{equation}
\Vert T_{ij}f\Vert_{C_{x}^{\alpha}(S_{T})}\leq c\,|f|_{C_{x}^{\alpha}(S_{T}%
)}\quad\text{for every $f\in C_{x}^{\alpha}(\tau;T)$}.
\label{eq:toproveboundTfSI}%
\end{equation}

\end{theorem}

\begin{proof}
Let $f\in C_{x}^{\alpha}(\tau;T)$ be arbitrarily fixed. Since $f(\cdot
,t)\equiv0$ for every $t\leq\tau$, we have $T_{ij}f(x,t)=0$ for every
$x\in\mathbb{R}^{N}$ and $t\leq\tau$. Thus, we derive that
\[
\Vert T_{ij}f\Vert_{C_{x}^{\alpha}(S_{T})}=\Vert T_{ij}f\Vert_{C_{x}^{\alpha
}(\Omega)},\quad\text{where $\Omega:=\mathbb{R}^{N}\times(\tau,T)$}.
\]
Hence, to prove \eqref{eq:toproveboundTfSI} it suffices to study
$T_{ij}f\left(  x,t\right)  $ for $(x,t)\in\Omega$.

First of all, owing to Pro\-po\-si\-tion \ref{Proposition Claim}, for every
$(x,t)\in\Omega$ we have
\begin{align*}
|T_{ij}f(x,t)|  &  \leq\int_{\mathbb{R}^{N}\times(\tau,t)}|\partial
_{x_{i}x_{j}}^{2}\Gamma(x,t;y,s)|\cdot|f(E(s-t)x,s)-f(y,s)|\,dy\,ds\\
&  \leq|f|_{C_{x}^{\alpha}(S_{T})}\,\int_{\mathbb{R}^{N}\times(\tau
,t)}|\partial_{x_{i}x_{j}}^{2}\Gamma(x,t;y,s)|\cdot\Vert E(s-t)x-y\Vert
^{\alpha}\,dy\,ds\\
&  \leq c\,|f|_{C_{x}^{\alpha}(S_{T})}\cdot(t-\tau)^{\alpha/2}\leq
c\,|f|_{C_{x}^{\alpha}(S_{T})}\cdot(T-\tau)^{\alpha/2},
\end{align*}
where $c>0$ is a constant only depending on $\alpha$. From this, we derive
\begin{equation}
\Vert T_{ij}f\Vert_{L^{\infty}(S_{T})}\leq c(T-\tau,\alpha)\,|f|_{C_{x}%
^{\alpha}(S_{T})}. \label{eq:Tfbound}%
\end{equation}
On the other hand, if $(x_{1},t),\,(x_{2},t)\in\Omega$ are such that $\Vert
x_{1}-x_{2}\Vert\geq1$, thanks to estimate \eqref{eq:Tfbound} we also obtain
the following bound
\[
|T_{ij}f(x_{1},t)-T_{ij}f(x_{2},t)|\leq2\Vert Tf\Vert_{L^{\infty}(S_{T})}\leq
c(T-\tau,\alpha)|f|_{C_{x}^{\alpha}(S_{T})}\Vert x_{1}-x_{2}\Vert^{\alpha},
\]
Thus, to prove \eqref{eq:toproveboundTfSI} we are left to show that
\begin{equation}
\begin{gathered} |T_{ij}f(x_1,t)-T_{ij}f(x_2,t)|\leq c(T-\tau,\alpha)\|x_1-x_2\|^\alpha \\[0.05cm] \text{for every $(x_1,t),\,(x_2,t)\in \Omega$ with $\|x_1-x_2\|<1$}. \end{gathered} \label{eq:toprovex1x2leqone}%
\end{equation}
To this end, taking into account the definition of $Tf$, we write
\begin{equation}%
\begin{split}
&  T_{ij}f(x_{1},t)-T_{ij}f(x_{2},t)\\
&  \qquad=\int_{\mathbb{R}^{N}\times(\tau,t)}\Big\{\partial_{x_{i}x_{j}}%
^{2}\Gamma(x_{1},t;y,s)\big[f(E(s-t)x_{1},s)-f(y,s)\big]\\
&  \qquad\qquad\quad-\partial_{x_{i}x_{j}}^{2}\Gamma(x_{2}%
,t;y,s)\big[f(E(s-t)x_{2},s)-f(y,s)\big]\Big\}\,dy\,ds\\
&  \qquad=\int_{\{(y,s):\,d((x_{2},t),(y,s))\geq4\bd{\kappa}\rho\}}%
\{\cdots\}\,dy\,ds\\
&  \qquad\qquad\quad+\int_{\{(y,s):\,d((x_{2},t),(y,s))<4\bd{\kappa}\rho
\}}\{\cdots\}\,dy\,ds\\[0.1cm]
&  \qquad=:\mathrm{A}_{1}+\mathrm{A}_{2},
\end{split}
\label{eq:spliTijSTART}%
\end{equation}
where $\bd\kappa>0$ is as in \eqref{quasitriangle}-\eqref{quasisymmetric} and
\[
\rho:=d((x_{2},t),(x_{1},t))=\Vert x_{1}-x_{2}\Vert.
\]
We then turn to estimate $\mathrm{A}_{1}$ and $\mathrm{A}_{2}$. \medskip

\noindent-\thinspace\thinspace\textsc{Estimate of $\mathrm{A}_{1}$.} To begin
with, we write $\mathrm{A}_{1}$ as follows:
\begin{align*}
\mathrm{A}_{1}  &  =\int_{\{(y,s):\,d((x_{2},t),(y,s))\geq4\bd{\kappa}\rho
\}}\Big\{\big[f(E(s-t)x_{1},s)-f(y,s)\big]\times\\[0.1cm]
&  \qquad\qquad\times\big[\partial_{x_{i}x_{j}}^{2}\Gamma(x_{1}%
,t;y,s)-\partial_{x_{i}x_{j}}^{2}\Gamma(x_{2}%
,t;y,s)\big]\Big\}\,dy\,ds\\[0.15cm]
&  \qquad+\int_{\{(y,s):\,d((x_{2},t),(y,s))\geq4\bd{\kappa}\rho
\}}\Big\{\partial_{x_{i}x_{j}}^{2}\Gamma(x_{2},t;y,s)\times\\
&  \qquad\qquad\times\big[f(E(s-t)x_{1},s)-f(E(s-t)x_{2}%
,s)\big]\Big\}\,dy\,ds\\[0.1cm]
&  =:\mathrm{A}_{11}+\mathrm{A}_{12}.
\end{align*}
\medskip

\emph{Estimate of $\mathrm{A}_{11}$}. First of all we observe that, owing to
the mean value i\-ne\-qua\-li\-ties in Theorem \ref{Thm mean value} (and
taking into account the definition of $\rho$), we have
\begin{align*}
&  \big|\partial_{x_{i}x_{j}}^{2}\Gamma(x_{1},t;y,s)-\partial_{x_{i}x_{j}}%
^{2}\Gamma(x_{2},t;y,s)\big|\\
&  \qquad\leq c\,\frac{d((x_{2},t),(x_{1},t))}{d((x_{2},t),(y,s))^{Q+3}%
}=c\,\frac{\Vert x_{1}-x_{2}\Vert}{d((x_{2},t),(y,s))^{Q+3}},
\end{align*}
for every $(y,s)\in\Omega$ such that $d((x_{2},t),(y,s))\geq4\bd\kappa\rho$.
Moreover, using the explicit expression of $d$ in \eqref{eq:explicitd} and the
quasi-symmetry property \eqref{quasisymmetric}, we get
\begin{align*}
\big|f(E(s-t)x_{1},s)-f(y,s)\big|  &  \leq|f|_{C_{x}^{\alpha}(S_{T})}\Vert
E(s-t)x_{1}-y\Vert^{\alpha}\\
&  \leq|f|_{C_{x}^{\alpha}(S_{T})}d((y,s),(x_{1},t))^{\alpha}\\
&  \leq\bd\kappa^{\alpha}\,|f|_{C_{x}^{\alpha}(S_{T})}d((x_{1}%
,t),(y,s))^{\alpha},
\end{align*}
where we have also used the fact that $f\in C_{x}^{\alpha}(\tau;T)$. Hence, by
combining these estimates and by using Lemma \ref{lem:equivalentd}, we get
\begin{equation}%
\begin{split}
&  \big|f(E(s-t)x_{1},s)-f(y,s)\big|\cdot\big|\partial_{x_{i}x_{j}}^{2}%
\Gamma(x_{1},t;y,s)-\partial_{x_{i}x_{j}}^{2}\Gamma(x_{2},t;y,s)\big|\\
&  \qquad\quad\leq c\,|f|_{C_{x}^{\alpha}(S_{T})}\Vert x_{1}-x_{2}\Vert
\cdot\frac{d((x_{1},t),(y,s))^{\alpha}}{d((x_{2},t),(y,s))^{Q+3}}\\
&  \qquad\quad\leq c\,|f|_{C_{x}^{\alpha}(S_{T})}\Vert x_{1}-x_{2}\Vert
\cdot\frac{1}{d((x_{2},t),(y,s))^{Q+3-\alpha}},
\end{split}
\label{eq:stimaconLemmatheta}%
\end{equation}
for every $(y,s)\in\mathbb{R}^{N}\times(\tau,t)$ satisfying $d((x_{2}%
,t),(y,s)\geq4\bd\kappa\rho>2\bd\kappa\rho$. Owing to
\eqref{eq:stimaconLemmatheta}, and exploiting \eqref{doubling >R} in Lemma
\ref{lem:doubling}, we finally obtain
\begin{equation}%
\begin{split}
|\mathrm{A}_{11}|  &  \leq c\,|f|_{C_{x}^{\alpha}(S_{T})}\,\Vert x_{1}%
-x_{2}\Vert\int_{\{\eta:\,d(\xi,\eta)\geq4\bd{\kappa}\rho\}}\frac{1}%
{d(\xi,\eta)^{Q+3-\alpha}}\,d\eta\\[0.1cm]
&  \leq c\,|f|_{C_{x}^{\alpha}(S_{T})}\,\Vert x_{1}-x_{2}\Vert\cdot
\rho^{\alpha-1}=c\,|f|_{C_{x}^{\alpha}(S_{T})}\,\Vert x_{1}-x_{2}\Vert
^{\alpha},
\end{split}
\label{eq:estimA11final}%
\end{equation}
where $c>0$ is a constant only depending on $\alpha$. \medskip

\emph{Estimate of $\mathrm{A}_{12}$}. First of all, using once again the fact
that $f\in C_{x}^{\alpha}(S_{T})$, jointly with Lemma \ref{Lemma E(s)x}, we
can bound the integral $\mathrm{A}_{12}$ as follows:
\begin{equation}%
\begin{split}
|\mathrm{A}_{12}|  &  \leq\int_{\tau}^{t}\big|f(E(s-t)x_{1},s)-f(E(s-t)x_{2}%
,s)\big|\cdot\mathcal{J}(s)\,ds\\
&  \leq|f|_{C_{x}^{\alpha}(S_{T})}\int_{\tau}^{t}\Vert E(s-t)(x_{1}%
-x_{2})\Vert^{\alpha}\cdot\mathcal{J}(s)\,ds\\
&  \leq c\,|f|_{C_{x}^{\alpha}(S_{T})}\int_{\tau}^{t}\big(\Vert x_{1}%
-x_{2}\Vert+\sqrt{t-s}\big)^{\alpha}\cdot\mathcal{J}(s)\,ds,
\end{split}
\label{eq:tostartperdistinguere}%
\end{equation}
where $c>0$ is an absolute constant and
\[
\mathcal{J}(s):=\bigg|\int_{\{y\in\mathbb{R}^{N}:\,d((x_{2},t),(y,s))\geq
4\bd{\kappa}\rho\}}\partial_{x_{i}x_{j}}^{2}\Gamma(x_{2},t;y,s)\,dy\bigg|.
\]
We now distinguish two cases, according to the value of $\theta:=t-16\bd\kappa
^{2}\rho^{2}$. \medskip

(i)\thinspace\thinspace$\theta>\tau$. In this case, we start from
\eqref{eq:tostartperdistinguere} and we write
\begin{equation}%
\begin{split}
|\mathrm{A}_{12}|  &  \leq c\,|f|_{C_{x}^{\alpha}(S_{T})}\int_{\tau}^{\theta
}\big(\Vert x_{1}-x_{2}\Vert+\sqrt{t-s}\big)^{\alpha}\cdot\mathcal{J}(s)\,ds\\
&  \qquad+c\,|f|_{C_{x}^{\alpha}(S_{T})}\int_{\theta}^{t}\big(\Vert
x_{1}-x_{2}\Vert+\sqrt{t-s}\big)^{\alpha}\cdot\mathcal{J}(s)\,ds.
\end{split}
\label{eq:splitA12ABM}%
\end{equation}
We now observe that, when $\theta\leq s\leq t$, we have $0\leq t-s\leq
16\bd\kappa^{2}\rho^{2}$; thus, by using the {cancellation property of
$\mathcal{J}$} in Theorem \ref{Thm cancellation property}, we get
\begin{equation}%
\begin{split}
&  \int_{\theta}^{t}\big(\Vert x_{1}-x_{2}\Vert+\sqrt{t-s}\big)^{\alpha}%
\cdot\mathcal{J}(s)\,ds\\
&  \qquad\leq(1+4\bd\kappa)^{\alpha}\Vert x_{1}-x_{2}\Vert^{\alpha}%
\int_{\theta}^{t}\mathcal{J}(s)\,ds\\
&  \qquad\leq(1+4\bd\kappa)^{\alpha}\,\Vert x_{1}-x_{2}\Vert^{\alpha}%
\int_{\tau}^{t}\mathcal{J}(s)\,ds\\[0.1cm]
&  \qquad=(1+4\bd\kappa)^{\alpha}\,\Vert x_{1}-x_{2}\Vert^{\alpha}\cdot
I_{4\bd\kappa\rho,\tau}(x_{2},t)\\[0.1cm]
&  \qquad\leq c\,\Vert x_{1}-x_{2}\Vert^{\alpha},
\end{split}
\label{eq:estimintegralAM}%
\end{equation}
where $c>0$ is a suitable constant only depending on $\alpha$. \vspace
*{0.05cm}

On the other hand, when $\tau\leq s<\theta$, by \eqref{eq:explicitd} we infer
that
\[
d((x_{2},t),(y,s))\geq\sqrt{t-s}\geq4\bd\kappa\rho\quad\forall\,\,y\in
\mathbb{R}^{N};
\]
as a consequence, from Lemma \ref{lem:integralvanishinggamma} we obtain
\begin{equation}%
\begin{split}
&  \int_{\tau}^{\theta}\big(\Vert x_{1}-x_{2}\Vert+\sqrt{t-s}\big)^{\alpha
}\cdot\mathcal{J}(s)\,ds\\
&  \qquad=\int_{\tau}^{\theta}\big(\Vert x_{1}-x_{2}\Vert+\sqrt{t-s}%
\big)^{\alpha}\cdot\bigg|\int_{\mathbb{R}^{N}}\partial_{x_{i}x_{j}}^{2}%
\Gamma(\xi_{1};y,s)\,dy\bigg|\,ds=0.
\end{split}
\label{eq:integralBMzero}%
\end{equation}
Summing up, by combining \eqref{eq:estimintegralAM}-\eqref{eq:integralBMzero}
with \eqref{eq:splitA12ABM}, we conclude that
\begin{equation}
|A_{12}|\leq c\,|f|_{C_{x}^{\alpha}(S_{T})}\,\Vert x_{1}-x_{2}\Vert^{\alpha},
\label{eq:estimA12Casei}%
\end{equation}
for a suitable constant $c>0$ only depending on $\alpha$. \medskip

(ii)\thinspace\thinspace$\theta\leq\tau$. In this case, starting from
\eqref{eq:tostartperdistinguere} and using once again the cancellation
property of $\mathcal{J}$ in Theorem \ref{Thm cancellation property}, we
immediately get
\begin{equation}%
\begin{split}
|\mathrm{A}_{12}|  &  \leq c\,|f|_{C_{x}^{\alpha}(S_{T})}\int_{\tau}%
^{t}\big(\Vert x_{1}-x_{2}\Vert+\sqrt{t-\tau}\big)^{\alpha}\cdot
\mathcal{J}(s)\,ds\\
&  \leq c\,|f|_{C_{x}^{\alpha}(S_{T})}\Vert x_{1}-x_{2}\Vert^{\alpha}%
\int_{\tau}^{t}\mathcal{J}(s)\,ds\\[0.1cm]
&  \leq c\,|f|_{C_{x}^{\alpha}(S_{T})}\Vert x_{1}-x_{2}\Vert^{\alpha}\cdot
I_{4\bd\kappa\rho,\tau}(x_{2},t)\\[0.1cm]
&  \leq c\,|f|_{C_{x}^{\alpha}(S_{T})}\Vert x_{1}-x_{2}\Vert^{\alpha},
\end{split}
\label{eq:estimA12Caseii}%
\end{equation}
where $c>0$ is another constant only depending on $\alpha$. \vspace*{0.05cm}

All in all, by combining \eqref{eq:estimA11final} with
\eqref{eq:estimA12Casei}-\eqref{eq:estimA12Caseii}, we conclude that
\begin{equation}
\label{eq:estimA1FINAL}|\mathrm{A}_{1}| \leq c\,|f|_{C^{\alpha}_{x}(S_{T}%
)}\|x_{1}-x_{2}\|^{\alpha},
\end{equation}
for a suitable constant $c > 0$ {only depending on $\alpha$}. \medskip

\noindent-\thinspace\thinspace\textsc{Estimate of $A_{2}$.} We first observe
that, since $f\in C_{x}^{\alpha}(\tau;T)$, one has
\begin{equation}
|\mathrm{A}_{2}|\leq|f|_{C_{x}^{\alpha}(S_{T})}\cdot\big(\mathrm{A}%
_{21}+\mathrm{A}_{22}\big), \label{eq:estimA2splitted}%
\end{equation}
where, for $k=1,2$, we have introduced the notation
\[
\mathrm{A}_{2k}:=\int_{\{(y,s):\,d((x_{2},t),(y,s))<4\bd{\kappa}\rho
\}}|\partial_{x_{i}x_{j}}^{2}(x_{k},t;y,s)\Gamma|\cdot\Vert E(s-t)x_{k}%
-y\Vert^{\alpha}\,dy\,ds.
\]
We then proceed by estimating the two integrals $\mathrm{A}_{21}%
,\,\mathrm{A}_{22}$ separately. \vspace*{0.05cm}

\emph{Estimate of $\mathrm{A}_{21}$}. First of all, by using the estimates for
$\partial_{x_{i}x_{j}}^{2}\Gamma$ gi\-ven in Theo\-rem
\ref{Thm bound derivatives}, jointly with \eqref{eq:explicitd}, we get
\[%
\begin{split}
\mathrm{A}_{21}  &  \leq c\,\int_{\{(y,s):\,d((x_{2}%
,t),(y,s))<4\bd{\kappa}\rho\}}\frac{\Vert E(s-t)x_{1}-y\Vert^{\alpha}%
}{d((x_{1},t),(y,s))^{Q+2}}\,dy\,ds\\
&  \leq c\,\int_{\{(y,s):\,d((x_{2},t),(y,s))<4\bd{\kappa}\rho\}}%
\frac{d((y,s),(x_{1},t))^{\alpha}}{d((x_{1},t),(y,s))^{Q+2}}\,dy\,ds\\
&  \leq c\,\int_{\{(y,s):\,d((x_{2},t),(y,s))<4\bd{\kappa}\rho\}}\frac
{1}{d((x_{1},t),(y,s))^{Q+2-\alpha}}\,dy\,ds=:(\bigstar).
\end{split}
\]
On the other hand, by the quasi-triangular inequality \eqref{quasitriangle},
we have
\begin{equation}%
\begin{split}
d((x_{1},t),(y,s))  &  \leq\bd\kappa\big(d((x_{1},t),(x_{2},t))+d((y,s),(x_{2}%
,t))\big)\\
&  \leq\bd\kappa^{2}\big(d((x_{2},t),(x_{1},t))+d((x_{2},t),(y,s))\big)\\
&  =\bd\kappa^{2}(1+4\bd\kappa)\rho,
\end{split}
\label{eq:triangularineq}%
\end{equation}
for every $(y,s)\in\mathbb{R}^{N+1}$ such that $d((x_{2}%
,t),(y,s))<4\bd{\kappa}\rho$. On account of \eqref{eq:triangularineq}, and
exploiting \eqref{doubling <R} in Lemma \ref{lem:doubling}, we finally obtain
\begin{equation}%
\begin{split}
(\bigstar)  &  \leq c\,\int_{\{(y,s):\,d((x_{1},t),(y,s))<\bd\kappa
^{2}(1+4\bd\kappa)\rho\}}\frac{1}{d((x_{1},t),(y,s))^{Q+2-\alpha}}\,dy\,ds\\
&  =c\,\int_{\{\eta:\,d(\xi,\eta)<\bd\kappa^{2}(1+4\bd\kappa)\rho\}}\frac
{1}{d(\xi,\eta)^{Q+2-\alpha}}\,d\eta\\[0.2cm]
&  \leq c\rho^{\alpha}=c\,\Vert x_{1}-x_{2}\Vert^{\alpha}.
\end{split}
\label{eq:estimA21final}%
\end{equation}

\emph{Estimate of $\mathrm{A}_{22}$}. Using once again the estimates for
$\partial_{x_{i}x_{j}}^{2}\Gamma$ in Theorem \ref{Thm bound derivatives},
together with \eqref{quasisymmetric}-\eqref{eq:explicitd} and
\eqref{doubling <R} in Lemma \ref{lem:doubling}, we readily obtain
\begin{equation}%
\begin{split}
\mathrm{A}_{22}  &  \leq c\,\int_{\{(y,s):\,d((x_{2}%
,t),(y,s))<4\bd{\kappa}\rho\}}\frac{\Vert E(s-t)x_{2}-y\Vert^{\alpha}%
}{d((x_{2},t),(y,s))^{Q+2}}\,dy\,ds\\
&  \leq c\,\int_{\{(y,s):\,d((x_{2},t),(y,s))<4\bd{\kappa}\rho\}}%
\frac{d((y,s),(x_{2},t))^{\alpha}}{d((x_{2},t),(y,s))^{Q+2}}\,dy\,ds\\
&  =c\,\int_{\{\eta:\,d(\xi,\eta)<4\bd{\kappa}\rho\}}\frac{1}{d(\xi
,\eta)^{Q+2-\alpha}}\,d\eta\\[0.2cm]
&  \leq c\rho^{\alpha}=c\,\Vert x_{1}-x_{2}\Vert^{\alpha},
\end{split}
\label{eq:estimA22final}%
\end{equation}
Summing up, by combining \eqref{eq:estimA21final}-\eqref{eq:estimA22final}
with \eqref{eq:estimA2splitted}, we conclude that
\begin{equation}
|\mathrm{A}_{2}|\leq c\,|f|_{C_{x}^{\alpha}(S_{T})}\Vert x_{1}-x_{2}%
\Vert^{\alpha}, \label{eq:estimA2FINAL}%
\end{equation}
where $c>0$ is a suitable constant only depending on $\alpha$. \vspace*{0.1cm}

Now we have estimated $\mathrm{A}_{1}$ and $\mathrm{A}_{2}$, we are finally
ready to complete the proof: in fact, gathering
\eqref{eq:estimA1FINAL}-\eqref{eq:estimA2FINAL}, and recalling
\eqref{eq:spliTijSTART}, we conclude that
\[
|T_{ij}f(x_{1},t)-T_{ij}f(x_{2},t)|\leq|\mathrm{A}_{1}|+|\mathrm{A}_{2}| \leq
c\,|f|_{C^{\alpha}_{x}(S_{T})}\|x_{1}-x_{2}\|^{\alpha},
\]
which is exactly the desired \eqref{eq:toprovex1x2leqone}.
\end{proof}

Thanks to all the results established so far, we can finally give the

\begin{proof}
[Proof of Theorem \ref{Thm Schauder space}]Let $T,\tau,\alpha$ be as in the
statement, and let $u\in\mathcal{S}^{0}(\tau;T)$ be such that $\mathcal{L}u\in
C_{x}^{\alpha}(S_{T})$. By the representation formula
\eqref{repr formula u_xx}, we have
\begin{align*}
\partial_{x_{i}x_{j}}u(x,t)  &  =\int_{\mathbb{R}^{N}\times(\tau,t)}%
\partial_{x_{i}x_{j}}^{2}\Gamma(x,t;y,s)\cdot\big[\mathcal{L}%
u(E(s-t)x,s)-\mathcal{L}u(y,s)\big]\,dy\,ds\\[0.1cm]
&  =T_{ij}(\mathcal{L}u)(x,t)\qquad\text{for every $(x,t)\in S_{T}$ and $1\leq
i,j\leq q$},
\end{align*}
where $T_{ij}$ is as in Theorem \ref{Thm holder singular integrals}. Then,
from \eqref{eq:toproveboundTfSI} we infer that
\begin{equation}
\Vert\partial_{x_{i}x_{j}}u\Vert_{C_{x}^{\alpha}(S_{T})}=\Vert T_{ij}%
(\mathcal{L}u)\Vert_{C_{x}^{\alpha}(S_{T})}\leq c\,|\mathcal{L}u|_{C_{x}%
^{\alpha}(S_{T})}, \label{eq:estimdexixjuprovaScht}%
\end{equation}
where $c>0$ is a constant only depending on $(T-\tau)$ and $\alpha$, and this
is \eqref{eq:Schauderspaceaijt}.

On the other hand, using the definition of $\mathcal{L}$, and recalling that
the coefficients $a_{ij}(\cdot)$ are globally bounded on $\mathbb{R}$ and
\emph{independent of $x$}, from \eqref{eq:estimdexixjuprovaScht} we also get
\begin{equation}%
\begin{split}
\Vert Yu\Vert_{C_{x}^{\alpha}(S_{T})}  &  =\Big\|\mathcal{L}u-\sum_{i,j=1}%
^{q}a_{ij}\,\partial_{x_{i}x_{j}}u\Big\|_{C_{x}^{\alpha}(S_{T})}\\
&  \leq\Vert\mathcal{L}u\Vert_{C_{x}^{\alpha}(S_{T})}+\sum_{i,j=1}^{q}\Vert
a_{ij}\Vert_{L^{\infty}(\mathbb{R})}\cdot\Vert\partial_{x_{i}x_{j}}%
u\Vert_{C_{x}^{\alpha}(S_{T})}\\
&  \leq c\,\Vert\mathcal{L}u\Vert_{C_{x}^{\alpha}(S_{T})}.
\end{split}
\label{eq:estimYuprovaScht}%
\end{equation}
This is \eqref{eq:SchauderspaceaijtDRIFT}, and we are done.
\end{proof}

\subsection{Schauder estimates in space and time\label{section Schauder time}}

Theorem \ref{Thm Schauder space} shows that, for the derivatives
$\partial_{x_{i}x_{j}}^{2}u$ (with $1\leq i,j\leq q$) we can bound the
$C_{x}^{\alpha}$-norm in terms of the quantity $|\mathcal{L}u|_{C_{x}^{\alpha
}(S_{T})}$. We now aim to show how to improve the previous result, giving a
control on the H\"{o}lder norm of $\partial_{x_{i}x_{j}}^{2}u$ with respect to
both space and time, without strengthening the assumptions on $\mathcal{L}u$.

\begin{theorem}
[Local Schauder estimates in space-time]\label{Thm local Schauder time} Let
$T>\tau>-\infty$, $\alpha\in(0,1)$, and let $K\subseteq\mathbb{R}^{N}$ be a
\emph{compact set}.

Then, there exists a constant ${c}={c}(K,\tau,T)>0$ such that, for every
$u\in\mathcal{S}^{0}(\tau;T)$ such that $\mathcal{L}u\in C_{x}^{\alpha}%
(S_{T})$, one has
\begin{equation}%
\begin{split}
&  |\partial_{x_{i}x_{j}}^{2}u(x_{1},t_{1})-\partial_{x_{i}x_{j}}^{2}%
u(x_{2},t_{2})|\\
&  \qquad\leq{c}\,|\mathcal{L}u|_{C_{x}^{\alpha}(S_{T})}\big(d((x_{1}%
,t_{1}),(x_{2},t_{2}))^{\alpha}+|t_{1}-t_{2}|^{\alpha/q_{N}}\big)
\end{split}
\label{eq:schauderspacetimeuxixj}%
\end{equation}
for every $1\leq i,j\leq q$ and every $(x_{1},t_{1}),(x_{2},t_{2})\in
K\times\lbrack\tau,T]$. We recall that $q_{N}\geq3$ is the largest exponent in
the dilations $D_{0}(\lambda)$, see \eqref{dilations}.
\end{theorem}

To prove Theorem \ref{Thm local Schauder time}, we first establish the
following technical lemma.

\begin{lemma}
\label{lem:stimaEx} Let $K\subseteq\mathbb{R}^{N}$ be a fixed compact set, and
let $T>\tau>-\infty$. There exists a constant $c=c(K,\tau,T)>0$ such that
\begin{align}
&  \Vert x-E(t-s)x\Vert\leq c\,|t-s|^{1/q_{N}} &  &  \text{for every $x\in K$
and $t\in\lbrack\tau,T]$}\label{eq:stimaExHolder}\\
&  \Vert(E(t)-E(s))x\Vert\leq c\,|t-s|^{1/q_{N}} &  &  \text{for every $x\in
K$ and $t,s\in\lbrack\tau,T]$}. \label{eq:stimaEtEsHolder}%
\end{align}

\end{lemma}

\begin{proof}
We begin with the proof \eqref{eq:stimaExHolder}. To this end, we fix $x\in K$
and $t\in[\tau,T]$, and we choose $\rho= \rho(K) \geq1$ such that
$K\subseteq\{|z|\leq\rho\}$. Taking into ac\-count the explicit expression of
$\|\cdot\|$ given in \eqref{eq:exprthetazeroij}, we have
\begin{equation}
\label{eq:stimaExPartI}%
\begin{split}
\|x-E(t-s)x\|  &  \leq\sum_{i = 1}^{N}\big|\big(\mathrm{Id}_{N}%
-E(t-s)\big)x\big|^{1/q_{i}}\\
&  \leq\rho\sum_{i = 1}^{N}\|\mathrm{Id}_{N}-E(t-s)\|_{\mathrm{Op}}^{1/q_{i}},
\end{split}
\end{equation}
where $\|\cdot\|_{\mathrm{Op}}$ denotes the \emph{operator norm} of a matrix.
On the other hand, recalling that $E(\sigma) = e^{-\sigma B}$ (and since
$\tau\leq t\leq T$), we also have
\begin{equation}
\label{eq:stimanormE}%
\begin{split}
&  \|\mathrm{Id}_{N}-E(t-s)\|_{\mathrm{Op}} \leq\sum_{k = 1}^{\infty}%
\frac{|t-s|^{k}\,\|B\|_{\mathrm{Op}}^{k}}{k!}\\
&  \qquad\leq|t-s|\sum_{k = 1}^{\infty} \frac{(T-\tau)^{k-1}%
\,\|B\|_{\mathrm{Op}}^{k}}{k!} = c(\tau,T)\cdot|t-s|.
\end{split}
\end{equation}
Gathering \eqref{eq:stimaExPartI}-\eqref{eq:stimanormE}, and recalling that $1
= q_{1}\leq\ldots\leq q_{N}$, we then get
\begin{align*}
\|x-E(t-s)x\| \leq c(\tau,T)\rho\cdot\sum_{i = 1}^{N}|t-s|^{1/q_{i}} \leq
c\,|t-s|^{1/q_{N}},
\end{align*}
where $c > 0$ only depends on $K,\tau,T$. This completes the proof of
\eqref{eq:stimaExHolder}. \vspace*{0.05cm}

We now turn to establish \eqref{eq:stimaEtEsHolder}. To this end, we fix $x\in
K$ and $t,s\in[\tau,T]$. By applying the Mean Value Theorem to the function
$\gamma(\sigma) = E(\sigma)x$ (and taking into account that $E(\sigma) =
e^{-\sigma B}$), we have the estimate
\begin{equation}
\label{eq:EtEsxfirst}%
\begin{split}
|(E(t)-E(s))x|  &  = |\gamma(t)-\gamma(s)| \leq|\gamma^{\prime}(\theta
)|\cdot|t-s|\\
&  = |t-s|\cdot|BE(\theta)x|\\
&  \leq\rho\,|t-s|\cdot\|BE(\theta)\|_{\mathrm{Op}},
\end{split}
\end{equation}
where $\theta$ is a suitable point between $t$ and $s$, and $\rho\geq1$ is as
before. On the other hand, observing that $\tau\leq\theta\leq T$ (as the same
is true of $t,s$), we also get
\begin{equation}
\label{eq:OpnormBEtheta}%
\begin{split}
\|BE(\theta)\|_{\mathrm{Op}}  &  \leq\sum_{k = 0}^{\infty}\frac{|\theta
|^{k}\,\|B\|^{k+1}_{\mathrm{Op}}}{k!}\\
&  \leq\sum_{k = 0}^{\infty}\frac{\max\{|\tau|,|T|\}^{k}\cdot\|B\|^{k+1}%
_{\mathrm{Op}}}{k!} =: c(\tau,T).
\end{split}
\end{equation}
Gathering \eqref{eq:EtEsxfirst}-\eqref{eq:OpnormBEtheta}, we then obtain
\begin{align*}
\|((E(t)-E(s))x\|  &  \leq\sum_{i = 1}^{N} |(E(t)-E(s))x|^{1/q_{i}}\\
&  \leq c(\tau,T)\rho\cdot\sum_{i = 1}^{N}|t-s|^{1/q_{i}} \leq
c\,|t-s|^{1/q_{N}},
\end{align*}
where $c > 0$ only depends on $K,\tau,T$. This completes the proof.
\end{proof}

With Lemma \ref{lem:stimaEx}, we can now prove Theorem
\ref{Thm local Schauder time}.

\begin{proof}
[Proof (of Theorem \ref{Thm local Schauder time})]Let $u\in\mathcal{S}%
^{0}(\tau;T)$ be such that $\mathcal{L}u\in C_{x}^{\alpha}(S_{T})$. First of
all we observe that, owing to Theorem \ref{Thm Schauder space} (and taking
into account the expression of $d$ given in \eqref{eq:explicitd}), there
exists an absolute constant ${c}>0$ such that
\begin{equation}%
\begin{split}
&  |\partial_{x_{i}x_{j}}^{2}u(x_{1},t)-\partial_{x_{i}x_{j}}^{2}%
u(x_{2},t)|\leq|\partial_{x_{i}x_{j}}^{2}u|_{C_{x}^{\alpha}(S_{T})}\Vert
x_{1}-x_{2}\Vert^{\alpha}\\
&  \qquad\leq{c}\,|\mathcal{L}u|_{C_{x}^{\alpha}(S_{T})}\Vert x_{1}-x_{2}%
\Vert^{\alpha},
\end{split}
\label{eq:estimuxixjtfisso}%
\end{equation}
{for every $(x_{1},t),(x_{2},t)\in S_{T}$}. As a consequence of
\eqref{eq:estimuxixjtfisso}, and ta\-king into account Lemma \ref{lem:stimaEx}%
, to prove \eqref{eq:schauderspacetimeuxixj} it suffices to show that
\begin{equation}
|\partial_{x_{i}x_{j}}^{2}u(x,t_{1})-\partial_{x_{i}x_{j}}^{2}u(x,t_{2})|\leq
c\,|\mathcal{L}u|_{C_{x}^{\alpha}(S_{T})}\,\left\{  d((x,t_{1}),(x,t_{2}%
))^{\alpha}+|t_{1}-t_{2}|^{\alpha/q_{N}}\right\}  ,
\label{eq:toproveuxixjxfissoGENERALE}%
\end{equation}
for every $(x,t_{1}),(x,t_{2})\in K\times\lbrack\tau,T]$, where $c>0$ is an
absolute constant independent of $u$ (but possibly depending on the fixed
$K,\tau,T$). In fact, once \eqref{eq:toproveuxixjxfissoGENERALE} has been
established, by combining
\eqref{eq:estimuxixjtfisso}-\eqref{eq:toproveuxixjxfissoGENERALE} with Lemma
\ref{lem:stimaEx} we get
\begin{align*}
&  |\partial_{x_{i}x_{j}}^{2}u(x_{1},t_{1})-\partial_{x_{i}x_{j}}^{2}%
u(x_{2},t_{2})|\\
&  \qquad\leq|\partial_{x_{i}x_{j}}^{2}u(x_{1},t_{1})-\partial_{x_{i}x_{j}%
}^{2}u(x_{2},t_{1})|+|\partial_{x_{i}x_{j}}^{2}u(x_{2},t_{1})-\partial
_{x_{i}x_{j}}^{2}u(x_{2},t_{2})|\\
&  \qquad\leq c\,|\mathcal{L}u|_{C_{x}^{\alpha}(S_{T})}\big(\Vert x_{1}%
-x_{2}\Vert^{\alpha}+d((x_{2},t_{1}),(x_{2},t_{2}))^{\alpha}+|t_{1}%
-t_{2}|^{\alpha/q_{N}}\big)\\[0.1cm]
&  \qquad(\text{by the explicit expression of $d$, see \eqref{eq:explicitd}}%
)\\
&  \qquad\leq c\,|\mathcal{L}u|_{C_{x}^{\alpha}(S_{T})}\big(\Vert x_{1}%
-x_{2}\Vert^{\alpha}+\Vert x_{2}-E(t_{1}-t_{2})x_{2}\Vert^{\alpha} \\
& \qquad\qquad\qquad\qquad\qquad+|t_{1}-t_{2}|^{\alpha/2}+|t_{1}-t_{2}|^{\alpha/q_{N}}\big)\\[0.1cm]
&  \qquad(\text{recalling that, by assumption, $q_{N}\geq3$})\\[0.1cm]
&  \qquad\leq c\,|\mathcal{L}u|_{C_{x}^{\alpha}(S_{T})}\big(\Vert x_{1}%
-x_{2}\Vert^{\alpha}+|t_{1}-t_{2}|^{\alpha/q_{N}}\big)=:(\bigstar);
\end{align*}
from this, using the quasi-triangle inequality \eqref{quasitriangle}, we
obtain
\begin{align*}
(\bigstar)  &  \leq c\,|\mathcal{L}u|_{C_{x}^{\alpha}(S_{T})}\times\\
&  \qquad\times\big(\Vert x_{1}-E(t_{1}-t_{2})x_{2}\Vert^{\alpha}+\Vert
x_{2}-E(t_{1}-t_{2})x_{2}\Vert^{\alpha}+|t_{1}-t_{2}|^{\alpha/q_{N}}\big)\\
&  \leq c\,|\mathcal{L}u|_{C_{x}^{\alpha}(S_{T})}\big(\Vert x_{1}%
-E(t_{1}-t_{2})x_{2}\Vert^{\alpha}+|t_{1}-t_{2}|^{\alpha/q_{N}}\big)\\[0.1cm]
&  (\text{again by the expression of $d$ in \eqref{eq:explicitd}})\\
&  \leq c\,|\mathcal{L}u|_{C_{x}^{\alpha}(S_{T})}\big(d((x_{1},t_{1}%
),(x_{2},t_{2}))^{\alpha}+|t_{1}-t_{2}|^{\alpha/q_{N}}\big),
\end{align*}
which is exactly \eqref{eq:schauderspacetimeuxixj}. Hence, we turn to prove
\eqref{eq:toproveuxixjxfissoGENERALE}. \vspace*{0.05cm}

This can be done adapting several computations exploited in the proof of
Theorem \ref{Thm holder singular integrals}. We will point out just the
relevant differences.

Let us fix two points $(x,t_{1}),(x,t_{2})\in K\times\lbrack\tau,T]$ and
exploit the representation formula \eqref{repr formula u_xx} for
$\partial_{x_{i}x_{j}}^{2}u$: assuming, to fix the ideas, that $t_{2}\geq
t_{1}$, we can write
\begin{equation}%
\begin{split}
&  \partial_{x_{i}x_{j}}^{2}u(x,t_{1})-\partial_{x_{i}x_{j}}^{2}u(x,t_{2})\\
&  \quad=\int_{\mathbb{R}^{N}\times(\tau,t_{1})}\Big\{\partial_{x_{i}x_{j}%
}^{2}\Gamma(x,t_{1};y,s)\big[\mathcal{L}u(E(s-t_{1})x,s)-\mathcal{L}%
u(y,s)\big]\\
&  \qquad\qquad\quad-\partial_{x_{i}x_{j}}^{2}\Gamma(x,t_{2}%
;y,s)\big[\mathcal{L}u(E(s-t_{2})x,s)-\mathcal{L}u(y,s)\big]\Big\}\,dy\,ds\\
&  \qquad-\int_{\mathbb{R}^{N}\times(t_{1},t_{2})}\partial_{x_{i}x_{j}}%
^{2}\Gamma(x,t_{2};y,s)\big[\mathcal{L}u(E(s-t_{2})x,s)-\mathcal{L}%
u(y,s)\big]\,dy\,ds\\
&  \quad=\int_{\{(y,s):\,d((x,t_{2}),(y,s))\geq4\bd{\kappa}\rho\}}%
\{\cdots\}\,dy\,ds\\
&  \qquad\qquad+\int_{\{(y,s):\,d((x,t_{2}),(y,s))<4\bd{\kappa}\rho\}}%
\{\cdots\}\,dy\,ds\\
&  \qquad\qquad-\int_{\mathbb{R}^{N}\times(t_{1},t_{2})}\{\cdots
\}\,dy\,ds\\[0.1cm]
&  \quad=:\mathrm{A}_{1}+\mathrm{A}_{2}-\mathrm{A}_{3},
\end{split}
\label{eq:spliuxixjSTART}%
\end{equation}
where $\bd\kappa>0$ is as in \eqref{quasitriangle}-\eqref{quasisymmetric} and
\[
\rho:=d((x,t_{2}),(x,t_{1})).
\]
We now turn to estimate the integrals $\mathrm{A}_{k}$ (for $k=1,2,3$).
\medskip

\noindent-\thinspace\thinspace\textsc{Estimate of $\mathrm{A}_{1}$.} To begin
with, we write $\mathrm{A}_{1}$ as follows:
\begin{align*}
\mathrm{A}_{1}  &  =\int_{\{(y,s):\,d((x,t_{2}),(y,s))\geq4\bd{\kappa}\rho
\}}\Big\{\big[\mathcal{L}u(E(s-t_{1})x,s)-\mathcal{L}u(y,s)\big]\times
\\[0.1cm]
&  \qquad\qquad\times\big[\partial_{x_{i}x_{j}}^{2}\Gamma(x,t_{1}%
;y,s)-\partial_{x_{i}x_{j}}^{2}\Gamma(x,t_{2}%
;y,s)\big]\Big\}\,dy\,ds\\[0.15cm]
&  \qquad+\int_{\{(y,s):\,d((x,t_{2}),(y,s))\geq4\bd{\kappa}\rho
\}}\Big\{\partial_{x_{i}x_{j}}^{2}\Gamma(x,t_{2};y,s)\times\\
&  \qquad\qquad\times\big[\mathcal{L}u(E(s-t_{1})x,s)-\mathcal{L}%
u(E(s-t_{2})x,s)\big]\Big\}\,dy\,ds\\[0.1cm]
&  =:\mathrm{A}_{11}+\mathrm{A}_{12}.
\end{align*}

\emph{Estimate of $\mathrm{A}_{11}$}. This can be done analogously to what
done in the proof of Theorem \ref{Thm holder singular integrals} for
$\mathrm{A}_{11}$, with $d\left(  \left(  x,t_{1}\right)  ,\left(
x,t_{2}\right)  \right)  $ now replacing $\left\Vert x_{1}-x_{2}\right\Vert $,
getting%
\begin{equation}
\left\vert \mathrm{A}_{11}\right\vert \leq c\left\vert \mathcal{L}u\right\vert
_{C_{x}^{\alpha}\left(  S_{T}\right)  }d\left(  \left(  x,t_{1}\right)
,\left(  x,t_{2}\right)  \right)  ^{\alpha} \label{A11y}%
\end{equation}
where $c>0$ is a constant only depending on $\alpha$. \medskip

\emph{Estimate of $\mathrm{A}_{12}$}. First of all, using once again the fact
that $\mathcal{L}u\in C_{x}^{\alpha}(S_{T})$, jointly with Lemma
\ref{lem:stimaEx}, we can bound the integral $\mathrm{A}_{12}$ as follows:
\[%
\begin{split}
|\mathrm{A}_{12}|  &  \leq\int_{\tau}^{t_{1}}\big|\mathcal{L}u(E(s-t_{1}%
)x,s)-\mathcal{L}u(E(s-t_{2})x,s)\big|\cdot\mathcal{J}(s)\,ds\\
&  \leq|\mathcal{L}u|_{C_{x}^{\alpha}(S_{T})}\int_{\tau}^{t_{1}}%
|(E(s-t_{1})-E(s-t_{2}))x|^{\alpha}\cdot\mathcal{J}(s)\,ds\\
&  (\text{since $|s-t_{1}|,|s-t_{2}|\leq T-\tau$ for all $\tau\leq s\leq
t_{1}$})\\
&  \leq c\,|\mathcal{L}u|_{C_{x}^{\alpha}(S_{T})}\cdot|t_{1}-t_{2}%
|^{\alpha/q_{N}}\int_{\tau}^{t_{1}}\mathcal{J}(s)\,ds=:(\bigstar)
\end{split}
\]
where $c>0$ is an absolute constant and
\[
\mathcal{J}(s):=\bigg|\int_{\{y\in\mathbb{R}^{N}:\,d((x,t_{2}),(y,s))\geq
4\bd{\kappa}\rho\}}\partial_{x_{i}x_{j}}^{2}\Gamma(x,t_{2};y,s)\,dy\bigg|.
\]
From this, using the cancellation property of $\mathcal{J}$ in Theorem
\ref{Thm cancellation property}, we obtain
\begin{equation}
(\bigstar)\leq c\,|\mathcal{L}u|_{C_{x}^{\alpha}(S_{T})}|t_{1}-t_{2}%
|^{\alpha/q_{N}} \label{eq:estimA12Scht}%
\end{equation}
for a suitable constant $c>0$ only depending on $\alpha$. \vspace*{0.05cm}

By combining \eqref{A11y} with \eqref{eq:estimA12Scht}, we conclude that
\begin{equation}
\left\vert \mathrm{A}_{1}\right\vert \leq c\,|\mathcal{L}u|_{C_{x}^{\alpha
}(S_{T})}\left\{  d((x,t_{1}),(x,t_{2}))^{\alpha}+|t_{1}-t_{2}|^{\alpha/q_{N}%
}\right\}  , \label{eq:estimA1FINALScht}%
\end{equation}
for a suitable constant $c>0$ {only depending on $\alpha$}. \medskip

\noindent-\thinspace\thinspace\textsc{Estimate of $A_{2}$.} This can be done
analogously to what done in the proof of Theorem
\ref{Thm holder singular integrals} for $\mathrm{A}_{2}$, with $d\left(
\left(  x,t_{1}\right)  ,\left(  x,t_{2}\right)  \right)  $ now replacing
$\left\Vert x_{1}-x_{2}\right\Vert $, getting
\begin{equation}
\left\vert \mathrm{A}_{2}\right\vert \leq c\,|\mathcal{L}u|_{C_{x}^{\alpha
}(S_{T})}\,d((x,t_{1}),(x,t_{2}))^{\alpha}, \label{eq:estimA2FINALScht}%
\end{equation}
where $c>0$ is a suitable constant only depending on $\alpha$. \medskip

-\thinspace\thinspace\textsc{Estimate of $\mathrm{A}_{3}$.} Using once again
the fact that $\mathcal{L}u\in C_{x}^{\alpha}(S_{T})$, together with the
estimate \eqref{eq:integraldexixjconv} in Proposition \ref{Proposition Claim},
we immediately obtain
\begin{equation}%
\begin{split}
|\mathrm{A}_{3}|  &  \leq|\mathcal{L}u|_{C_{x}^{\alpha}(S_{T})}\int%
_{\mathbb{R}^{N}\times(t_{1},t_{2})}|\partial_{x_{i}x_{j}}^{2}\Gamma
(x,t_{2};y,s)|\cdot\Vert E(s-t_{2})x-y\Vert^{\alpha}\,dy\,ds\\
&  \leq c\,(t_{2}-t_{1})^{\alpha/2}\leq c\,|t_{1}-t_{2}|^{\alpha/q_{N}}%
\end{split}
\label{eq:estimA3Scht}%
\end{equation}
where $c>0$ only depends on $\alpha$. \vspace*{0.1cm}

Now we have estimated $\mathrm{A}_{1},\mathrm{A}_{2}$ and $\mathrm{A}_{3}$, we
can complete the proof: in fact, gathering
\eqref{eq:estimA1FINALScht},\eqref{eq:estimA2FINALScht} and
\eqref{eq:estimA3Scht}, and recalling \eqref{eq:spliuxixjSTART}, we conclude
that
\begin{align*}
&  |\partial_{x_{i}x_{j}}^{2}u(x,t_{1})-\partial_{x_{i}x_{j}}^{2}%
u(x,t_{2})|\leq|\mathrm{A}_{1}|+|\mathrm{A}_{2}|+|\mathrm{A}_{3}|\\
&  \qquad\leq c\,|\mathcal{L}u|_{C_{x}^{\alpha}(S_{T})}\big(d((x,t_{1}%
),(x,t_{2}))^{\alpha}+|t_{1}-t_{2}|^{\alpha/2}\big),
\end{align*}
which is exactly the desired \eqref{eq:toproveuxixjxfissoGENERALE}.
\end{proof}

\section{Schauder estimates for operators with coefficients depending on
$\left(  x,t\right)  $\label{Sec operators a(x,t)}}

Throughout this section we study operators \eqref{L} with coefficients
$a_{ij}\left(  x,t\right)  $ depending on both space and time, fulfilling
assumptions (H1), (H2), (H3) stated in section \ref{sec intro}.

Here we will prove our main result, Theorem \ref{Thm main a priori estimates},
exploiting all the results proved so far.

\subsection{Local Schauder estimates in space\label{sec local schauder space}}

Throughout this section we will consider metric balls $B_{r}\left(
\xi\right)  $ centered at points $\xi\in\mathbb{R}^{N}\times(0,T)$ but
possibly overlapping the hyperplanes $t=0$ and $t=T$ (since these balls will
eventually build a covering of $\mathbb{R}^{N}\times(0,T)$). Our functions
$u\in\mathcal{S}^{\alpha}(0;T)$, so that they are actually defined and jointly
continuous in the whole ball $B_{r}\left(  \xi\right)  \cap S_{T}$; however,
the derivative $Yu$ is merely an $L^{\infty}$ function of the joint variables.
\medskip

\noindent\textbf{Notation.} Throughout this section, we will set
\[
B_{\rho}^{T}({\xi}):=B_{\rho}({\xi})\cap S_{T}\qquad\text{for every $\xi
\in\mathbb{R}^{N+1}$ and $\rho>0$}.
\]

\begin{theorem}
\label{Thm local schauder x} Let $\mathcal{L}$ be an operator of type
\eqref{L} satisfying assumptions \emph{(H1), (H2), (H3)} stated in Section
\ref{sec intro}, for some $\alpha\in(0,1)$.

Then, there exist constants $c,r_{0}>0$ depending on $T$, $\alpha$, the matrix
$B$ in \eqref{B} and the numbers $\nu$ and $\Lambda$ in \eqref{nu} and
\eqref{Lambda}, respectively, such that, for every point $\overline{\xi}\in
S_{T}$, $r\leq r_{0}$ and $u\in\mathcal{S}^{\alpha}(S_{T})$ with
$\mathrm{supp}(u)\subseteq B_{r}(\overline{\xi})\cap\overline{S_{T}}$, one
has
\begin{equation}
\Vert\partial_{x_{k}x_{h}}^{2}u\Vert_{C_{x}^{\alpha}(B_{r}^{T}(\overline{\xi
}))}\leq c\,|\mathcal{L}u|_{C_{x}^{\alpha}(B_{r}^{T}(\overline{\xi}))},
\label{local holder}%
\end{equation}
for every $1\leq h,k\leq q$. We stress that the constant $c$ in
\eqref{local holder} is independent of the ball $B_{r}(\overline{\xi})$.
\end{theorem}

\begin{proof}
Let $r\leq1$ to be chosen later. For a fixed $\overline{\xi}=(\overline
{x},\overline{t})$, we consider the o\-pe\-ra\-tor $\mathcal{L}_{\overline{x}%
}$ with coefficients $a_{ij}(\overline{x},t)$ (frozen in space, variable in
time). Let $\Gamma^{\overline{x}}$ be its fundamental solution, as described
in Theorem \ref{Thm fund sol coeff t dip}. Let $u\in\mathcal{S}^{\alpha}%
(S_{T})$ with $\mathrm{supp}(u)\subseteq B_{r}(\overline{\xi})\cap
\overline{S_{T}}$; then $\mathcal{L}_{\overline{x}}u\in C_{x}^{\alpha}(S_{T})$
and, by Theorem \ref{Thm repr formula derivatives}, we can write
\begin{align*}
&  \partial_{x_{k}x_{h}}^{2}u(x,t)\\
&  \quad=\int_{\overline{t}-1}^{t}\bigg(\int_{\mathbb{R}^{N}}\partial
_{x_{i}x_{j}}^{2}\Gamma^{\overline{x}}(x,t;y,s)\big[\mathcal{L}_{\overline{x}%
}u(E(s-t)x,s)-\mathcal{L}_{\overline{x}}u(y,s)\big]dy\bigg)ds,
\end{align*}
for every $(x,t)\in B_{r}^{T}(\overline{\xi})$ (so that, in particular,
$|t-\overline{t}|\leq r\leq1$). Writing
\[
\mathcal{L}_{\overline{x}}=\mathcal{L}+(\mathcal{L}_{\overline{x}}%
-\mathcal{L}),
\]
we then have%
\begin{align*}
\partial_{x_{k}x_{h}}^{2}u(x,t)  &  =\int_{\overline{t}-1}^{t}\bigg(\int%
_{\mathbb{R}^{N}}\partial_{x_{k}x_{h}}^{2}\Gamma^{\overline{x}}%
(x,t;y,s)\big[\mathcal{L}u(E(s-t)x,s)-\mathcal{L}u(y,s)\big]dy\bigg)ds\\
&  \qquad+\sum_{i,j=1}^{q}\int_{\overline{t}-1}^{t}\int_{\mathbb{R}^{N}%
}\partial_{x_{k}x_{h}}^{2}\Gamma^{\overline{x}}(x,t;y,s)\cdot\\
&  \qquad\qquad\cdot\Big\{\big[a_{ij}(\overline{x},s)-a_{ij}%
(E(s-t)x,s)\big]\partial_{x_{i}x_{j}}^{2}u(E(s-t)x,s)\\
&  \qquad\qquad-\big[a_{ij}(\overline{x},s)-a_{ij}(y,s)\big]\partial
_{x_{i}x_{j}}^{2}u(y,s)\Big\}dyds\\
&  \equiv A+\sum_{i,j=1}^{q}B_{ij}.
\end{align*}
For the term $A$ we have, by Theorem \ref{Thm holder singular integrals},
\begin{equation}
\Vert A\Vert_{C_{x}^{\alpha}(S_{T})}\leq c\,|\mathcal{L}u|_{C_{x}^{\alpha
}(S_{T})}. \label{eq:estimAlocalSch}%
\end{equation}
On the other hand,%
\begin{equation}
B_{ij}=\int_{\overline{t}-1}^{t}\int_{\mathbb{R}^{N}}\partial_{x_{k}x_{h}}%
^{2}\Gamma^{\overline{x}}(x,t;y,s)\big[f_{ij}(E(s-t)x,s)-f_{ij}(y,s)\big]dyds
\label{eq:estimBijlocalSch}%
\end{equation}
with%
\[
f_{ij}(y,s)=[a_{ij}(\overline{x},s)-a_{ij}(y,s)]\partial_{x_{i}x_{j}}%
^{2}u(y,s),
\]
hence, again by Theorem \ref{Thm holder singular integrals},%
\[
\Vert B_{ij}\Vert_{C_{x}^{\alpha}(S_{T})}\leq c\,|f_{ij}|_{C_{x}^{\alpha
}(S_{T})}.
\]
We point out that the constant $c$ in
\eqref{eq:estimAlocalSch}-\eqref{eq:estimBijlocalSch} is independent of the
ball $B_{r}(\overline{\xi})$, since $\mathrm{supp}(u)\subseteq B_{r}%
(\overline{\xi})\subseteq\{(x,t):\,|t-\overline{t}|\leq1\}$, so that we can
apply Theorem \ref{Thm holder singular integrals} with $T-\tau\leq2$.
\vspace*{0.1cm}

We then turn to bound $|f_{ij}|_{C_{x}^{\alpha}(S_{T})}$. We now exploit the
fact that $u$ has small support in space, namely $u(x,t)\neq0$ only if $\Vert
x-\overline{x}\Vert<r$; therefore we can assume that $\Vert x_{k}-\overline
{x}\Vert<r$ for $k=1,2.$ Hence, we have
\begin{align*}
&  f_{ij}(x_{1},s)-f(x_{2},s)\\
&  =[a_{ij}(\overline{x},s)-a_{ij}(x_{1},s)]\,\partial_{x_{i}x_{j}}^{2}%
u(x_{1},s)-[a_{ij}(\overline{x},s)-a_{ij}(x_{2},s)]\,\partial_{x_{i}x_{j}}%
^{2}u(x_{2},s)\\
&  =[a_{ij}(x_{2},s)-a_{ij}(x_{1},s)]\,\partial_{x_{i}x_{j}}^{2}u(x_{1},s)\\
&  \qquad\qquad+[a_{ij}(\overline{x},s)-a_{ij}(x_{2},s)]\,[\partial
_{x_{i}x_{j}}^{2}u(x_{1},s)-\partial_{x_{i}x_{j}}^{2}u(x_{2},s)].
\end{align*}
Then, writing briefly $|\cdot|_{\alpha}$ for $|\cdot|_{C_{x}^{\alpha}(S_{T})}$%
\begin{align*}
&  |f_{ij}(x_{1},s)-f(x_{2},s)|\\
&  \qquad\qquad\leq|a_{ij}|_{\alpha}\Vert x_{2}-x_{1}\Vert^{\alpha}\cdot
\sup|\partial_{x_{i}x_{j}}^{2}u|+|a_{ij}|_{\alpha}r^{\alpha}|\partial
_{x_{i}x_{j}}^{2}u|_{\alpha}\Vert x_{2}-x_{1}\Vert^{\alpha}%
\end{align*}
so that%
\[
|\partial_{x_{k}x_{h}}^{2}u|_{\alpha}+\sup|\partial_{x_{k}x_{h}}^{2}u|\leq
c\,|\mathcal{L}u|_{\alpha}+c\big\{|a_{ij}|_{\alpha}\sup|\partial_{x_{i}x_{j}%
}^{2}u|+|a_{ij}|_{\alpha}r^{\alpha}|\partial_{x_{i}x_{j}}^{2}u|_{\alpha
}\big\}.
\]
Exploiting again the fact that $u$ has compact support, we have
\[
\sup_{B_{r}(\overline{\xi})}|\partial_{x_{i}x_{j}}^{2}u|\leq|\partial
_{x_{i}x_{j}}^{2}u|_{\alpha}(cr)^{\alpha},
\]
so that%
\[
|\partial_{x_{k}x_{h}}^{2}u|_{\alpha}+\sup|\partial_{x_{k}x_{h}}^{2}u|\leq
c\,|\mathcal{L}u|_{\alpha}+c\,|a_{ij}|_{\alpha}r^{\alpha}|\partial_{x_{i}%
x_{j}}^{2}u|_{\alpha},
\]
and for $r$ small enough we get \eqref{local holder}. Note that the small
number $r$ and the constant $c$ are independent of the fixed point
$\overline{x}$. The independence of the constant on $\overline{x}$ also relies
on the uniformity (in $\overline{x}$) of the upper bounds on $\partial
_{x_{i}x_{j}}^{2}\Gamma^{\overline{x}}$. Actually these bounds depend on the
coefficients $a_{ij}(x,t)$ only through the number $\nu$.
\end{proof}

\subsection{Some interpolation inequalities\label{sec interpolation}}

Interpolation inequalities are a typical tool to deduce global estimates
starting with local estimates for compactly supported functions. We will need
the following:

\begin{theorem}
\label{Thm interpolaz}For every $r>0$ there exist $c>0$ and $\gamma>1$ such
that for every $\varepsilon\in\left(  0,1\right)  $, $\overline{\xi}\in S_{T}$
and $u\in\mathcal{S}^{0}\left(  S_{T}\right)  ,$%
\begin{align}
&  \sum_{h=1}^{q}\left\Vert \partial_{x_{h}u}\right\Vert _{C^{\alpha}\left(
B_{r}^{T}\left(  \overline{\xi}\right)  \right)  }+\left\Vert u\right\Vert
_{C^{\alpha}\left(  B_{r}^{T}\left(  \overline{\xi}\right)  \right)
}\nonumber\\
&  \leq\varepsilon\left\{  \sum_{h,k=1}^{q}\left\Vert \partial_{x_{k}x_{h}%
}^{2}u\right\Vert _{C^{0}\left(  B_{4r}^{T}\left(  \overline{\xi}\right)
\right)  }+\left\Vert Yu\right\Vert _{C^{0}\left(  B_{4r}^{T}\left(
\overline{\xi}\right)  \right)  }\right\}  +\frac{c}{\varepsilon^{\gamma}%
}\left\Vert u\right\Vert _{C^{0}\left(  B_{4r}^{T}\left(  \overline{\xi
}\right)  \right)  }. \label{disug interpolaz}%
\end{align}

\end{theorem}

The proof of the above inequality will be reached in several steps. The first
step is based on the analysis of fractional integral operators carried out in
Proposition \ref{Prop fractional generale} and has an independent interest,
since it contains a regularity result for functions in $\mathcal{S}^{0}\left(
S_{T}\right)  $.

\begin{proposition}
\label{Prop interpolaz sharp} \emph{(i)}\thinspace\thinspace Let
$\mathcal{L}_{0}$ be the constant-coefficient operator
\[
\textstyle\mathcal{L}_{0}=\sum_{i=1}^{q}\partial_{x_{i}x_{i}}^{2}+Y
\]
and let $R>0$ be fixed. For every $\alpha\in(0,1)$ there exists $\gamma>2$ and
$c>0$ such that, for every $\overline{\xi}\in S_{T},u\in\mathcal{S}^{0}%
(S_{T})$ with $\mathrm{supp}(u)\subseteq B_{R}(\overline{\xi})\cap
\overline{S_{T}}$ and every $\varepsilon\in(0,1)$ we have:%
\begin{align*}
&  \left\Vert \partial_{x_{k}}u\right\Vert _{C^{\alpha}\left(  B_{R}%
^{T}\left(  \overline{\xi}\right)  \right)  }+\left\Vert u\right\Vert
_{C^{\alpha}\left(  B_{R}^{T}\left(  \overline{\xi}\right)  \right)  }\\
&  \qquad\qquad\leq\varepsilon\left\Vert \mathcal{L}_{0}u\right\Vert
_{C^{0}\left(  B_{R}^{T}\left(  \overline{\xi}\right)  \right)  }+\frac
{c}{\varepsilon^{\gamma}}\left\Vert u\right\Vert _{C^{0}\left(  B_{R}%
^{T}\left(  \overline{\xi}\right)  \right)  }\quad\text{ for }k=1,2,...,q.
\end{align*}
\emph{(}The constant $c$ depends on $r$ and $\alpha$ but not on $\overline
{\xi},u$ and $\varepsilon$\emph{)}. \medskip

\emph{(ii)}\thinspace\thinspace Let $u\in\mathcal{S}^{0}(S_{T})$,
$\overline{\xi}\in S_{T}$ and $R>0$. Then, we have
\[
\text{$u,\,\partial_{x_{k}}u\in C^{\alpha}(B_{R}^{T}(\overline{\xi}))$ for
every $1\leq k\leq q$}.
\]

\end{proposition}

\begin{proof}
Point (ii) will simply follow applying point (i) with $\varepsilon=1$ to the
function $u\phi$, where $\phi\in C_{0}^{\infty}(B_{2R}(\overline{\xi}))$ and
$\phi\equiv1$ on $B_{R}(\overline{\xi})$. So, let us prove (i). This proof is
inspired to \cite[Prop.\,7.1]{BBJDE}. \medskip

Let $\Gamma^{0}$ be the fundamental solution of $\mathcal{L}_{0}$ and let us
write%
\begin{align*}
u\left(  \xi\right)   &  =\int\Gamma^{0}\left(  \xi,\eta\right)
\mathcal{L}_{0}u\left(  \eta\right)  d\eta\\
\partial_{x_{k}}u\left(  \xi\right)   &  =\int\partial_{x_{k}}\Gamma
^{0}\left(  \xi,\eta\right)  \mathcal{L}_{0}u\left(  \eta\right)  d\eta.
\end{align*}
For a fixed $\varepsilon>0$ (that we can assume $<\min\left(  1,R\right)  $)
let $k_{\varepsilon}\left(  \xi,\eta\right)  $ a cutoff function such that%
\[
B_{\varepsilon/2}\left(  \xi\right)  \prec k_{\varepsilon}\left(  \xi
,\cdot\right)  \prec B_{\varepsilon}\left(  \xi\right)  .
\]
We will prove the desired bound for $|\partial_{x_{k}}u|_{C^{\alpha}%
(B_{R}(\overline{\xi}))}.$ A completely analogous proof, starting from the
above representation formula for $u(\xi)$, gives an analogous bound for
$|u|_{C^{\alpha}(B_{R}(\overline{\xi}))}$, possibly with a different exponent
$\gamma$ in the constant $c/\varepsilon^{\gamma}$. Since $\varepsilon\in
(0,1)$, the assertion then follows choosing the bigger exponent.

Let us write%
\begin{equation}%
\begin{split}
\partial_{x_{k}}u\left(  \xi\right)   &  =\int\partial_{x_{k}}\Gamma
^{0}\left(  \xi,\eta\right)  k_{\varepsilon}\left(  \xi,\eta\right)
\mathcal{L}_{0}u\left(  \eta\right)  d\eta\\
&  \qquad\qquad+\int\partial_{x_{k}}\Gamma^{0}\left(  \xi,\eta\right)  \left[
1-k_{\varepsilon}\left(  \xi,\eta\right)  \right]  \mathcal{L}_{0}u\left(
\eta\right)  d\eta\\
&  =\int\partial_{x_{k}}\Gamma^{0}\left(  \xi,\eta\right)  k_{\varepsilon
}\left(  \xi,\eta\right)  \mathcal{L}_{0}u\left(  \eta\right)  d\eta\\
&  \qquad\qquad+\int\left(  \mathcal{L}_{0}^{\ast}\right)  ^{\eta}\left(
\partial_{x_{k}}\Gamma^{0}\left(  \xi,\eta\right)  \left[  1-k_{\varepsilon
}\left(  \xi,\eta\right)  \right]  \right)  u\left(  \eta\right)  d\eta\\
&  =T_{1}\left(  \mathcal{L}_{0}u\right)  +T_{2}\left(  u\right)
\end{split}
\label{uxk interpo}%
\end{equation}
where%
\[
\mathcal{L}_{0}^{\ast}=\sum_{i=1}^{q}\partial_{x_{i}x_{i}}^{2}-Y.
\]
Now we handle $T_{1}$ as a fractional integral. Since the kernel
\[
K_{1}\left(  \xi,\eta\right)  =\partial_{x_{k}}\Gamma^{0}\left(  \xi
,\eta\right)  k_{\varepsilon}\left(  \xi,\eta\right)
\]
does not vanish only if $d\left(  \xi,\eta\right)  <\varepsilon$, owing to
Theorems \ref{Thm bound derivatives}-\ref{Thm mean value} we see that, for
every $\delta\in(0,1)$, the kernel $K_{1}$ satisfies the bounds
\begin{align*}
\left\vert K_{1}\left(  \xi,\eta\right)  \right\vert  &  \leq\frac{c}{d\left(
\xi,\eta\right)  ^{Q+1}}\leq\frac{c\varepsilon^{\delta}}{d\left(  \xi
,\eta\right)  ^{Q+1+\delta}}\\
\left\vert K_{1}\left(  \xi_{1},\eta\right)  -K_{1}\left(  \xi_{2}%
,\eta\right)  \right\vert  &  \leq c\frac{d\left(  \xi_{1},\xi_{2}\right)
}{d\left(  \xi_{1},\eta\right)  ^{Q+2}}\leq c\varepsilon^{\delta}%
\frac{d\left(  \xi_{1},\xi_{2}\right)  }{d\left(  \xi_{1},\eta\right)
^{Q+2+\delta}}\\[0.1cm]
&  \qquad\text{when $d\left(  \xi_{1},\eta\right)  >4\bd{\kappa}d\left(
\xi_{1},\xi_{2}\right)  $}.
\end{align*}
For a fixed $\alpha\in\left(  0,1\right)  $, choosing $\delta<1-\alpha$, by
Proposition \ref{Prop fractional generale} (applied by extending our functions
equal to $0$ out of $S_{T}$) we get%
\begin{equation}
\left\Vert T_{1}\left(  \mathcal{L}_{0}u\right)  \right\Vert _{C^{\alpha
}\left(  B_{R}^{T}\left(  \overline{\xi}\right)  \right)  }\leq c\left(
R\right)  \varepsilon^{\delta}\left\Vert \mathcal{L}_{0}u\right\Vert
_{C^{0}\left(  B_{R}^{T}\left(  \overline{\xi}\right)  \right)  }.
\label{T1 interpo}%
\end{equation}
As to $T_{2}\left(  u\right)  $, let us consider the kernel
\begin{equation}
K_{2}\left(  \xi,\eta\right)  =\mathcal{L}_{0}^{\ast}\left(  \partial_{x_{k}%
}\Gamma^{0}\left(  \xi,\cdot\right)  \left[  1-k_{\varepsilon}\left(
\xi,\cdot\right)  \right]  \right)  \left(  \eta\right)  \text{.}
\label{K2 def}%
\end{equation}
We now claim that the kernel $K_{2}\left(  \xi,\eta\right)  $ satisfies the
following fractional integral estimates:%
\begin{align}
\left\vert K_{2}\left(  \xi,\eta\right)  \right\vert  &  \leq\frac
{c}{\varepsilon^{4}}\frac{1}{d\left(  \xi,\eta\right)  ^{Q-1}}\label{frac1 K2}%
\\
\left\vert K_{2}\left(  \xi_{1},\eta\right)  -K_{2}\left(  \xi_{2}%
,\eta\right)  \right\vert  &  \leq\frac{c}{\varepsilon^{4}}\frac{d\left(
\xi_{1},\xi_{2}\right)  }{d\left(  \xi_{1},\eta\right)  ^{Q}}\text{ for
}d\left(  \xi_{1},\eta\right)  >4\bd{\kappa}d\left(  \xi_{1},\xi_{2}\right)  .
\label{frac2 K2}%
\end{align}
These bounds will be proved in Lemma \ref{Lemma stima frazionario L0}. Taking
these bounds for granted, by Proposition \ref{Prop fractional generale} we get%
\[
\left\Vert T_{2}\left(  u\right)  \right\Vert _{C^{\alpha}\left(  B_{R}%
^{T}\left(  \overline{\xi}\right)  \right)  }\leq\frac{c\left(  R\right)
}{\varepsilon^{4}}\left\Vert u\right\Vert _{C^{0}\left(  B_{R}^{T}\left(
\overline{\xi}\right)  \right)  }%
\]
and then, by \eqref{uxk interpo} and \eqref{T1 interpo}, for some constants
$c_{1},c_{2}$ depending on $R$ but independent of $\varepsilon,\overline{\xi}$
and $u,$%
\[
\left\Vert \partial_{x_{k}}u\right\Vert _{C^{\alpha}\left(  B_{R}^{T}\left(
\overline{\xi}\right)  \right)  }\leq c_{1}\varepsilon^{\delta}\left\Vert
\mathcal{L}_{0}u\right\Vert _{C^{0}\left(  B_{R}^{T}\left(  \overline{\xi
}\right)  \right)  }+\frac{c_{2}}{\varepsilon^{4}}\left\Vert u\right\Vert
_{C^{0}\left(  B_{R}^{T}\left(  \overline{\xi}\right)  \right)  }.
\]
Rescaling $c_{1}\varepsilon^{\delta}=\varepsilon_{1}$ we get%
\[
\left\Vert \partial_{x_{k}}u\right\Vert _{C^{\alpha}\left(  B_{R}^{T}\left(
\overline{\xi}\right)  \right)  }\leq\varepsilon_{1}\left\Vert \mathcal{L}%
_{0}u\right\Vert _{C^{0}\left(  B_{R}^{T}\left(  \overline{\xi}\right)
\right)  }+\frac{c}{\varepsilon_{1}^{4/\delta}}\left\Vert u\right\Vert
_{C^{0}\left(  B_{R}^{T}\left(  \overline{\xi}\right)  \right)  }%
\]
for some $c$ depending on $R$ but not on $\varepsilon_{1}$. So the assertion
is proved, with $\gamma=4/\delta$ and some fixed $\delta\in\left(  0,1\right)
$.

The analogous bound on $\left\Vert u\right\Vert _{C^{\alpha}\left(  B_{R}%
^{T}\left(  \overline{\xi}\right)  \right)  }$ can be proved, with a
completely analogous reasoning, starting with the representation formula
\begin{align*}
u\left(  \xi\right)   &  =\int\Gamma^{0}\left(  \xi,\eta\right)
\mathcal{L}_{0}u\left(  \eta\right)  d\eta\\
&  =\int\Gamma^{0}\left(  \xi,\eta\right)  k_{\varepsilon}\left(  \xi
,\eta\right)  \mathcal{L}_{0}u\left(  \eta\right)  d\eta+\int\left(
\mathcal{L}_{0}^{\ast}\right)  ^{\eta}\left(  \Gamma^{0}\left(  \xi
,\eta\right)  \left[  1-k_{\varepsilon}\left(  \xi,\eta\right)  \right]
\right)  u\left(  \eta\right)  d\eta\\
&  =T_{1}^{\prime}\left(  \mathcal{L}_{0}u\right)  +T_{2}^{\prime}\left(
u\right)
\end{align*}
where $T_{1}^{\prime},T_{2}^{\prime}$ are fractional integral operators with
kernels $K_{1}^{\prime},K_{2}^{\prime},$ respectively, satisfying the
following bounds:%
\begin{align*}
\left\vert K_{1}^{\prime}\left(  \xi,\eta\right)  \right\vert  &  \leq\frac
{c}{d\left(  \xi,\eta\right)  ^{Q}}\leq\frac{c\varepsilon^{\delta}}{d\left(
\xi,\eta\right)  ^{Q+\delta}}\\
\left\vert K_{1}^{\prime}\left(  \xi_{1},\eta\right)  -K_{1}^{\prime}\left(
\xi_{2},\eta\right)  \right\vert  &  \leq c\frac{d\left(  \xi_{1},\xi
_{2}\right)  }{d\left(  \xi_{1},\eta\right)  ^{Q+1}}\leq c\varepsilon^{\delta
}\frac{d\left(  \xi_{1},\xi_{2}\right)  }{d\left(  \xi_{1},\eta\right)
^{Q+1+\delta}}\\[0.1cm]
&  \qquad\text{when $d\left(  \xi_{1},\eta\right)  >4\bd{\kappa}\left(
\xi_{1},\xi_{2}\right)  $}\\
\left\vert K_{2}^{\prime}\left(  \xi,\eta\right)  \right\vert  &  \leq\frac
{c}{\varepsilon^{4}}\frac{1}{d\left(  \xi,\eta\right)  ^{Q-2}}\\
\left\vert K_{2}^{\prime}\left(  \xi_{1},\eta\right)  -K_{2}^{\prime}\left(
\xi_{1},\eta\right)  \right\vert  &  \leq\frac{c}{\varepsilon^{4}}%
\frac{d\left(  \xi_{1},\xi_{2}\right)  }{d\left(  \xi_{1},\eta\right)  ^{Q-1}%
}\\
&  \qquad\text{when $d\left(  \xi_{1},\eta\right)  >4\bd{\kappa}d\left(
\xi_{1},\xi_{2}\right)  $}..
\end{align*}
The bounds on $K_{1}^{\prime}$ are immediate, while those on $K_{2}^{\prime}$
can be proved with the same reasoning used in the proof of Lemma
\ref{Lemma stima frazionario L0} here below, exploiting the corresponding
upper bounds on the derivatives of $\Gamma^{0}$. The upper bound on $\Vert
u\Vert_{C^{\alpha}(B_{R}^{T}(\overline{\xi}))}$ leads to an exponent
$\gamma^{\prime}$ possibly different from the exponent $\gamma$ found in the
bound on $\Vert\partial_{x_{k}}u\Vert_{C^{\alpha}(B_{R}^{T}(\overline{\xi}%
))},$ but since $\varepsilon\in\left(  0,1\right)  $ it is enough to choose
$\max\left(  \gamma,\gamma^{\prime}\right)  $.
\end{proof}

\begin{lemma}
\label{Lemma stima frazionario L0}For every $\varepsilon\in\left(  0,1\right)
$, the kernel $K_{2}$ defined in \eqref{K2 def} satisfies the bounds \eqref{frac1 K2}-\eqref{frac2 K2}.
\end{lemma}

\begin{proof}
Recalling that $\mathcal{L}_{0}^{\ast}\left(  \Gamma^{0}\left(  \xi
,\cdot\right)  \right)  =0$ and the $x$ and $y$ derivatives of $\Gamma_{0}$
commute, we have:%
\begin{equation}%
\begin{split}
K_{2}\left(  \xi,\eta\right)   &  =\left(  \sum_{i=1}^{q}\partial_{y_{i}y_{i}%
}^{2}-Y^{\left(  y,s\right)  }\right)  \left(  \partial_{x_{k}}\Gamma
^{0}\left(  \left(  x,t\right)  ,\left(  y,s\right)  \right)  \left[
1-k_{\varepsilon}\left(  \left(  x,t\right)  ,\left(  y,s\right)  \right)
\right]  \right) \\
&  =\partial_{x_{k}}\Gamma^{0}\left(  \left(  x,t\right)  ,\left(  y,s\right)
\right)  \left(  \sum_{i=1}^{q}\partial_{y_{i}y_{i}}^{2}-Y^{\left(
y,s\right)  }\right)  \left[  1-k_{\varepsilon}\left(  \left(  x,t\right)
,\left(  y,s\right)  \right)  \right] \\
&  +2\sum_{i=1}^{q}\partial_{x_{k}y_{i}}^{2}\Gamma^{0}\left(  \left(
x,t\right)  ,\left(  y,s\right)  \right)  \left[  1-k_{\varepsilon}\left(
\left(  x,t\right)  ,\left(  y,s\right)  \right)  \right]  _{y_{i}}.
\end{split}
\label{eq:exprK2explicit}%
\end{equation}
Exploiting the growth estimates of $\Gamma_{x_{k}}^{0},\Gamma_{x_{k}y_{i}}%
^{0}$ (see Theorem \ref{Thm bound derivatives}) we get%
\[
\left\vert K_{2}\left(  \xi,\eta\right)  \right\vert \leq\frac{c}{d\left(
\xi,\eta\right)  ^{Q+1}}\frac{c}{\varepsilon^{2}}+\frac{c}{d\left(  \xi
,\eta\right)  ^{Q+2}}\frac{c}{\varepsilon}\leq\frac{c}{\varepsilon^{4}}%
\frac{1}{d\left(  \xi,\eta\right)  ^{Q-1}}%
\]
since $K_{2}\left(  \xi,\eta\right)  $ vanishes for $d\left(  \xi,\eta\right)
<\varepsilon/2.$ So we have \eqref{frac1 K2}. In order to prove
\eqref{frac2 K2} we are going to bound $\partial_{x_{h}}K_{2}$ for
$h=1,2,...,q$ and $Y^{\left(  x,t\right)  }K_{2}$ and then apply Lagrange'
theorem with respect to the vector fields. Note that the operator
$\mathcal{L}_{0}$ has \emph{smooth }coefficients, independent of $t$.%
\begin{align}
\left\vert \partial_{x_{h}}K_{2}\left(  \xi,\eta\right)  \right\vert  &
\leq\frac{c}{d\left(  \xi,\eta\right)  ^{Q+2}}\frac{c}{\varepsilon^{2}}%
+\frac{c}{d\left(  \xi,\eta\right)  ^{Q+1}}\frac{c}{\varepsilon^{3}}+\frac
{c}{d\left(  \xi,\eta\right)  ^{Q+3}}\frac{c}{\varepsilon}\nonumber\\
&  \leq\frac{c}{d\left(  \xi,\eta\right)  ^{Q+1}}\frac{c}{\varepsilon^{3}}%
\leq\frac{c}{\varepsilon^{4}}\frac{1}{d\left(  \xi,\eta\right)  ^{Q}}
\label{DxK2}%
\end{align}
since $K_{2}\left(  \xi,\eta\right)  $ vanishes for $d\left(  \xi,\eta\right)
<\varepsilon/2.$ Moreover, using the bounds for $YD_{x}^{\bd\alpha}\Gamma$
established in the proof of Theorem \ref{Thm mean value}, we have
\[
|Y^{(x,t)}\partial_{x_{k}}\Gamma^{0}|\leq\frac{c}{d(\xi,\eta)^{Q+3}}%
\quad\text{and}\quad|Y^{(x,t)}\partial_{x_{k}y_{i}}^{2}\Gamma^{0}|\leq\frac
{c}{d(\xi,\eta)^{Q+4}}.
\]
Therefore, by \eqref{eq:exprK2explicit} we obtain
\begin{align}
\left\vert Y^{\left(  x,t\right)  }K_{2}\left(  \xi,\eta\right)  \right\vert
&  \leq\frac{c}{d\left(  \xi,\eta\right)  ^{Q+3}}\frac{c}{\varepsilon^{2}%
}+\frac{c}{d\left(  \xi,\eta\right)  ^{Q+1}}\frac{c}{\varepsilon^{4}%
}\nonumber\\
&  \qquad+\frac{c}{d\left(  \xi,\eta\right)  ^{Q+4}}\frac{c}{\varepsilon
}+\frac{c}{d\left(  \xi,\eta\right)  ^{Q+2}}\frac{c}{\varepsilon^{3}%
}\nonumber\\
&  \leq\frac{c}{d\left(  \xi,\eta\right)  ^{Q+1}}\frac{c}{\varepsilon^{4}}
\label{YK2}%
\end{align}
where we have used again the vanishing of $K_{2}\left(  \xi,\eta\right)  $ for
$d\left(  \xi,\eta\right)  <\varepsilon/2.$

Hence, by Lagrange' theorem (Theorem \ref{Lagrange}), \eqref{DxK2}-\eqref{YK2}
imply \eqref{frac2 K2}. This completes the proof of Lemma and therefore of
Proposition \ref{Prop interpolaz sharp}.
\end{proof}

The second ingredient of of proof of Theorem \ref{Thm interpolaz} is the
following inequality, which seems a standard Euclidean result.\ The only
difference is that the norms are based on \emph{metric }balls.

\begin{proposition}
\label{Prop interp euclidea} For every $r>\varepsilon>0$ and $u\in
C^{0}(\overline{B_{2r}(\overline{\xi})\cap S_{T}})$ possessing
con\-ti\-nuo\-us derivatives $\partial_{x_{h}}u$ and $\partial_{x_{h}x_{h}%
}^{2}u$ in $\overline{B_{2r}(\overline{\xi})\cap S_{T}}$ for some $1\leq h\leq
q$, we have
\[
\Vert\partial_{x_{h}}u\Vert_{C^{0}(B_{r}^{T}(\overline{\xi}))}\leq
\varepsilon\Vert\partial_{x_{h}x_{h}}^{2}u\Vert_{C^{0}(B_{2r}^{T}%
(\overline{\xi}))}+\frac{2}{\varepsilon}\Vert u\Vert_{C^{0}(B_{2r}%
^{T}(\overline{\xi}))}.
\]

\end{proposition}

\begin{proof}
For a fixed $\xi\in B_{r}^{T}\left(  \overline{\xi}\right)  $, let
\[
f\left(  t\right)  =u\left(  \xi+t\varepsilon e_{h}\right)  \text{ for }%
t\in\left[  0,1\right]  ,
\]
with $e_{h}$ the $h$-th unit vector. Then the identity%
\[
f\left(  1\right)  -f\left(  0\right)  =f^{\prime}\left(  0\right)  +\int%
_{0}^{1}\left(  1-s\right)  f^{\prime\prime}\left(  s\right)  ds
\]
gives%
\[
u\left(  \xi+\varepsilon e_{h}\right)  -u\left(  \xi\right)  =\varepsilon
\partial_{x_{h}}u\left(  \xi\right)  +\varepsilon^{2}\int_{0}^{1}\left(
1-s\right)  \partial_{x_{h}x_{h}}^{2}u\left(  \xi+s\varepsilon e_{h}\right)
ds.
\]
Moreover, for $\xi=\left(  x,t\right)  $ ranging in $B_{r}^{T}\left(
\overline{\xi}\right)  $ and $\varepsilon<r,$ $s\in\left(  0,1\right)  $, we
claim that
\[
\xi+s\varepsilon e_{h}\in B_{2r}^{T}\left(  \overline{\xi}\right)  .
\]
This fact is not trivial because $B_{r}$ are not Euclidean balls but balls
w.r.t. the quasidistance $d.$ Let us compute%
\begin{align*}
&  d\left(  \xi+s\varepsilon e_{h},\overline{\xi}\right)  =\left\Vert
x+s\varepsilon e_{h}-E\left(  t-\overline{t}\right)  \overline{x}\right\Vert
+\sqrt{\left\vert t-\overline{t}\right\vert }\\
&  \qquad=\sum_{i\neq h}\left\vert \left(  x-E\left(  t-\overline{t}\right)
\overline{x}\right)  _{i}\right\vert ^{1/q_{i}}+\left\vert \left(  x-E\left(
t-\overline{t}\right)  \overline{x}\right)  _{h}+s\varepsilon\right\vert
+\sqrt{\left\vert t-\overline{t}\right\vert }\\
&  \qquad\leq\left(  \sum_{i}\left\vert \left(  x-E\left(  t-\overline
{t}\right)  \overline{x}\right)  _{i}\right\vert ^{1/q_{i}}+\sqrt{\left\vert
t-\overline{t}\right\vert }\right)  +\varepsilon<2r,
\end{align*}
where we have exploited the fact that the $h$-th variable (for $h=1,2,...,q$)
has homogeneity $1.$ Hence $\xi+s\varepsilon e_{h}\in B_{2r}(\overline{\xi})$.
Note also that $\xi$ and $\xi+s\varepsilon e_{h}$ have the same $t$-component.
Therefore%
\[
\varepsilon\sup_{B_{r}^{T}}\left\vert \partial_{x_{h}}u\right\vert \leq
2\sup_{B_{2r}^{T}}\left\vert u\right\vert +\varepsilon^{2}\sup_{B_{2r}^{T}%
}\left\vert \partial_{x_{h}x_{h}}^{2}u\right\vert
\]
as desired.
\end{proof}

We can now come to the

\bigskip

\begin{proof}
[Proof of Theorem \ref{Thm interpolaz}]Let us first prove the result under the
additional assumption that $u\in C^{2,\alpha}(\overline{B_{4r}^{T}%
(\overline{\xi})})$. Let $\mathcal{L}_{0}$ be as in Proposition
\ref{Prop interpolaz sharp}, and choose $\phi\in C_{0}^{\infty}\left(
B_{2r}\left(  \overline{\xi}\right)  \right)  $ with $\phi=1$ in $B_{r}\left(
\overline{\xi}\right)  .$ To be more precise, we can fix a \textquotedblleft
mother function\textquotedblright\ $\Phi\in C_{0}^{\infty}\left(  B_{R}\left(
0\right)  \right)  $ with $\Phi=1$ in $B_{R/2}\left(  0\right)  $ and define
$\phi\left(  \xi\right)  =\Phi(\overline{\xi}^{-1}\circ\xi)$ so that, by left
invariance of $\mathcal{L}_{0}$ and $\partial_{x_{h}}$ ($h=1,2,..,q$) the
quantities
\[
\left\Vert \phi\right\Vert _{C^{0}\left(  B_{2r}\left(  \overline{\xi}\right)
\right)  },\left\Vert \partial_{x_{h}}\phi\right\Vert _{C^{0}\left(
B_{2r}\left(  \overline{\xi}\right)  \right)  },\left\Vert \mathcal{L}_{0}%
\phi\right\Vert _{C^{0}\left(  B_{2r}\left(  \overline{\xi}\right)  \right)  }%
\]
do not depend on the center $\overline{\xi}$ of the ball (they will depend on
$r$, which however is fixed).

Applying Proposition \ref{Prop interpolaz sharp} to $u\phi$ we get, for every
$\varepsilon\in\left(  0,1\right)  ,h=1,2,...,q$,
\begin{align*}
\Vert\partial_{x_{h}}u\Vert_{C^{\alpha}(B_{r}^{T}(\overline{\xi}))} &  +\Vert
u\Vert_{C^{\alpha}(B_{r}^{T}(\overline{\xi}))}\leq\Vert\partial_{x_{h}}%
(u\phi)\Vert_{C^{\alpha}(B_{2r}^{T}(\overline{\xi}))}+\Vert u\phi
\Vert_{C^{\alpha}(B_{2r}^{T}(\overline{\xi}))}\\[0.1cm]
&  \qquad\leq\varepsilon\Vert\mathcal{L}_{0}(u\phi)\Vert_{C^{0}(B_{2r}%
^{T}(\overline{\xi}))}+\frac{c}{\varepsilon^{\gamma}}\Vert u\phi\Vert
_{C^{0}(B_{2r}^{T}(\overline{\xi}))}\\
&  \qquad\leq c\varepsilon\bigg\{\sum_{h,k=1}^{q}\Vert\partial_{x_{k}x_{h}%
}^{2}u\Vert_{C^{0}(B_{2r}^{T}(\overline{\xi}))}+\Vert Yu\Vert_{C^{0}%
(B_{2r}^{T}(\overline{\xi}))}\\
&  \qquad\qquad\qquad+\sum_{h=1}^{q}\Vert\partial_{x_{h}}u\Vert_{C^{0}%
(B_{2r}^{T}(\overline{\xi}))}+\Vert u\Vert_{C^{0}(B_{2r}^{T}(\overline{\xi}%
))}\bigg\}\\[0.1cm]
&  \qquad\qquad+\frac{c}{\varepsilon^{\gamma}}\Vert u\Vert_{C^{0}(B_{2r}%
^{T}(\overline{\xi}))}%
\end{align*}
(with $c$ independent of $\overline{\xi}$, by the construction of $\phi$), by
Proposition \ref{Prop interp euclidea} (applied with $\varepsilon=1$), which
can be applied because we are assuming $u\in C^{2,\alpha}(\overline{B_{4r}%
^{T}(\overline{\xi})})$,
\begin{align*}
&  \leq c\varepsilon\left\{  \sum_{h,k=1}^{q}\left\Vert \partial_{x_{k}x_{h}%
}^{2}u\right\Vert _{C^{0}\left(  B_{4r}^{T}\left(  \overline{\xi}\right)
\right)  }+\left\Vert Yu\right\Vert _{C^{0}\left(  B_{2r}^{T}\left(
\overline{\xi}\right)  \right)  }+\left\Vert u\right\Vert _{C^{0}\left(
B_{4r}^{T}\left(  \overline{\xi}\right)  \right)  }\right\}  +\frac
{c}{\varepsilon^{\gamma}}\left\Vert u\right\Vert _{C^{0}\left(  B_{2r}%
^{T}\left(  \overline{\xi}\right)  \right)  }\\
&  \leq c\varepsilon\left\{  \sum_{h,k=1}^{q}\left\Vert \partial_{x_{k}x_{h}%
}^{2}u\right\Vert _{C^{0}\left(  B_{4r}^{T}\left(  \overline{\xi}\right)
\right)  }+\left\Vert Yu\right\Vert _{C^{0}\left(  B_{4r}^{T}\left(
\overline{\xi}\right)  \right)  }\right\}  +\frac{c}{\varepsilon^{\gamma}%
}\left\Vert u\right\Vert _{C^{0}\left(  B_{4r}^{T}\left(  \overline{\xi
}\right)  \right)  }.
\end{align*}
Next, let $u\in\mathcal{S}^{0}\left(  S_{T}\right)  $, extend it to $0$ out of
$S_{T}$ and define its mollified version $u_{\delta}$ as in the proof of
Theorem \ref{Thm repr formula u}. Then $u_{\delta}$ satisfies
\eqref{disug interpolaz}, however $\partial_{x_{k}}$ and $Y$ commute with the
mollification, so that
\begin{align*}
&  \sum_{h=1}^{q}\left\Vert \left(  \partial_{x_{h}}u\right)  _{\delta
}\right\Vert _{C^{\alpha}\left(  B_{r}^{T}\left(  \overline{\xi}\right)
\right)  }+\left\Vert u_{\delta}\right\Vert _{C^{\alpha}\left(  B_{r}%
^{T}\left(  \overline{\xi}\right)  \right)  }\\
&  \leq\varepsilon\left\{  \sum_{h,k=1}^{q}\left\Vert \left(  \partial
_{x_{k}x_{h}}^{2}u\right)  _{\delta}\right\Vert _{C^{0}\left(  B_{4r}%
^{T}\left(  \overline{\xi}\right)  \right)  }+\left\Vert \left(  Yu\right)
_{\delta}\right\Vert _{C^{0}\left(  B_{4r}^{T}\left(  \overline{\xi}\right)
\right)  }\right\}  +\frac{c}{\varepsilon^{\gamma}}\left\Vert u_{\delta
}\right\Vert _{C^{0}\left(  B_{4r}^{T}\left(  \overline{\xi}\right)  \right)
}\\
&  \leq\varepsilon\left\{  \sum_{h,k=1}^{q}\left\Vert \partial_{x_{k}x_{h}%
}^{2}u\right\Vert _{C^{0}\left(  B_{4r}^{T}\left(  \overline{\xi}\right)
\right)  }+\left\Vert Yu\right\Vert _{C^{0}\left(  B_{4r}^{T}\left(
\overline{\xi}\right)  \right)  }\right\}  +\frac{c}{\varepsilon^{\gamma}%
}\left\Vert u\right\Vert _{C^{0}\left(  B_{4r}^{T}\left(  \overline{\xi
}\right)  \right)  }.
\end{align*}
We already know that $u_{\delta}$ uniformly converges to $u$, which is a
priori continuous, on $S_{T-\e_{0}}$ for every $\e_{0}>0$. The uniform bound
on $\Vert(\partial_{x_{h}}u)_{\delta}\Vert_{C^{\alpha}(B_{r}^{T}(\overline
{\xi}))},\Vert u_{\delta}\Vert_{C^{\alpha}(B_{r}^{T}(\overline{\xi}))}$
implies that the functions $u_{\delta},$ $\left(  \partial_{x_{h}}u\right)
_{\delta}$ are equicontinuous and equibounded, then by Ascoli-Arzel\`{a}'s
theorem we can extract a sequence $(\partial_{x_{k}}u)_{\delta}$ uniformly
converging to some function $v_{k}$ which must coincide with $\partial_{x_{k}%
}u$.

The uniform convergence allows to get the bound
\begin{align*}
&  \sum_{h=1}^{q}\left\Vert \partial_{x_{h}}u\right\Vert _{C^{\alpha}%
(B_{r}^{T-\e_{0}}(\overline{\xi}))}+\left\Vert u\right\Vert _{C^{\alpha}%
(B_{r}^{T-\e_{0}}(\overline{\xi}))}\\
&  \leq\varepsilon\left\{  \sum_{h,k=1}^{q}\left\Vert \partial_{x_{k}x_{h}%
}^{2}u\right\Vert _{C^{0}\left(  B_{4r}^{T}\left(  \overline{\xi}\right)
\right)  }+\left\Vert Yu\right\Vert _{C^{0}\left(  B_{4r}^{T}\left(
\overline{\xi}\right)  \right)  }\right\}  +\frac{c}{\varepsilon^{\gamma}%
}\left\Vert u\right\Vert _{C^{0}\left(  B_{4r}^{T}\left(  \overline{\xi
}\right)  \right)  }.
\end{align*}
Since this holds for every $\e_{0}>0$ with a constant $c$ independent of
$\e_{0}$, we obtain \eqref{disug interpolaz}, and we are done.
\end{proof}

\subsection{Global Schauder estimates in
space\label{sec global schauder space}}

Here we want to get global Schauder estimates on the strip $S_{T}$, starting
with the local Schauder estimates proved in Theorem \ref{Thm local schauder x}
for functions which are compactly supported on small balls. To this aim, we
will basically make use of cutoff functions and the interpolation inequalities
proved in the previous section.

We start with a brief discussion about how a control of $C_{x}^{\alpha}(S_{T})
$-norm can be obtained starting with the control of norms $C_{x}^{\alpha
}(B_{r}(\xi_{i}))$ for a suitable family of balls $\{B_{r}(\xi_{i})\} _{i}.$

Let us start defining, for some fixed small $r>0,$ the seminorms%
\begin{align*}
\left\vert f\right\vert _{C_{x,r}^{\alpha}\left(  S_{T}\right)  }  &
\equiv\sup_{\substack{\left(  x_{1},t\right)  ,\left(  x_{2},t\right)  \in
S_{T}\\0<\left\Vert x_{1}-x_{2}\right\Vert \leq r}}\frac{\left\vert f\left(
x_{1},t\right)  -f\left(  x_{2},t\right)  \right\vert }{\left\Vert x_{1}%
-x_{2}\right\Vert ^{\alpha}}\\
\left\vert f\right\vert _{C_{r}^{\alpha}\left(  S_{T}\right)  }  &  \equiv
\sup_{\substack{\xi_{1},\xi_{2}\in S_{T}\\0<d\left(  \xi_{1},\xi_{2}\right)
\leq r}}\frac{\left\vert f\left(  \xi_{1}\right)  -f\left(  \xi_{2}\right)
\right\vert }{d\left(  \xi_{1},\xi_{2}\right)  ^{\alpha}}%
\end{align*}
and let
\[
\left\Vert f\right\Vert _{C_{x,r}^{\alpha}\left(  S_{T}\right)  }=\left\vert
f\right\vert _{C_{x,r}^{\alpha}\left(  S_{T}\right)  }+\left\Vert f\right\Vert
_{C^{0}\left(  S_{T}\right)  }.
\]
Then the following holds:

\begin{proposition}
\label{Prop covering} Let $r>0$ and $\alpha\in\left(  0,1\right)  $ be fixed, then:

\begin{enumerate}
\item[(i)] There exists $c>0$, depending on $\alpha$ and $r$, such that%
\begin{equation}
\left\Vert f\right\Vert _{C_{x}^{\alpha}\left(  S_{T}\right)  }\leq
c\left\Vert f\right\Vert _{C_{x,r}^{\alpha}\left(  S_{T}\right)  }.
\label{C x r - C x}%
\end{equation}

\item[(ii)] Moreover, let $\left\{  B_{r}\left(  \overline{\xi}_{i}\right)
\right\}  _{i=1}^{\infty}$ be a covering of $S_{T}$, then%
\begin{align}
\left\vert f\right\vert _{C_{x,r}^{\alpha}( S_{T}) }  &  \leq\sup
_{i}\left\vert f\right\vert _{C_{x}^{\alpha}( B^{T}_{\theta r}( \overline{\xi
}_{i})) }\label{C x r Sup}\\
\left\vert f\right\vert _{C_{r}^{\alpha}\left(  S_{T}\right)  }  &  \leq
\sup_{i}\left\vert f\right\vert _{C^{\alpha}( B^{T}_{\theta r}( \overline{\xi
}_{i}) ) } \label{C r C}%
\end{align}
where $\theta\geq1$ is an absolute constant.
\end{enumerate}
\end{proposition}

\begin{proof}
(i) Noting that
\[
\sup_{\substack{\left(  x_{1},t\right)  ,\left(  x_{2},t\right)  \in
S_{T}\\\left\Vert x_{1}-x_{2}\right\Vert >r}}\frac{\left\vert f\left(
x_{1},t\right)  -f\left(  x_{2},t\right)  \right\vert }{\left\Vert x_{1}%
-x_{2}\right\Vert ^{\alpha}}\leq\frac{2}{r^{\alpha}}\left\Vert f\right\Vert
_{C^{0}\left(  S_{T}\right)  }%
\]
we immediately derive%
\[
\left\vert f\right\vert _{C_{x}^{\alpha}\left(  S_{T}\right)  }\leq\max\left(
\left\vert f\right\vert _{C_{x,r}^{\alpha}\left(  S_{T}\right)  },\frac
{2}{r^{\alpha}}\left\Vert f\right\Vert _{C^{0}\left(  S_{T}\right)  }\right)
\]
which in turn implies \eqref{C x r - C x}.

(ii) Next, for any two points $\left(  x_{1},t\right)  ,\left(  x_{2}%
,t\right)  \in S_{T}$ such that $\left\Vert x_{1}-x_{2}\right\Vert \leq r$,
let $( x_{1},t) \in B_{r}( \overline{\xi}_{i_{1}}) $ for some $i_{1}$. Then
$\left(  x_{2},t\right)  \in B_{\theta r}( \overline{\xi}_{i_{1}}) $ for some
absolute $\theta\geq1 $, and hence%
\[
\frac{\left\vert f\left(  x_{1},t\right)  -f\left(  x_{2},t\right)
\right\vert }{\left\Vert x_{1}-x_{2}\right\Vert ^{\alpha}}\leq\left\vert
f\right\vert _{C_{x}^{\alpha}( B^{T}_{\theta r}( \overline{\xi}_{i_{1}})) }.
\]
Therefore%
\[
\sup_{\substack{\left(  x_{1},t\right)  ,\left(  x_{2},t\right)  \in
S_{T}\\0<\left\Vert x_{1}-x_{2}\right\Vert \leq r}}\frac{\left\vert f\left(
x_{1},t\right)  -f\left(  x_{2},t\right)  \right\vert }{\left\Vert x_{1}%
-x_{2}\right\Vert ^{\alpha}}\leq\sup_{i}\left\vert f\right\vert _{C_{x}%
^{\alpha}( B^{T}_{\theta r}( \overline{\xi}_{i})) },
\]
which is \eqref{C x r Sup}. Analogously one can prove \eqref{C r C}.
\end{proof}

We are now ready for

\begin{theorem}
[Global Schauder estimates]\label{Thm global Schauder in space} Let
$\mathcal{L}$ be the operator \eqref{L} in $S_{T}$ and assume \emph{(H1),
(H2), (H3)} hold, for some $\alpha\in(0,1)$.

Then, there exists a constant $c>0,$ depending on $T$, $\alpha$, the matrix
$B$ in \eqref{B} and the numbers $\nu$ and $\Lambda$ in \eqref{nu},
\eqref{Lambda}, respectively, such that%
\begin{align*}
&  \sum_{h,k=1}^{q}\left\Vert \partial_{x_{h}x_{k}}^{2}u\right\Vert
_{C_{x}^{\alpha}\left(  S_{T}\right)  }+\left\Vert Yu\right\Vert
_{C_{x}^{\alpha}\left(  S_{T}\right)  }+\sum_{k=1}^{q}\left\Vert
\partial_{x_{k}}u\right\Vert _{C^{\alpha}\left(  S_{T}\right)  }+\left\Vert
u\right\Vert _{C^{\alpha}\left(  S_{T}\right)  }\\
&  \qquad\qquad\leq c\left\{  \left\Vert \mathcal{L}u\right\Vert
_{C_{x}^{\alpha}\left(  S_{T}\right)  }+\left\Vert u\right\Vert _{C^{0}\left(
S_{T}\right)  }\right\}
\end{align*}
for every $u\in\mathcal{S}^{\alpha}\left(  S_{T}\right)  $.
\end{theorem}

\begin{proof}
For a fixed $r>0$, small enough so that the local Schauder estimates of
Theorem \ref{Thm local schauder x} hold on balls of radius $2\theta r$ (with
$\theta\geq1$ as in Proposition \ref{Prop covering}), let $\{B_{r}%
(\overline{\xi}_{i})\}_{i=1}^{\infty}$ be a covering of $S_{T}$.

Let $\Phi\in C_{0}^{\infty}(B_{2\theta r}(0))$ such that $\Phi\equiv1$ in
$B_{\theta r}(0)$, and let $\phi_{i}(\xi)=\Phi(\overline{\xi}_{i}^{-1}\circ
\xi),$ so that $\phi_{i}\in C_{0}^{\infty}\left(  B_{2\theta r}\left(
\overline{\xi}_{i}\right)  \right)  ,\phi_{i}=1$ in $B_{\theta r}\left(
\overline{\xi}_{i}\right)  $. Moreover, by construction of $\phi_{i}$ and left
invariance of $Y$ and $\partial_{x_{k}}$ for $k=1,2,...,q$, the $C^{\alpha}$
norms of $\phi_{i},\partial_{x_{k}}\phi_{i},\mathcal{L}\left(  \phi
_{i}\right)  $ are bounded independently of $i$. Throughout this proof the
constants involved may depend on $r$, which however is by now fixed.

To begin with, applying Theorem \ref{Thm local schauder x} to $u\phi_{i}$ on
$B_{2\theta r}(\overline{\xi}_{i})$ we have%
\begin{equation}
\left\Vert \partial_{x_{k}x_{h}}^{2}u\right\Vert _{C_{x}^{\alpha}\left(
B_{\theta r}^{T}\left(  \overline{\xi}_{i}\right)  \right)  }\leq\left\Vert
\partial_{x_{k}x_{h}}^{2}(u\phi_{i})\right\Vert _{C_{x}^{\alpha}\left(
B_{2\theta r}^{T}\left(  \overline{\xi}_{i}\right)  \right)  }\leq c\left\vert
\mathcal{L}\left(  u\phi_{i}\right)  \right\vert _{C_{x}^{\alpha}\left(
B_{2\theta r}^{T}\left(  \overline{\xi}_{i}\right)  \right)  ,}\label{Scha 1}%
\end{equation}
where the constant $c>0$ is independent of the ball. On the other hand,
\[
\mathcal{L}\left(  u\phi_{i}\right)  =\left(  \mathcal{L}u\right)  \phi
_{i}+u(\mathcal{L}\phi_{i})+2\sum_{h,k=1}^{q}a_{hk}\partial_{x_{h}}%
u\cdot\partial_{x_{k}}\phi_{i}%
\]
hence, for some constant $c$ independent of $\overline{\xi}_{i}$,
\begin{equation}%
\begin{split}
&  \left\vert \mathcal{L}\left(  u\phi_{i}\right)  \right\vert _{C_{x}%
^{\alpha}\left(  B_{2\theta r}^{T}\left(  \overline{\xi}_{i}\right)  \right)
}\\
&  \qquad\leq c\left\{  \left\Vert \mathcal{L}u\right\Vert _{C_{x}^{\alpha
}\left(  B_{2\theta r}^{T}\left(  \overline{\xi}_{i}\right)  \right)  }%
+\sum_{h=1}^{q}\left\Vert \partial_{x_{h}}u\right\Vert _{C_{x}^{\alpha}\left(
B_{2\theta r}^{T}\left(  \overline{\xi}_{i}\right)  \right)  }+\left\Vert
u\right\Vert _{C_{x}^{\alpha}\left(  B_{2\theta r}^{T}\left(  \overline{\xi
}_{i}\right)  \right)  }\right\}
\end{split}
\label{Scha 2}%
\end{equation}
Inserting \eqref{Scha 2} in \eqref{Scha 1} and adding to both sides
\[
\sum_{k=1}^{q}\left\Vert \partial_{x_{k}}u\right\Vert _{C^{\alpha}\left(
B_{\theta r}^{T}\left(  \overline{\xi}_{i}\right)  \right)  }+\left\Vert
u\right\Vert _{C^{\alpha}\left(  B_{\theta r}^{T}\left(  \overline{\xi}%
_{i}\right)  \right)  }%
\]
(note that this quantity is finite by Proposition \ref{Prop interpolaz sharp}
(ii) since $u\in\mathcal{S}^{0}\left(  S_{T}\right)  $) we get:
\begin{align*}
&  \sum_{h,k=1}^{q}\left\Vert \partial_{x_{k}x_{h}}^{2}u\right\Vert
_{C_{x}^{\alpha}\left(  B_{\theta r}^{T}\left(  \overline{\xi}_{i}\right)
\right)  }+\sum_{k=1}^{q}\left\Vert \partial_{x_{k}}u\right\Vert _{C^{\alpha
}\left(  B_{\theta r}^{T}\left(  \overline{\xi}_{i}\right)  \right)
}+\left\Vert u\right\Vert _{C^{\alpha}\left(  B_{\theta r}^{T}\left(
\overline{\xi}_{i}\right)  \right)  }\\
&  \leq c\left\{  \left\Vert \mathcal{L}u\right\Vert _{C_{x}^{\alpha}\left(
B_{2\theta r}^{T}\left(  \overline{\xi}_{i}\right)  \right)  }+\sum_{h=1}%
^{q}\left\Vert \partial_{x_{h}}u\right\Vert _{C^{\alpha}\left(  B_{2\theta
r}^{T}\left(  \overline{\xi}_{i}\right)  \right)  }+\left\Vert u\right\Vert
_{C^{\alpha}\left(  B_{2\theta r}^{T}\left(  \overline{\xi}_{i}\right)
\right)  }\right\}
\end{align*}
by Theorem \ref{Thm interpolaz}, for any $\varepsilon\in\left(  0,1\right)  $
(to be fixed later)
\begin{align*}
&  \leq c\left\{  \left\Vert \mathcal{L}u\right\Vert _{C_{x}^{\alpha}\left(
B_{2\theta r}^{T}\left(  \overline{\xi}_{i}\right)  \right)  }+\varepsilon
\left[  \sum_{h,k=1}^{q}\left\Vert \partial_{x_{k}x_{h}}^{2}u\right\Vert
_{C^{0}\left(  B_{8\theta r}^{T}\left(  \overline{\xi}_{i}\right)  \right)
}+\left\Vert Yu\right\Vert _{C^{0}\left(  B_{8\theta r}^{T}\left(
\overline{\xi}_{i}\right)  \right)  }\right]  \right.  \\
&  \left.  +\frac{1}{\varepsilon^{\gamma}}\left\Vert u\right\Vert
_{C^{0}\left(  B_{8\theta r}^{T}\left(  \overline{\xi}_{i}\right)  \right)
}\frac{{}}{{}}\right\}
\end{align*}
from the equation $Yu=\LL u-\sum_{h,k=1}^{q}a_{hk}\partial_{x_{h}x_{k}}^{2}u$
\begin{align*}
&  \leq c\,\bigg\{\left\Vert \mathcal{L}u\right\Vert _{C_{x}^{\alpha}\left(
B_{2\theta r}^{T}\left(  \overline{\xi}_{i}\right)  \right)  } \\
& \qquad+\varepsilon
\bigg[\left(  1+c\left(  \nu\right)  \right)  \sum_{h,k=1}^{q}\left\Vert
\partial_{x_{k}x_{h}}^{2}u\right\Vert _{C^{0}\left(  B_{8\theta r}^{T}\left(
\overline{\xi}_{i}\right)  \right)  }+\left\Vert \mathcal{L}u\right\Vert
_{C^{0}\left(  B_{8\theta r}^{T}\left(  \overline{\xi}_{i}\right)  \right)
}\bigg]\\
&  \qquad\qquad+\frac{1}{\varepsilon^{\gamma}}\left\Vert u\right\Vert
_{C^{0}\left(  B_{8\theta r}^{T}\left(  \overline{\xi}_{i}\right)  \right)
}\frac{{}}{{}}\bigg\}\\
&  \leq c\bigg\{\left\Vert \mathcal{L}u\right\Vert _{C_{x}^{\alpha}\left(
S_{T}\right)  }+c_{1}\varepsilon\sum_{h,k=1}^{q}\left\Vert \partial
_{x_{k}x_{h}}^{2}u\right\Vert _{C^{0}\left(  S_{T}\right)  }+\frac
{1}{\varepsilon^{\gamma}}\left\Vert u\right\Vert _{C^{0}\left(  S_{T}\right)
}\bigg\}
\end{align*}
We now fix $\varepsilon>0$ small enough so that $cc_{1}\varepsilon\leq1/2$, so
that for every ball $B_{r}\left(  \overline{\xi}_{i}\right)  $ of the fixed
covering we have%
\begin{align*}
&  \sum_{h,k=1}^{q}\left\Vert \partial_{x_{k}x_{h}}^{2}u\right\Vert
_{C_{x}^{\alpha}\left(  B_{\theta r}^{T}\left(  \overline{\xi}_{i}\right)
\right)  }+\sum_{k=1}^{q}\left\Vert \partial_{x_{k}}u\right\Vert _{C^{\alpha
}\left(  B_{\theta r}^{T}\left(  \overline{\xi}_{i}\right)  \right)
}+\left\Vert u\right\Vert _{C^{\alpha}\left(  B_{\theta r}^{T}\left(
\overline{\xi}_{i}\right)  \right)  }\\
&  \leq c\left\{  \left\Vert \mathcal{L}u\right\Vert _{C_{x}^{\alpha}\left(
S_{T}\right)  }+\left\Vert u\right\Vert _{C^{0}\left(  S_{T}\right)
}\right\}  +\frac{1}{2}\sum_{h,k=1}^{q}\left\Vert \partial_{x_{k}x_{h}}%
^{2}u\right\Vert _{C^{0}\left(  S_{T}\right)  }.
\end{align*}
Finally, taking the supremum for $i=1,2,3...$ we get, by
\eqref{C x r Sup}-\eqref{C x r Sup}
\begin{align*}
&  \sum_{h,k=1}^{q}\left\Vert \partial_{x_{k}x_{h}}^{2}u\right\Vert
_{C_{x,r}^{\alpha}\left(  S_{T}\right)  }+\sum_{k=1}^{q}\left\Vert
\partial_{x_{k}}u\right\Vert _{C_{r}^{\alpha}\left(  S_{T}\right)
}+\left\Vert u\right\Vert _{C_{r}^{\alpha}\left(  S_{T}\right)  }\\
&  \leq c\left\{  \left\Vert \mathcal{L}u\right\Vert _{C_{x}^{\alpha}\left(
S_{T}\right)  }+\left\Vert u\right\Vert _{C^{0}\left(  S_{T}\right)
}\right\}  +\frac{1}{2}\sum_{h,k=1}^{q}\left\Vert \partial_{x_{k}x_{h}}%
^{2}u\right\Vert _{C^{0}\left(  S_{T}\right)  }%
\end{align*}
so that%
\begin{align*}
&  \sum_{h,k=1}^{q}\left\Vert \partial_{x_{k}x_{h}}^{2}u\right\Vert
_{C_{x,r}^{\alpha}\left(  S_{T}\right)  }+\sum_{k=1}^{q}\left\Vert
\partial_{x_{k}}u\right\Vert _{C_{r}^{\alpha}\left(  S_{T}\right)
}+\left\Vert u\right\Vert _{C_{r}^{\alpha}\left(  S_{T}\right)  }\\
&  \qquad\qquad\leq c\left\{  \left\Vert \mathcal{L}u\right\Vert
_{C_{x}^{\alpha}\left(  S_{T}\right)  }+\left\Vert u\right\Vert _{C^{0}\left(
S_{T}\right)  }\right\}
\end{align*}
and by \eqref{C x r - C x} we conclude%
\begin{align}
&  \left\Vert \partial_{x_{k}x_{h}}^{2}u\right\Vert _{C_{x}^{\alpha}\left(
S_{T}\right)  }+\sum_{k=1}^{q}\left\Vert \partial_{x_{k}}u\right\Vert
_{C^{\alpha}\left(  S_{T}\right)  }+\left\Vert u\right\Vert _{C^{\alpha
}\left(  S_{T}\right)  }\nonumber\\
&  \qquad\leq c\left\{  \left\Vert \mathcal{L}u\right\Vert _{C_{x}^{\alpha
}\left(  S_{T}\right)  }+\left\Vert u\right\Vert _{C^{0}\left(  S_{T}\right)
}\right\}  .\label{Schauder 4}%
\end{align}
Finally, from the equation $Yu=\mathcal{L}u-\sum_{h,k=1}^{q}a_{hk}%
\partial_{x_{h}x_{k}}^{2}u$ we also get%
\[
\left\Vert Yu\right\Vert _{C_{x}^{\alpha}\left(  S_{T}\right)  }\leq c\left\{
\left\Vert \mathcal{L}u\right\Vert _{C_{x}^{\alpha}\left(  S_{T}\right)
}+\sum_{h,k=1}^{q}\left\Vert \partial_{x_{h}x_{k}}^{2}u\right\Vert
_{C_{x}^{\alpha}\left(  S_{T}\right)  }\right\}
\]
with $c$ also depending on the H\"{o}lder norms of the coefficients $a_{ij}$,
and by \eqref{Schauder 4}%
\[
\leq c\left\{  \left\Vert \mathcal{L}u\right\Vert _{C_{x}^{\alpha}\left(
S_{T}\right)  }+\left\Vert u\right\Vert _{C^{0}\left(  S_{T}\right)
}\right\}  .
\]
So we are done.
\end{proof}

\subsection{Schauder estimates in space and time\label{sec schauder time}}

For an arbitrary set $\Omega\subseteq\overline{S_{T}}$, let us define the
seminorms:%
\[
\left\vert f\right\vert _{C_{t}^{\alpha}\left(  \Omega\right)  }%
=\sup_{\substack{\left(  x_{1},t_{1}\right)  ,\left(  x_{2},t_{2}\right)
\in\Omega\\\left(  x_{1},t_{1}\right)  \neq\left(  x_{2},t_{2}\right)  }%
}\frac{\left\vert f\left(  x_{1},t_{1}\right)  -f\left(  x_{2},t_{2}\right)
\right\vert }{d\left(  \left(  x_{1},t_{1}\right)  ,\left(  x_{2}%
,t_{2}\right)  \right)  ^{\alpha}+\left\vert t_{1}-t_{2}\right\vert
^{\alpha/q_{N}}}%
\]%
\[
\left\vert f\right\vert _{C_{t,r}^{\alpha}\left(  \Omega\right)  }\equiv
\sup_{\substack{\left(  x_{1},t_{1}\right)  ,\left(  x_{2},t_{2}\right)
\in\Omega\\0<d\left(  \left(  x_{1},t_{1}\right)  ,\left(  x_{2},t_{2}\right)
\right)  \leq r}}\frac{\left\vert f\left(  x_{1},t_{1}\right)  -f\left(
x_{2},t_{2}\right)  \right\vert }{d\left(  \left(  x_{1},t_{1}\right)
,\left(  x_{2},t_{2}\right)  \right)  ^{\alpha}+\left\vert t_{1}%
-t_{2}\right\vert ^{\alpha/q_{N}}}.
\]
Here the number $q_{N}$ is the largest homogeneity exponent in the dilations,
see \eqref{dilations}. Let also:%
\begin{align*}
\left\Vert f\right\Vert _{C_{t}^{\alpha}\left(  \Omega\right)  }  &
=\left\vert f\right\vert _{C_{t}^{\alpha}\left(  \Omega\right)  }+\left\Vert
f\right\Vert _{C^{0}\left(  \Omega\right)  }\\
\left\Vert f\right\Vert _{C_{t,r}^{\alpha}\left(  \Omega\right)  }  &
=\left\vert f\right\vert _{C_{t,r}^{\alpha}\left(  \Omega\right)  }+\left\Vert
f\right\Vert _{C^{0}\left(  \Omega\right)  }.
\end{align*}
Then the following holds, with a proof perfectly analogoys to that of
Proposition \ref{Prop covering}:

\begin{proposition}
Let $r>0$ and $\alpha\in\left(  0,1\right)  $ be fixed, then:

\begin{enumerate}
\item[(i)] There exists $c>0$, depending on $\alpha$ and $r$, such that%
\begin{equation}
\left\Vert f\right\Vert _{C_{t}^{\alpha}\left(  \Omega\right)  }\leq
c\left\Vert f\right\Vert _{C_{t,r}^{\alpha}\left(  \Omega\right)  }.
\label{covering 1b}%
\end{equation}

\item[(ii)] Moreover, let $\{ B_{r}\left(  \overline{\xi}_{i}\right)  \}
_{i=1}^{\infty}$ be a covering of $\Omega$, then%
\begin{equation}
\left\vert f\right\vert _{C_{t,r}^{\alpha}( \Omega) }\leq\sup_{i}\left\vert
f\right\vert _{C_{t}^{\alpha}(B^{T}_{\theta r}(\overline{\xi}_{i}))}
\label{covering 2b}%
\end{equation}
where $\theta\geq1$ is an absolute constant.
\end{enumerate}
\end{proposition}

We can now state our H\"{o}lder estimate in space and time:

\begin{theorem}
\label{Thm global Schauder in space time 2}Let $\LL$ be the operator \eqref{L}
in $S_{T}$ and assume \emph{(H1), (H2), (H3)} hold, for some $\alpha\in\left(
0,1\right)  $. For every $T>\tau>-\infty$ and every compact set $K\subset
\mathbb{R}^{N}$ there exists $c>0$ depending on $K,\tau,T,\alpha,B,\nu
,\Lambda$ such that, for every $u\in\mathcal{S}^{\alpha}\left(  S_{T}\right)
$ the derivatives $\partial_{x_{h}x_{k}}^{2}u$ satisfy the following local
H\"{o}lder continuity in space-time:%
\[
\left\vert \partial_{x_{i}x_{j}}^{2}u\right\vert _{C_{t}^{\alpha}\left(
K\times\left[  \tau,T\right]  \right)  }+\Vert\partial_{x_{i}x_{j}}^{2}%
u\Vert_{C^{0}(S_{T})}\leq c\left\{  \left\Vert \mathcal{L}u\right\Vert
_{C_{x}^{\alpha}\left(  S_{T}\right)  }+\left\Vert u\right\Vert _{C^{0}\left(
S_{T}\right)  }\right\}  .
\]
In particular, even the second derivatives $\partial_{x_{i}x_{j}}^{2}u$ are
jointly continuous in $S_{T}$.
\end{theorem}

\begin{proof}
Fix a compact set $K\subset\mathbb{R}^{N}$ , let $T>\tau>-\infty$ and let
$\psi\left(  t\right)  $ be a smooth function such that $\psi\left(  t\right)
=1$ for $t\geq\tau$, $\psi\left(  t\right)  =0$ for $t\leq\tau-1$, $0\leq
\psi\left(  t\right)  \leq1.$

For $\overline{\xi}=\left(  \overline{x},\overline{t}\right)  ,$ let us
consider the frozen operator $\mathcal{L}_{\overline{x}}$ with coefficients
$a_{ij}\left(  \overline{x},t\right)  $. Applying Theorem
\ref{Thm local Schauder time} to the operator $\mathcal{L}_{\overline{x}}$ we
get the existence of a constant, depending on $K,\tau,T,\alpha,B,\nu$ but not
on $\overline{\xi}$, such that for every $u\in\mathcal{S}^{\alpha}(S_{T})$,
since $u\psi\in\mathcal{S}^{0}(\tau-1,T)$,
\begin{align*}
&  \left\vert \partial_{x_{i}x_{j}}^{2}u\left(  x_{1},t_{1}\right)
-\partial_{x_{i}x_{j}}^{2}u\left(  x_{2},t_{2}\right)  \right\vert \\
&  \qquad\leq c\left\vert \mathcal{L}_{\overline{x}}\left(  u\psi\right)
\right\vert _{C_{x}^{\alpha}\left(  B_{r}^{T}\left(  \overline{\xi}%
_{i}\right)  \right)  }\left\{  d\left(  \left(  x_{1},t_{1}\right)  ,\left(
x_{2},t_{2}\right)  \right)  ^{\alpha}+\left\vert t_{1}-t_{2}\right\vert
^{\alpha/q_{N}}\right\}
\end{align*}
for $\left(  x_{1},t_{1}\right)  ,\left(  x_{2},t_{2}\right)  \in
K\times\left[  \tau,T\right]  $. However, since $\mathcal{L}_{\overline{x}%
}\left(  u\psi\right)  =\psi\mathcal{L}_{\overline{x}}u-\psi_{t}u$, we have%
\begin{align*}
& \left\vert \mathcal{L}_{\overline{x}}\left(  u\psi\right)  \right\vert
_{C_{x}^{\alpha}\left(  B_{r}^{T}\left(  \overline{\xi}_{i}\right)  \right)
}   \leq\left\vert \psi\mathcal{L}_{\overline{x}}u\right\vert _{C_{x}%
^{\alpha}\left(  S_{T}\right)  }+\left\vert \psi_{t}u\right\vert
_{C_{x}^{\alpha}\left(  S_{T}\right)  }\\
&  \qquad \leq\left\vert \mathcal{L}_{\overline{x}}u\right\vert _{C_{x}^{\alpha
}\left(  S_{T}\right)  }+c\left\vert u\right\vert _{C_{x}^{\alpha}\left(
S_{T}\right)  }\\
& \qquad \leq\left\vert \mathcal{L}u\right\vert _{C_{x}^{\alpha}\left(
S_{T}\right)  }+c\left\vert u\right\vert _{C_{x}^{\alpha}\left(  S_{T}\right)
}+\sum_{i,j=1}^{q}\left\vert \left[  a_{ij}\left(  \overline{x},t\right)
-a_{ij}\left(  \cdot,t\right)  \right]  \partial_{x_{i}x_{j}}^{2}u\right\vert
_{C_{x}^{\alpha}\left(  S_{T}\right)  }.
\end{align*}
On the other hand, since%
\begin{align*}
& \left\vert \left[  a_{ij}\left(  \overline{x},t\right)  -a_{ij}\left(
\cdot,t\right)  \right]  \partial_{x_{i}x_{j}}^{2}u\right\vert _{C_{x}%
^{\alpha}\left(  S_{T}\right)  }  \\
& \qquad \leq2\Lambda\left\vert \partial
_{x_{i}x_{j}}^{2}u\right\vert _{C_{x}^{\alpha}\left(  S_{T}\right)
}+\left\vert a_{ij}\left(  \cdot,t\right)  \right\vert _{C_{x}^{\alpha}\left(
S_{T}\right)  }\left\Vert \partial_{x_{i}x_{j}}^{2}u\right\Vert _{L^{\infty
}\left(  S_{T}\right)  }\\
&  \qquad \leq2\Lambda\left\Vert \partial_{x_{i}x_{j}}^{2}u\right\Vert _{C_{x}%
^{\alpha}\left(  S_{T}\right)  },
\end{align*}
by Theorem \ref{Thm global Schauder in space} we conclude%
\begin{align*}
&  \left\vert \partial_{x_{i}x_{j}}^{2}u\left(  x_{1},t_{1}\right)
-\partial_{x_{i}x_{j}}^{2}u\left(  x_{2},t_{2}\right)  \right\vert \\
&  \quad\leq c\left\{  \left\Vert \mathcal{L}u\right\Vert _{C_{x}^{\alpha}\left(
S_{T}\right)  }+\left\Vert u\right\Vert _{C^{0}\left(  S_{T}\right)
}\right\}  \left\{  d\left(  \left(  x_{1},t_{1}\right)  ,\left(  x_{2}%
,t_{2}\right)  \right)  ^{\alpha}+\left\vert t_{1}-t_{2}\right\vert
^{\alpha/q_{N}}\right\}  .
\end{align*}
So we are done.
\end{proof}

\end{document}